\documentclass[11pt]{likejdg-p}
\usepackage{amsmath, amsthm, amssymb}
\usepackage{graphicx}
\usepackage{mathbbol}
\usepackage{txfonts}
\usepackage[usenames]{color}

\newtheorem{thm}{Theorem}[section]
\newcommand{\bt}{\begin{thm}}
\newcommand{\et}{\end{thm}}

\newtheorem{ex}[thm]{Example}

\newtheorem{cor}[thm]{Corollary}   

\newcommand{\bc}{\begin{cor}}
\newcommand{\ec}{\end{cor}}

\newtheorem{lem}[thm]{Lemma}   

\newcommand{\bl}{\begin{lem}}
\newcommand{\el}{\end{lem}}

\newtheorem{prop}[thm]{Proposition}
\newcommand{\bp}{\begin{prop}}
\newcommand{\ep}{\end{prop}}

\newtheorem{defn}[thm]{Definition}
\newcommand{\bd}{\begin{defn}}    
\newcommand{\ed}{\end{defn}}

\newtheorem{rmrk}[thm]{Remark}   

\newtheorem{conjecture}[thm]{Conjecture}   

\newcommand{\weaklyto}{\rightharpoonup}

\newtheorem{example}[thm]{Example}

\newcommand{\lemref}[1]{Lemma~\ref{#1}}

\newcommand{\GHto}{\stackrel { \textrm{GH}}{\longrightarrow} }
\newcommand{\Fto}{\stackrel {\mathcal{F}}{\longrightarrow} }

\newcommand{\sgn}{\operatorname{sgn}}

\newcommand{\be}{\begin{equation}}

 \newcommand{\ee}{\end{equation}}

\newcommand{\injrad}{\textrm{injrad}}

\newcommand{\N}{\mathbb{N}}

\newcommand{\R}{\mathbb{R}}
\newcommand{\One}{{\bf \rm{1}}}

\newcommand{\Z}{\mathbb{Z}}

\newcommand{\diam}{\operatorname{diam}}

\newcommand{\Fm}{{\mathcal F}}
\newcommand{\hm}{{\mathcal H}}

\newcommand{\set}{\rm{set}}

\newcommand{\disjointunion}{\sqcup}

\newcommand{\Lip}{\operatorname{Lip}}
\newcommand{\Tan}{\operatorname{Tan}}
\newcommand{\mass}{{\mathbf M}}

\newcommand{\CS}[1]{{\textcolor{black}{#1}}}

\newcommand{\curr}{{\mathbf M}}         
\newcommand{\intrectcurr}{{\mathcal I}} 
\newcommand{\intcurr}{{\mathbf I}}      

\newcommand{\dil}{\textrm{dil}}

\newcommand{\vol}{\operatorname{Vol}}
\newcommand{\nmass}{\mathbf N}
\newcommand{\flatnorm}{\mathcal F}
\newcommand{\fillvol}{{\operatorname{Fillvol}}}

\newcommand{\rstr}{\:\mbox{\rule{0.1ex}{1.2ex}\rule{1.1ex}{0.1ex}}\:}
\newcommand{\bdry}{\partial}

\newcommand{\spt}{\operatorname{spt}}

\begin{document}
\title[The Intrinsic Flat Distance]{The Intrinsic Flat Distance\\ 
between Riemannian Manifolds and\\
other Integral Current Spaces}

\author{C. Sormani}
\thanks{To appear in the Journal of Differential Geometry}
\thanks{C. Sormani partially supported by a PSC CUNY Research Grant.}
\address{CUNY Graduate Center
and Lehman College}
\email{sormanic@member.ams.org}

\author{S. Wenger}
\thanks{S. Wenger partially supported by  {\em \bf NSF DMS \#0956374}.}
\address{University of Illinois at Chicago}
\email{wenger@math.uic.edu}

\date{}




\begin{abstract}  
Inspired by the Gromov-Hausdorff distance, we define a new notion
called the intrinsic flat distance  between oriented $m$ dimensional 
Riemannian manifolds with boundary
by isometrically embedding the manifolds into a common metric space,
measuring the flat distance between them and taking an infimum over
all isometric embeddings and all common metric spaces.  This is
made rigorous by applying Ambrosio-Kirchheim's extension of Federer-Fleming's
notion of integral currents to arbitrary metric spaces.
 
We prove the intrinsic flat distance between two 
compact oriented Riemannian manifolds is zero iff 
they have an orientation preserving isometry between them. Using the theory of 
Ambrosio-Kirchheim, we study converging sequences of manifolds and their 
limits, which are in a class of metric spaces that we call integral current spaces. 
We describe the properties of such spaces including the fact that they are 
countably $\mathcal{H}^m$ rectifiable spaces and present numerous examples.
\end{abstract}

\maketitle

\newpage

\noindent
{\Large{\bf{Contents:}}}

\vspace{.3cm}

\noindent
{\bf Section 1: Introduction}

\indent
Section 1.1: A Brief History

\indent
Section 1.2: An Overview

\indent
Section 1.3: Recommended Reading

\indent
Section 1.4:  Acknowledgements

\noindent
{\bf Section 2: Defining Current Spaces }

\indent
Subsection 2.1: Weighted Oriented Countably $\mathcal{H}^m$ Rectifiable Metric Spaces

\indent
Subsection 2.2: Reviewing Ambroiso-Kirchheim's Currents on Metric Spaces

\indent
Subsection 2.3: Parametrized Integer Rectifiable Currents

\indent
Subsection 2.4: Current Structures on Metric Spaces

\indent
Subsection 2.5: Integral Current Spaces

\noindent
{\bf Section 3: The Intrinsic Flat Distance between Integral Current Spaces}

\indent
Subsection 3.1: The Triangle Inequality

\indent
Subsection 3.2: A Brief Review of Existing Compactness Theorems

\indent
Subsection 3.3: The Infimum is Attained  

\indent
Subsection 3.4: Current Preserving Isometries

\noindent
{\bf Section 4: Sequences of Integral Current Spaces}

\indent
Subsection 4.1: Embedding into a Common Metric Space

\indent
Subsection 4.2: Properties of Intrinsic Flat Convergence

\indent
Subsection 4.3: Cancellation under Intrinsic Flat Convergence

\indent
Subsection 4.4: Ricci and Scalar Curvature

\indent
Subsection 4.5: Wenger's Compactness Theorem

\noindent
{\bf Section 5: Lipschitz Maps and Convergence}

\indent
Subsection 5.1: Lipschitz Maps

\indent
Subsection 5.2:  Lipschitz and Smooth Convergence

\noindent
{\bf Section 6: Examples Appendix}

\indent
Subsection 6.1: Isometric Embeddings

\indent
Subsection 6.2: Disappearing Tips and Ilmanen's Example

\indent
Subsection 6.3: Limits with Point Singularities

\indent
Subsection 6.4: Limits Need Not be Precompact

\indent
Subsection 6.5: Pipe-filling and Disconnected Limits

\indent
Subsection 6.6: Collapse in the Limit

\indent
Subsection 6.7: Cancellation in the Limit

\indent
Subsection 6.8: Doubling in the Limit

\indent
Subsection 6.9: Taxi Cab Limit Space

\indent
Subsection 6.10: Limit with a Higher Dimensional Completion

\indent
Subsection 6.11: Gabriel's Horn and the Cauchy Sequence with No Limit

\newpage
\section{Introduction}

\vspace{.3cm}
 
 \subsection{A Brief History}\label{subsect-history}
 
 In 1981, Gromov introduced the Gromov-Hausdorff distance between Riemannian
 manifolds as an intrinsic version of the Hausdorff distance.  Recall that the Hausdorff distance measures distances between subsets in a common metric space \cite{Gromov-metric}.
 To measure the distance between Riemannian manifolds, Gromov isometrically 
 embeds the pair of manifolds into a common metric
 space, $Z$, then measures the Hausdorff distance between them in $Z$, and 
 then takes the
 infimum over all isometric embeddings into all common metric spaces, $Z$.  Two
 compact Riemannian manifolds have $d_{GH}(M_1,M_2)=0$ if and only if they are
 isometric.    This notion of distance enables Riemannian geometers to study
 sequences of Riemannian manifolds which are not diffeomorphic to their limits
 and have no uniform lower bounds on their injectivity radii.
 The limits of converging sequences of
 compact Riemannian manifolds with a uniform upper bound on diameter need not
 be Riemannian manifolds at all.  However they are 
 compact  geodesic metric spaces.
 
 Gromov's compactness theorem states that a sequence of compact metric
 spaces, $X_j$, has a  Gromov-Hausdorff converging subsequence to a compact metric
 space, $X$, if and only if there is a uniform upper bound on diameter and a uniform upper
 bound on the function, $N(r)$, equal to the number of disjoint balls of radius $r$
 contained in the metric space.   He observes that manifolds with nonnegative
 Ricci curvature, for example, have a uniform upper bound on $N(r)$
 and thus have converging subsequences \cite{Gromov-metric}.
 Such sequences need not have uniform lower bounds on their injectivity radii
(c.f. \cite{Perelman-example}) and their limit spaces can have locally infinite topological type
\cite{Menguy-inf-top-type}.    Nevertheless Cheeger-Colding proved these limit
spaces have many intriguing properties which has lead to a wealth of further
research.  One particularly relevant result states
that when the sequence also has a uniform lower bound on volume, 
then the limit spaces are countably $\mathcal{H}^m$ rectifiable of the 
same dimension as the sequence \cite{ChCo-PartIII}.    
However, Gromov-Hausdorff convergence does not apply well to
sequences with positive scalar curvature.
 
In 2004, Ilmanen described the following example of a sequence of 
three dimensional spheres with positive scalar curvature which has
no Gromov-Hausdorff converging subsequence.  He felt the sequence should
converge in some weak sense to a standard sphere [Figure~\ref{fig-hairy-sphere}].

\begin{figure}[h] 
   \centering
   \includegraphics[width=4.3in]{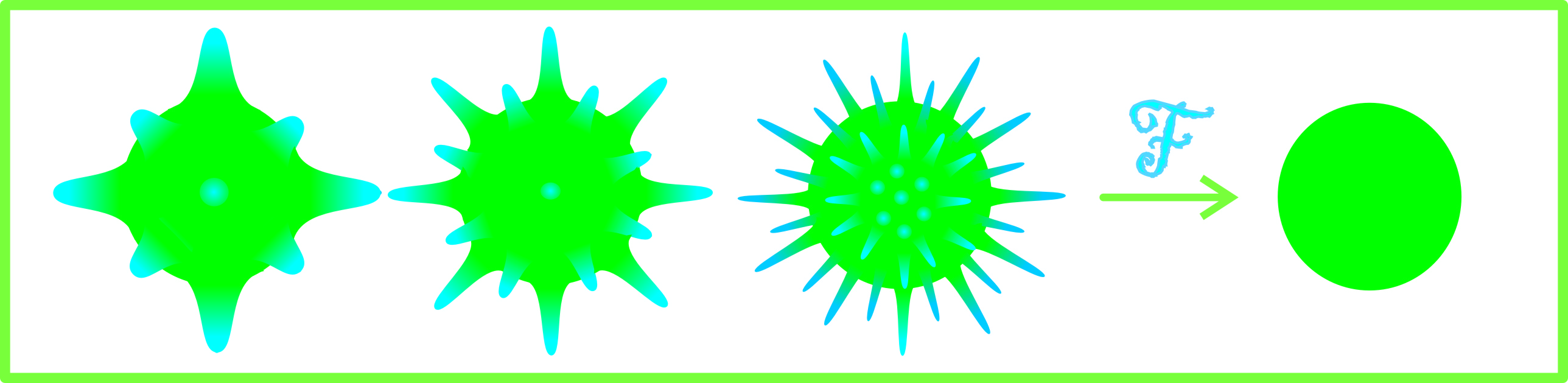} 
   \caption{Ilmanen's sequence of increasingly hairy spheres}
   \label{fig-hairy-sphere}
\end{figure}

Viewing the Riemannian manifolds in Figure~\ref{fig-hairy-sphere}
as submanifolds of Euclidean space,
they are seen to converge in Federer-Fleming's flat sense
as integral currents to the standard sphere.    One of the beautiful properties of limits
under Federer-Fleming's flat convergence is that they are countably $\mathcal{H}^m$
rectifiable with the same dimension as the sequence.   In light of Cheeger-Colding's
work, it seems natural, therefore, to look for an intrinsic flat convergence whose
limit spaces would be countably $\mathcal{H}^m$ rectifiable metric spaces.
The intrinsic flat distance introduced in this paper leads to exactly this kind of convergence.
The sequence of $3$ dimensional manifolds depicted in
Figure~\ref{fig-hairy-sphere} does in fact converge to the 
sphere in this intrinsic flat sense [Example~\ref{ex-hairy-sphere}].   

Ambrosio-Kirchheim's 2000 paper \cite{AK} developing the theory of currents on
arbitrary metric spaces is an essential ingredient for this paper.  Without it we
could not define the intrinsic flat distance, we could not define an integral current
space and we could not explore the properties of converging sequences.   
Other important
background to this paper is prior work of the second author, particularly
\cite{Wenger-flat}, and a coauthored piece \cite{SorWen1}.
Riemannian geometers may not have read these papers (which are aimed
at geometric measure theorists); so we review key
results as they are needed within.  

\subsection{An Overview} \label{subsect-intro}
 
 In this paper, we view a compact oriented Riemannian manifold with boundary, $M^m$,
 as a metric space, $(X,d)$, with an integral current, $T \in \intcurr_m(M)$,
 defined by integration over $M$: $T(\omega) :=\int_M \omega$.  We write
 $M=(X,d,T)$ and refer to $T$ as the integral current structure.
 Using this structure we can define an intrinsic flat distance
between such manifolds and study the intrinsic flat limits of sequences of 
such spaces.  As an immediate consequence of
the theory of Ambrosio-Kirchheim,
the limits of converging sequences of such spaces are countably $\hm^m$
rectifiable metric spaces, $(X,d)$, endowed with a current structure, 
$T \in \intcurr_m(Z)$, 
which represents an orientation and a multiplicity on $X$.  

In Section 2 we describe these spaces in more detail referring to them as
{\em $m$ dimensional integral current spaces} [Defn~\ref{defn-current-space}]
[Defn~\ref{defn-integral-current-space}].
The class of such spaces is denoted ${\mathcal M}^m$ and includes the
zero current space, denoted $\bf{0}=(0,0,0)$. 
Given an integral current space $(X,d,T)$, we define its boundary using the boundary,
$\partial T$, of the integral current structure
[Defn~\ref{defn-integral-current-space}].  We also define the mass
of the space using the mass, $\mass(T)$, of the current structure
[Defn~\ref{defn-space-mass}].  When
$(X,d,T)$ is an oriented Riemannian manifold, the boundary is
just the usual boundary and the mass is just the volume.

Recall that the flat distance between $m$ dimensional integral currents 
$S,T\in\intcurr_m\left(Z\right)$ is given by 
\begin{equation} \label{eqn-Federer-Flat}
d^Z_{F}\left(S,T\right):= 
\inf\{\mass\left(U\right)+\mass\left(V\right):
S-T=U+\bdry V \}
\end{equation}
where $U\in\intcurr_m\left(Z\right)$ and $V\in\intcurr_{m+1}\left(Z\right)$.
This notion of a flat distance was first introduced by Whitney
in \cite{Whitney} and later adapted to rectifiable currents by Federer-Fleming \cite{FF}.
The flat distance between integral currents on an arbitrary metric space was
introduced by the second author in \cite{Wenger-flat}.  

Our definition of the intrinsic flat distance between elements 
of ${\mathcal M}^m$ is modeled after Gromov's intrinsic Hausdorff distance \cite{Gromov-metric}:

\begin{defn} \label{def-flat1}
 For $M_1=\left(X_1,d_1,T_1\right)$ and $M_2=\left(X_2,d_2,T_2\right)\in\mathcal M^m$ let the 
intrinsic flat distance  be defined:
 \begin{equation}\label{equation:def-abstract-flat-distance}
  d_{\Fm}\left(M_1,M_2\right):=
 \inf d_F^Z
\left(\varphi_{1\#} T_1, \varphi_{2\#} T_2 \right),
 \end{equation}
where the infimum is taken over all complete metric spaces 
$\left(Z,d\right)$ and isometric embeddings 
$\varphi_1 : \left(\bar{X}_1,d_1\right)\to \left(Z,d\right)$ and $\varphi_2: \left(\bar{X}_2,d_2\right)\to \left(Z,d\right)$
and the flat norm $d_F^Z$ is taken in $Z$.
Here $\bar{X}_i$ denotes the metric completion of $X_i$ and $d_i$ is the extension
of $d_i$ on $\bar{X}_i$, while $\phi_\# T$ denotes the push forward of $T$.
\end{defn}

All notions from Ambrosio-Kirchheim's work needed to understand this definition are reviewed in detail in Section 2.  As in Gromov, an isometric embedding is a map 
$\phi: A \to B$ which preserves distances not just the Riemannian metric tensors:
\be \label{def-isom-embed}
d_B\left(\phi\left(x\right),\phi\left(y\right)\right)=d_A\left(x,y\right) \qquad \forall x,y \in A.
\ee
For example a map $f: S^1 \to D^2$ mapping the circle to the boundary of a flat
disk is not an isometric embedding while the map $\varphi:S^1 \to S^2$
mapping the circle to a great circle in the sphere is an isometric embedding.
If the infimum in (\ref{equation:def-abstract-flat-distance}) were taken over maps
preserving the Riemannian metric tensors rather than isometric embeddings in
the sense of Gromov, then the value would not be positive.  

It is fairly easy to estimate the
intrinsic flat distances between 
compact oriented Riemannian manifolds using standard
methods from Riemannian geometry.  If $M^m_1$ and $M^m_2$ 
are $m$ dimensional Riemannian manifolds which isometrically
embed into an $m+1$ dimensional Riemannian manifold, $V$, such that
the boundary, $\partial V= \varphi(M_1) \disjointunion \varphi(M_2) \disjointunion U$,
then by (\ref{eqn-Federer-Flat}) we have
$$
d_{\mathcal F}(M_1, M_2) \le \vol_{m}(U) +\vol_{m+1}(V).
$$
This technique and others are applied in the Appendix to
explicitly compute the intrinsic flat limits of converging sequences of 
manifolds depicted here.

It should be noted that $d_{\mathcal{F}}(M,\bf{0})$  is related
to Gromov's filling volume of a manifold \cite{Gromov-filling}
via \cite{Wenger-flat} and \cite{SorWen1}.
DePauw and Hardt have recently defined a flat norm
a la Gromov for chains in a metric space. When the chain is an isometrically
embedded Riemannian manifold, $M$, then their "flat norm" of $M$
seems to take on the same value as $d_{\mathcal{F}}(M,\bf{0})$
\cite{DH-chains}. \footnote{See \cite{DH-chains} page 20 and page 26.}

In Section 3 we explore the properties of our intrinsic flat distance, $d_\mathcal{F}$.
It is always finite 
and, in particular, 
satisfies $d_{\Fm}\left(M_1, M_2\right) \le \vol\left(M_1\right)+\vol\left(M_2\right)$
when $M_i$ are compact oriented Riemannian manifolds [Remark~\ref{finite}].
We prove $d_{\Fm}$ is a distance on 
$\mathcal{M}^m_0$, the space of precompact
integral current spaces
[Theorem~\ref{zero-mani} and Theorem~\ref{triangle}].  
In particular, for compact oriented Riemannian manifolds,
$M$ and $N$,  $d_{\Fm}\left(M,N\right)=0$ iff there is an orientation
preserving isometry from $M$ to $N$.

Applying the Compactness Theorem of Ambrosio-Kirchheim, we see that when a
sequence of Riemannian manifolds, $M_j$, has volume uniformly bounded above
and converges in the Gromov-Hausdorff sense
to a compact metric space, $Y$, then a subsequence of the $M_j$ converges to an integral
current space, $X$, where $X \subset Y$ [Theorem~\ref{GH-to-flat}].    
Example~\ref{example-one-hair} depicted in Figure~\ref{figure-one-hair},
demonstrates that the intrinsic flat and Gromov-Hausdorff limits
need not always agree: the Gromov-Hausdorff  
limit is a sphere with an interval attached while the intrinsic flat limit is just the sphere.

\begin{figure}[h] 
   \centering
   \includegraphics[width=4.7in]{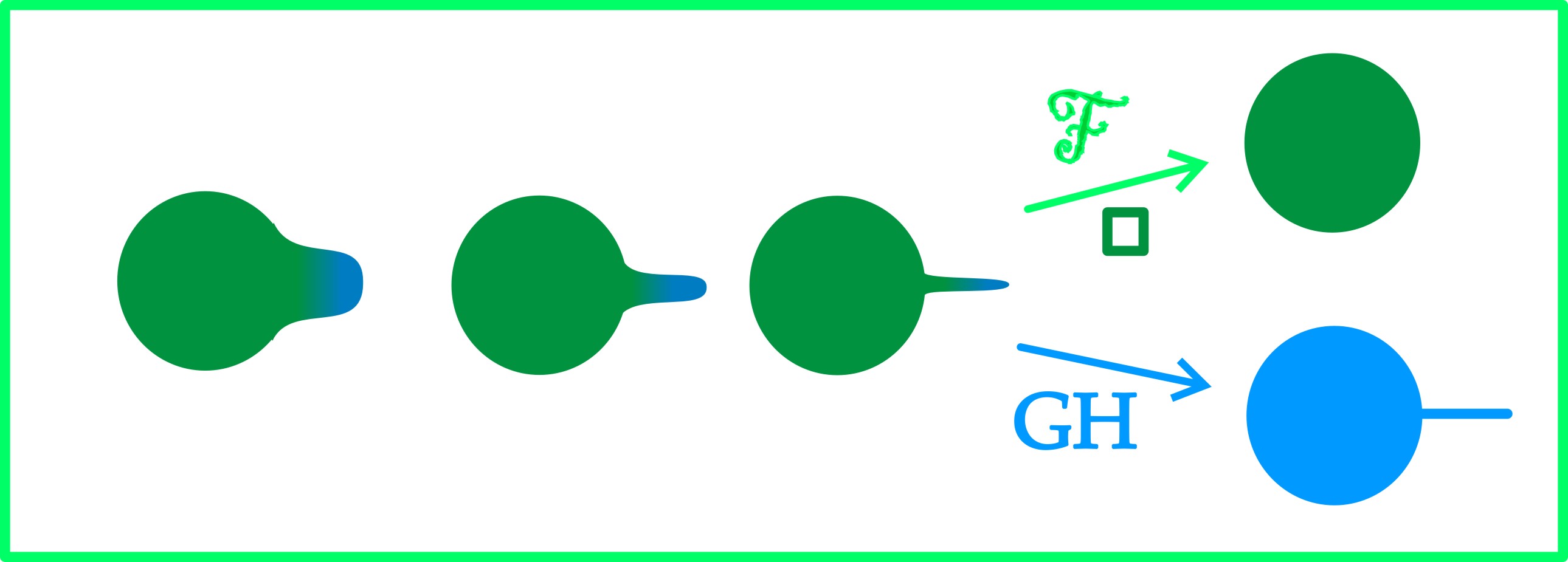} 
   \caption{A sphere with a disappearing hair [Ex~\ref{example-one-hair}].}
   \label{figure-one-hair}
\end{figure}

Gromov-Hausdorff limits of Riemannian manifolds are geodesic spaces.  Recall that
a geodesic space is a metric space such that
\be
d(x,y)=\inf\{ L(c): \,\, c \textrm{ is a curve s.t. }c(0)=x, c(1)=y\}
\ee
and the infimum is attained by a curve called a geodesic segment.
In Example~\ref{example-not-length} 
depicted in Figure~\ref{figure-not-length},
we show that the intrinsic flat limit of Riemannian manifolds need not 
be a geodesic space.  
In fact the intrinsic flat limit is not even path connected.

\begin{figure}[h] 
   \centering
   \includegraphics[width=4.7in]{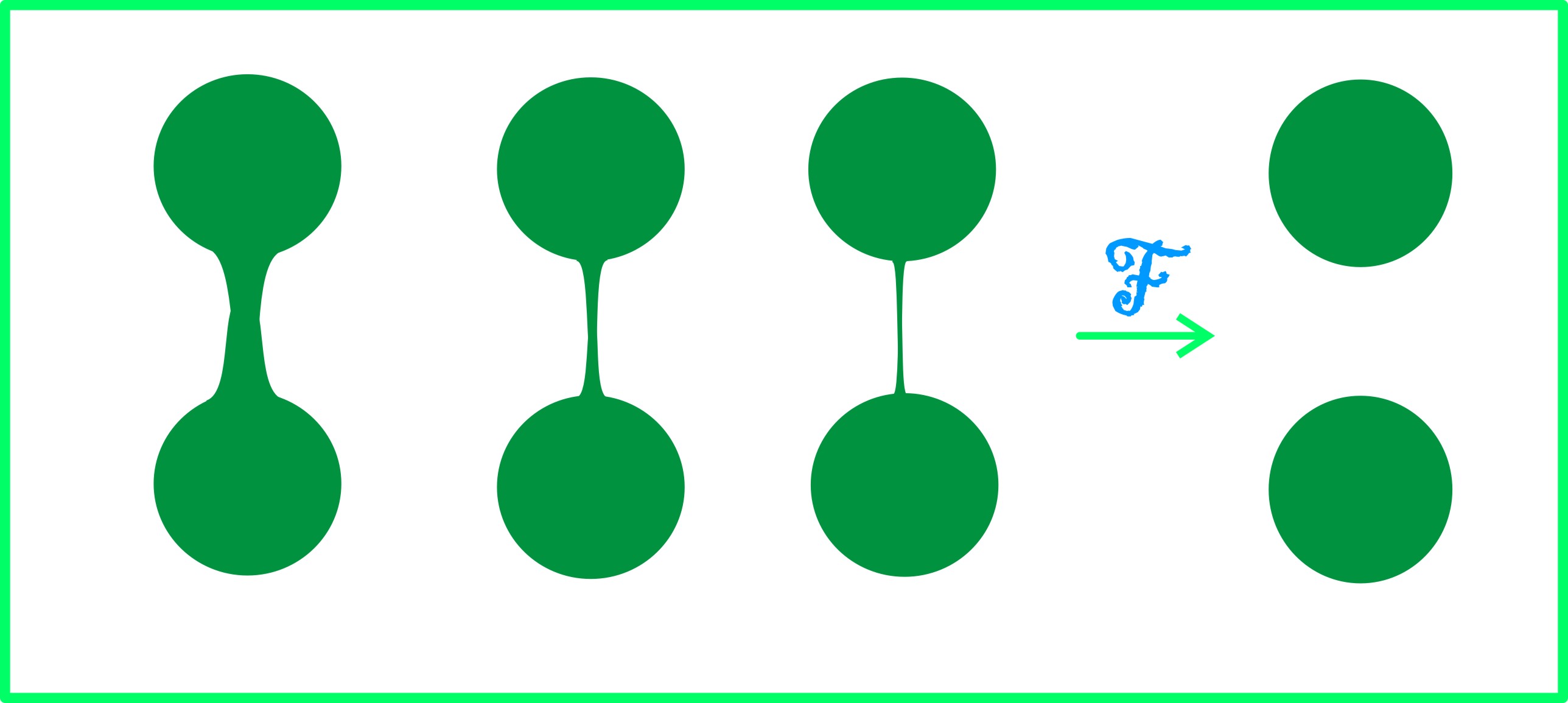} 
   \caption{The intrinsic flat limit is a disjoint pair of spheres [Ex~\ref{example-not-length}].}
   \label{figure-not-length}
\end{figure}


While the limit spaces are not geodesic spaces, they are countably $\mathcal{H}^m$
rectifiable metric spaces of the same dimension.  
These spaces, introduced and studied by Kirchheim in \cite{Kirchheim}, are
covered almost everywhere by the bi-Lipschtiz charts of Borel sets in $\R^m$.  
Gromov-Hausdorff limits do not in general have rectifiability properties.

An interesting example
of such a space is depicted in Figure~\ref{figure-gym2} [Example~\ref{example-gym2}].   
The intrinsic flat limit depicted here is the disjoint collection of spheres while the
Gromov-Hausdorff limit has line segments between them.

\begin{figure}[h] 
   \centering
   \includegraphics[width=4.7in]{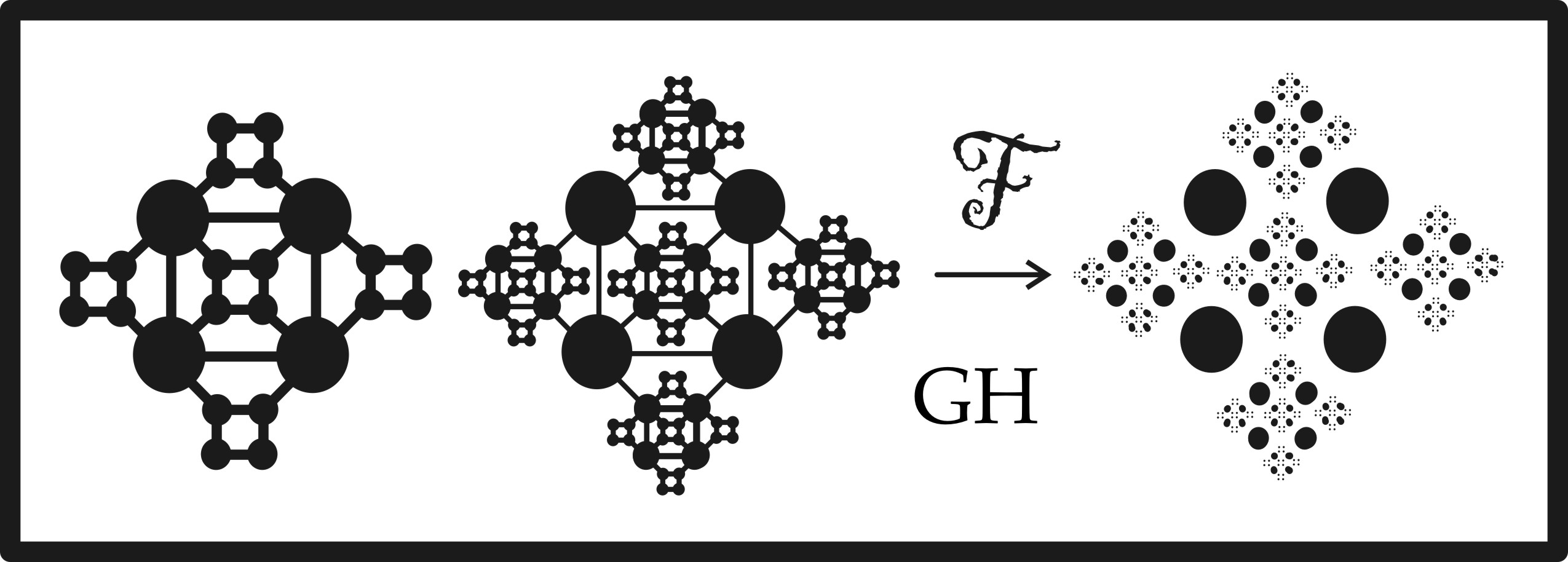} 
   \caption{The limit is a countably ${\mathcal H}^m$ rectifiable space 
   [Ex~\ref{example-gym2}].}
   \label{figure-gym2}
\end{figure}

If a sequence of Riemannian manifolds, $M^m_j$,
 has volume converging to 0 or has a Gromov-Hausdorff
limit whose dimension is less than $m$, then the intrinsic flat limit is the zero space
[Remark~\ref{vol-to-zero} and Corollary~\ref{lower-dim}].  
See Figure~\ref{figure-tori-GH} [Example~\ref{example-tori-GH}].  Such
sequences are referred to as collapsing sequences.

\begin{figure}[h] 
   \centering
   \includegraphics[width=4.7in]{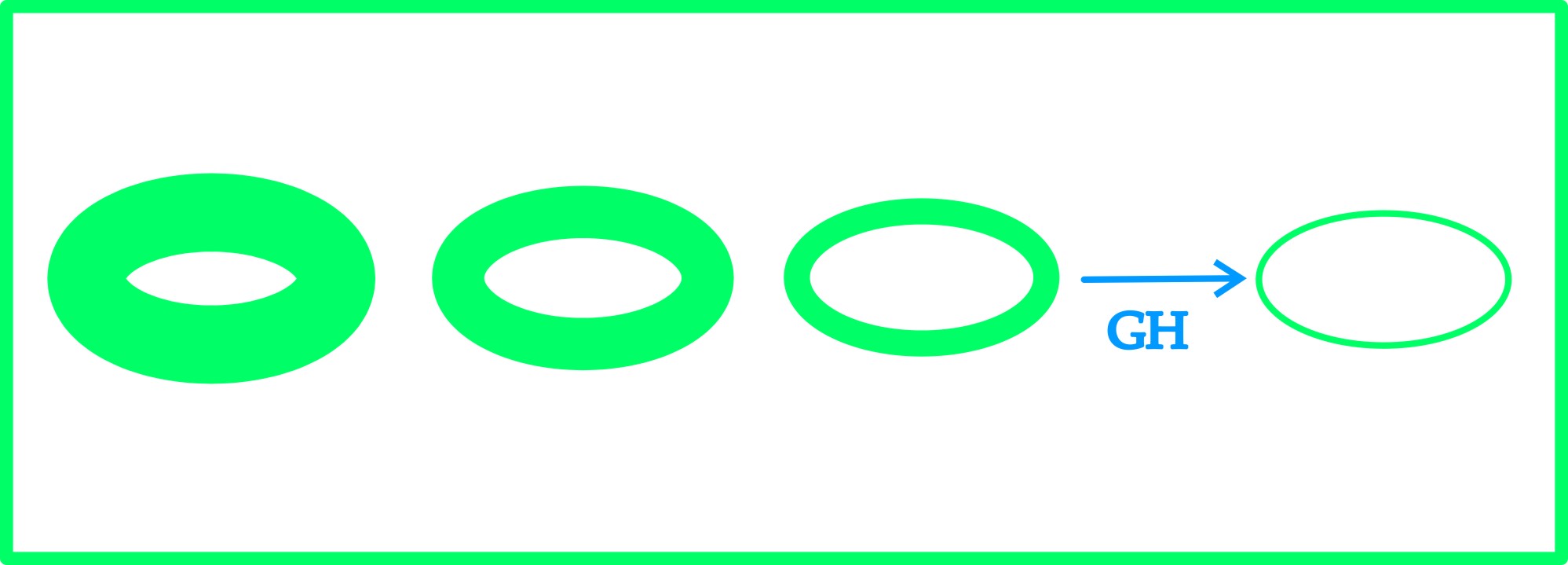} 
   \caption{The Gromov-Hausdorff limit is lower dimensional and the intrinsic flat limit is the zero space [Example~\ref{example-tori-GH}].}
   \label{figure-tori-GH}
\end{figure}

Sequences may also converge to the zero integral current space due to an effect called
cancellation.  With significantly growing local topology,
a sequence of $M^m_j$ which Gromov-Hausdorff converges to a Riemannian manifold, 
$X$, of the same dimension might cancel with itself so that $Y=0$.
In \cite{SorWen1}, the authors gave an example of two
standard three dimensional spheres joined together by increasingly dense tunnels,
providing a sequence of compact manifolds of positive scalar curvature
which converges in the Gromov-Hausdorff sense to a standard sphere.
However the sequence could be isometrically embedded into a common space
$\varphi_j: M_j\to Z$ such that $\varphi_{j\#}M_j$ converges in the flat sense to
$0$ due to cancellation.  Thus $M_j \Fto \bf{0}$.
In Figure~\ref{figure-cancels} we depict a two dimensional example.  Here
two sheets are joined together by many tunnels so that they isometrically
embed into the boundary of a Riemannian manifold of arbitrarily
small volume.

\begin{figure}[h] 
   \centering
   \includegraphics[width=4.7in]{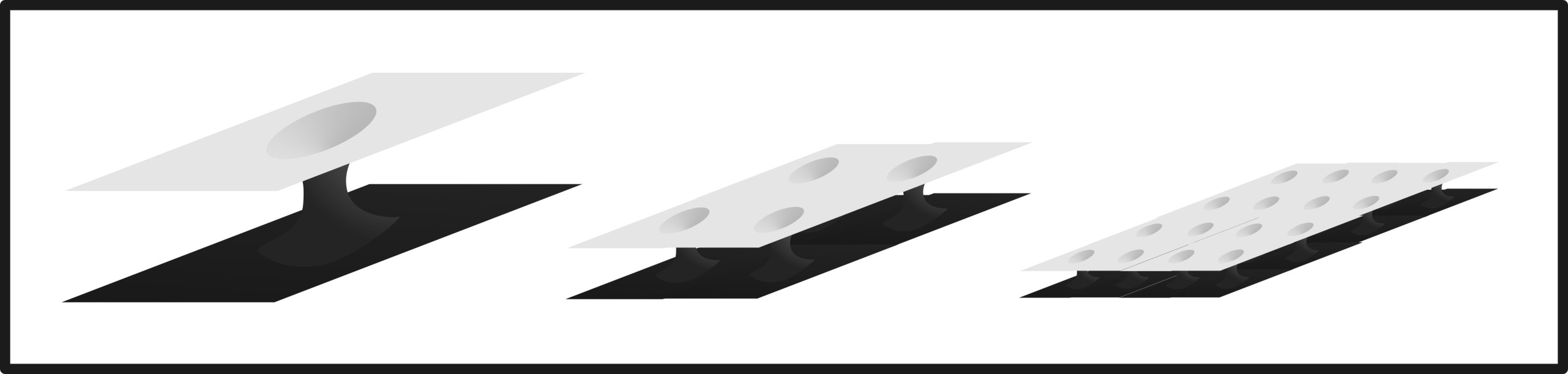} 
   \caption{A sequence converging in the intrinsic flat sense to the zero space
   due to cancellation [Example~\ref{example-cancels}].}
   \label{figure-cancels}
\end{figure}

It is also possible for a sequence of Riemannian manifolds with increasing
local topology to overlaps with itself so that the limit $Y=2X$ 
[Example~\ref{example-doubles}].   If one provides a twist in the middle of 
each tunnel in Figure~\ref{figure-cancels} so as to flip the orientation
of one of the two sheets, then the sequence of manifolds doesn't cancel in
the limit but instead doubles.  We say the limit space has weight or multiplicity
$2$.  In general, intrinsic flat limit spaces have a weight function, which is an
integer valued Borel measurable function, just like integral currents 
[Defn~\ref{def-weight}].

In Section 4 we examine the properties of intrinsic flat convergence.  
We first have a section proving that converging and Cauchy sequences
embed into a common metric space.  This allows us to then immediately
extend properties of weakly converging sequences of integral currents
to integral current spaces.  In particular the mass is lower semicontinuous
as in Ambrosio-Kirchheim \cite{AK} and the the filling volume
is continuous as in \cite{Wenger-flat}.

When $M_j^m$ have nonnegative Ricci curvature, the intrinsic flat limits and
Gromov-Hausdorff limits agree  \cite{SorWen1}.  In this sense one may think of 
intrinsic flat convergence as a means of extending to a larger class of manifolds the
rectifiability properties already proven by Cheeger-Colding to hold on Gromov-Hausdorff
limits of noncollapsing sequences of such manifolds \cite{ChCo-PartI}.  

When $M^m_j$
have a common lower bound on injectivity radius or a uniform linear local
contractibility radius, then work of Croke applying Berger's volume estimates
and work of Greene-Petersen applying Gromov's filling volume inequality imply that a subsequence
of the $M^m_j$ converge in the Gromov-Hausdorff sense 
\cite{Croke-inj}\cite{Greene-Petersen}.  In \cite{SorWen1}, the
authors proved cancellation does not occur in that setting either, so that the
Gromov-Hausdorff limit $X$ agrees with the flat limit $Y$ and is countable $\hm^m$
rectifiable.  

The second author has 
proven a compactness theorem:
{\em Any sequence of oriented Riemannian manifolds with
boundary, $M_j^m$, with a uniform upper bound on $\diam\left(M_j^m\right)$, 
$\vol_m\left(M_j^m\right)$ and
$\vol_{m-1}\left(\partial M_j^m\right)$ always has a subsequence which converges in the intrinsic flat
sense to an integral current space} \cite{Wenger-compactness}.
In fact Wenger's compactness theorem holds for integral current spaces.
We do not apply this theorem in this paper except
for a few immediate corollaries given in Subsection 4.5 and occasional footnotes.

Unlike the limits in Gromov's compactness theorem, the sequences in Wenger's 
compactness theorem need not
converge to a compact limit space.  In Figure~\ref{fig-many-tips}
we see that the limit space need not be precompact  even when the
sequence of manifolds has a uniform upper bound
on volume and diameter.   

\begin{figure}[h] 
   \centering
   \includegraphics[width=4.7in]{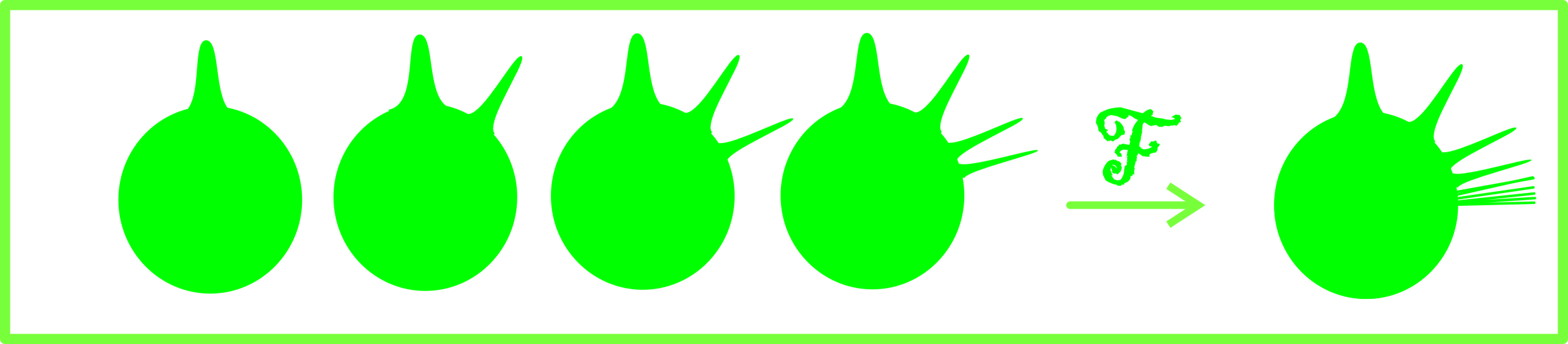} 
   \caption{Spheres with increasingly thin extra bumps converging to a bounded noncompact limit [Ex~\ref{ex-many-tips}].}
   \label{fig-many-tips}
\end{figure}

In Section 5, we describe the relationship between the intrinsic flat convergence
of Riemannian manifolds and other forms of convergence including $C^\infty$
convergence, $C^{k,\alpha}$ convergence, and
Gromov's Lipschitz convergence. 

In the Appendix by the first author, 
we include many examples of sequences explicitly proving they
converge to their limits.  Although the examples are referred to throughout 
the textbook, they are deferred to the final section so that proofs of convergence
may apply any or all lemmas proven in the paper.

While we do not have room in this introduction to refer to all the results presented here,
we refer the reader to the contents at the beginning of the paper and we introduce each
section with a more detailed description of what is contained within it.  Some sections
mention explicit open problems and conjectures.

\subsection{Recommended Reading}
For Riemannian geometry recommended background is a standard
one semester graduate course.  For metric
geometry background, the beginning of Burago-Burago-Ivanov \cite{BBI}
is recommended or Gromov's classic \cite{Gromov-metric}.  For geometric measure
theory a basic guide to Federer is provided in Morgan's textbook \cite{Morgan}.
One may also consult Lin-Yang \cite{Lin-Yang-text}.
We try to cover what is needed from Ambrosio-Kirchheim's seminal
paper \cite{AK}, but we recommend that paper as well.  

\subsection{Acknowledgements}

The first author would like to thank Columbia for its
hospitality in Spring-Summer 2004 and  Ilmanen for many
interesting conversations at that time regarding the necessity of
a weak convergence of Riemannian manifolds and what properties
such a convergence ought to have.  She would also like to
thank Courant Institute for its hospitality in Spring 2007 and 
Summer 2008 enabling the two authors first to develop the
notion of the intrinsic flat distance between Riemannian manifolds
and later to develop the notion of an integral current space in
general extending their prior results to this setting.
The second author would like to thank Courant Institute for
providing such an excellent research environment.  The first
author would also like to thank Paul Yang, Blaine Lawson, Steve Ferry
and Carolyn Gordon for their comments on the 2008 version of the
paper, as well as the participants in the CUNY 2009 Differential
Geometry Workshop\footnote{Marcus Khuri, Michael Munn, Ovidiu Munteanu,
Natasa Sesum, Mu-Tao Wang, William Wylie} for suggestions 
leading to many of the examples added as an appendix that summer.


\section{{\bf Defining Current Spaces}} \label{Sect-Def-Current-Space}


In this section we introduce current spaces $\left(X,d,T\right)$.  
Everything in this section is a reformulation of Ambrosio-Kirchheim's theory
of currents on metric spaces, so that we may clearly define the new notions an
integer rectifiable current space [Defn~\ref{defn-current-space}] and an integral
current space [Defn~\ref{defn-integral-current-space}].
Experts in the theory of Ambrosio-Kirchheim
may wish to skip to these definitions.  In Section~\ref{sect-flat-distance} we will discuss the
intrinsic flat distance between such spaces.  This section is aimed at Riemannian Geometers
who have not yet read Ambrosio-Kirchheim's work \cite{AK}.

In Subsection~\ref{Subsect-Def-Current-Space-1}, we provide a description of these spaces as weighted oriented countably
$\mathcal{H}^m$-rectifiable metric spaces.  Our
spaces need not be complete but must be "completely settled" as defined in Definition~\ref{def-settled}.
In Subsections~\ref{Subsect-Def-Current-Space-2} and~\ref{subsect-param}, we review Ambrosio-Kirchheim's integer rectifiable currents
on complete metric spaces, emphasizing a parametric perspective and proving a
couple lemmas regarding
this parametrization.
In Subsection~\ref{Subsect-Def-Current-Space-3}, we introduce the notion of an integer rectifiable current structure on a metric space [Definition~\ref{defn-current-space}]
and prove in Proposition~\ref{prop-current-spaces} that metric spaces with such current structures are exactly the completely settled
weighted oriented rectifiable metric spaces defined in the first subsection.
In Subsection~\ref{Subsect-Def-Current-Space-4},  we introduce the notion of the boundary of a current space
and define integral current spaces [Definition~\ref{defn-integral-current-space}].


\subsection{Weighted Oriented Countably $\mathcal{H}^m$ Rectifiable Metric Spaces} \label{Subsect-Def-Current-Space-1}

We begin with the following standard definition (\cite{Federer} c.f. \cite{AK}):

\begin{defn}  \label{def-rectspace}
A metric space $X$ is called
countably $\mathcal{H}^m$ rectifiable  iff there exists countably many Lipschitz maps $\varphi_i$ from Borel measurable
subsets $A_i \subset \R^m$ to $X$  such that the Hausdorff measure
\be
\mathcal{H}^m\left( X\setminus \bigcup_{i=1}^\infty \varphi_i\left(A_i\right) \right)=0.
\ee
\end{defn}

\begin{rmrk} \label{defn-metric-differential}
Note that Kirchheim \cite{Kirchheim} defined a metric differential for Lipschitz maps $\varphi:A\subset \R^k \to Z$ where $Z$ is a metric space.  When $A$ is open, 
\be \label{eqn-metric-differential}
md\varphi_y\left(v\right) :=\lim_{r\to 0} \frac {d\left(\varphi\left(y+rv\right),\varphi\left(y\right)\right)}{r},
\ee
if the limit exists.
In fact Kirchheim has proven that for almost every $y \in A$,
$md\varphi_y\left(v\right)$ is defined for all $v \in \R^m$ and $md\varphi_y$ is a seminorm. 
On a Riemannian manifold $Z$ with a smooth map $f$,  $md f_y\left(v\right)=|df_y\left(v\right)|$.  See also Korevaar-Schoen \cite{Korevaar-Schoen}.
\end{rmrk}



In \cite{Kirchheim}, Kirchheim proved
this collection of charts can be chosen so that the maps $\varphi_i$ are bi-Lipschitz.
So we may extend the Riemannian notion of an atlas to this setting:

\begin{defn} \label{defn-atlas-of-X}
A bi-Lipschitz collection of charts, $\{\varphi_i\}$, is called an {\bf atlas} of $X$.
\end{defn}

\begin{rmrk} \label{rmrk-tan-space-is-norm-space}  {\em \bf }
Note that when $\varphi: A \subset \R^m \to X$ is bi-Lipschitz, then
$md\varphi_y$ is a norm on $\R^m$.   In fact there is a notion of
an approximate tangent space at almost every $y\in X$ which is 
a normed space of dimension $m$ whose norm is defined by the metric
differential of a well chosen bi-Lipscitz chart.   (c.f. \cite{Kirchheim})
\end{rmrk}



Recall that by Rademacher's Theorem
we know that given a Lipschitz function $f: \R^m \to \R^m$, $\nabla f$ is defined $\mathcal{H}^m$ almost everywhere.
In particular given two bi-Lipschitz charts, $\varphi_i, \varphi_j$, $\det [\nabla\left(\varphi_i^{-1} \circ {\varphi}_j\right)]$ is
defined almost everywhere.  So we can extend the Riemannian definitions
of an atlas and an oriented atlas to countably $\mathcal{H}^m$ rectifiable spaces:

\begin{defn} \label{def-orient-atlas}
An  atlas on a countably $\mathcal{H}^m$ rectifiable space
 $X$ is called an {\bf oriented atlas} if the orientations agree on all overlapping charts:
\be \label{eqn-def-oriented}
\det \left[\nabla\left(\varphi_i^{-1} \circ {\varphi}_j\right)\right] >0
\ee
 almost everywhere
on $A_j \cap \varphi^{-1}_j \left( \varphi_i\left(A_i\right) \right)$.
\end{defn}

\begin{defn} \label{def-orientation}
An {\bf orientation} on a countably $\mathcal{H}^m$ rectifiable space
 $X$ is an equivalence class of atlases where two atlases, $\{\varphi_i\}, \{\bar{\varphi}_j\}$ are considered to be equivalent if
their union is an oriented atlas.
\end{defn}

\begin{rmrk}
Given an orientation $[\{\varphi_i\}]$, we can choose a representative atlas such that the charts are pairwise disjoint,
$\varphi_i(A_i)\cap\varphi_j(A_j)=\emptyset$,
and the domains $A_i$ are precompact.  We call such an oriented atlas 
a preferred oriented atlas.
\end{rmrk}

\begin{rmrk}  \label{Lip-mani-charts}
Orientable Riemannian manifolds and, more generally, connected orientable
Lipschitz manifolds have only two
standard orientations  because they are connected metric spaces and
their charts overlap.  
Countably $\mathcal{H}^m$ rectifiable
spaces may have uncountably many orientations as each disjoint 
chart may be flipped on its own.  
Recall that a Lipschitz manifold is a metric space, $X$, such that for all $x \in X$ there
is an open set $U$ about $x$ with a bi-Lipschitz homeomorphism to the open unit ball in Euclidean
space.  
A Lipschitz
manifold is said to be orientable when the bi-Lipschitz maps can be chosen so that 
(\ref{eqn-def-oriented}) holds for all pairs of charts. 

When we say "oriented", we will mean that the orientation has been provided, and
we will always orient Riemannian manifolds and Lipschitz manifolds according to one
of their two standard orientations, and we will always assign them an atlas restricted
from the standard charts used to define them as manifolds.
\end{rmrk}

\begin{defn} \label{def-weight}
A {\bf multiplicity} function (or weight) on a countably $\mathcal{H}^m$ rectifiable space
$X$  with $\mathcal{H}^m(X)<\infty$
is a Borel measurable function $\theta: X \to \N$ whose weighted volume,
\be
Vol\left(X, \theta\right):=\int_X \theta d\mathcal{H}^m,
\ee
is finite.
\end{defn}

Note that on a Riemannian manifold, with multiplicity $\theta=1$, 
the weighted volume is the volume.  Later we will
define the mass of these spaces which will agree with the weighted volume on Riemannian manifolds with
arbitrary multiplicity functions but will not be equal to the weighted volume for
more general spaces.

\begin{rmrk}
Given a multiplicity function and an atlas, one may refine the atlas so that the multiplicity function is constant on the image of
each chart.
\end{rmrk}


Recall the notion of the lower $m$-dimensional density, 
$\theta_{*m}(\mu,p)$, of a Borel measure $\mu$ at $p\in X$ is defined by
\be  \label{eqn-lower-density}
 \Theta_{*m}\left(\mu, p\right):= \liminf_{r\to 0} \frac{\mu(B_p(r))}{\omega_m r^m}.
\ee

We introduce the following new concept:

\begin{defn} \label{def-settled}
A weighted oriented countably $\mathcal{H}^m$ rectifiable metric space, $\left(X, d, [\{\phi_i\}], \theta\right)$, is 
called {\bf completely settled} iff
\be
X=\{p\in \bar{X}: \,\, \Theta_{*m}\left(\theta \mathcal{H}^m, p\right)>0\}.
\ee
\end{defn}

\begin{example} \label{example-settled}
An oriented Riemannian manifold with a conical singular point and constant
multiplicity $\theta=1$, which includes the singular point, is a completely
settled space.  An oriented Riemannian manifold with a cusped singular point and 
constant multiplicity $\theta=1$, which does not
include the singular point is a completely settled space.  In particular a completely settled space need not be complete.

An oriented Riemannian manifold with a cusped singular point $p$
and a multiplicity function, $\theta$,  approaching infinity at $p$ such that
\be
\lim_{r\to 0} \frac{1}{r^m}\int_{B_p\left(r\right)} \theta \, d\mathcal{H}^m >0
\ee
is completely settled only if it includes $p$.  
\end{example}

In Subsection~\ref{Subsect-Def-Current-Space-3} we will define our
current spaces as metric
spaces with current structures.  We will prove in Proposition~\ref{prop-current-spaces}   that a metric space is a nonzero integer rectifiable current space iff it
is a completely settled weighted oriented countably $\mathcal{H}^m$-rectifiable metric space.  Note that the notion of a completely settled space does not appear
in Ambrosio-Kirchheim's work and is introduced here to allow us to understand
current spaces in an intrinsic way.
Integral current spaces will have an added condition that their boundaries 
are integer rectifiable metric spaces as well.

\subsection{Reviewing Ambrosio-Kirchheim's Currents on Metric Spaces} \label{Subsect-Def-Current-Space-2} 

In this subsection we review all definitions and theorems of Ambrosio-Kirchheim
and Federer-Fleming necessary to define current structures on metric spaces
\cite{AK}\cite{FF}.

For readers familiar with the Federer-Fleming theory of currents
one may recall that an $m$ dimensional current, $T$, 
acts on smooth $m$ forms (e.g. $\omega=f d\pi_1\wedge\cdots \wedge d\pi_m$).   
An integer rectifiable current is defined by integration over a rectifiable set
in a precise way with integer weight
and the notion of the boundary of $T$ is defined as in
Stokes theorem: $\partial T(\omega)=T(d\omega)$.  
This approach extends naturally
to smooth manifolds but not to metric spaces which do not have differential
forms.  

In the place of differential forms, Ambrosio-Kirchheim use
$m+1$ tuples, $\omega\in \mathcal{D}^m(Z)$,
\be
\omega=f\pi=\left(f,\pi_1 ..., \pi_m\right) \in \mathcal{D}^m(Z)
\ee 
where
$f: Z \to \R$ is a bounded Lipschitz function and
$\pi_i: Z \to \R$ are Lipschitz.  They credit this approach to DeGiorgi \cite{DeGiorgi}.

In \cite{AK} Definitions 2.1, 2.2, 2.6 and 3.1, an $m$ dimensional
current $T\in \curr_m(Z)$ is defined
as a multilinear functional on $\mathcal{D}^m(Z)$ such that
$T\left(f,\pi_1,..., \pi_m\right) $ satisfies a variety of functional properties similar to $T\left(\omega\right)$ where
$\omega= f d\pi_1 \wedge \dots \wedge d\pi_m$ in the
smooth setting as follows:  

\begin{defn}[Ambrosio-Kirchheim]
An
$m$ dimensional {\bf \em current}, $T$, on a complete metric space, $Z$, 
is a real valued
{\em multilinear functional} on $\mathcal{D}^m(Z)$, with the
following required properties:

i) {\bf Locality}:
\be \label{def-locality}
T(f, \pi_1,..., \pi_m)=0 \textrm{ if }\exists i\in \{1,...m\} \textrm{ s.t. }\pi_i
\textrm{ is constant on a nbd of } \{f\neq0\}.
\ee

ii) {\bf Continuity}: 
$$
T \textrm{ is continuous with respect to the pointwise convergence
of the }\pi_i \textrm{ such that } \Lip(\pi_i)\le 1.
$$

iii) {\bf Finite mass}: there exists a
finite Borel measure $\mu$ on $Z$ such that
\be \label{def-AK-current-iii}
|T(f,\pi_1,..., \pi_m)| \le \prod_{i=1}^m \Lip(\pi_i)  \int_Z |f| \,d\mu \qquad \forall (f,\pi_1,..., \pi_m)\in \mathcal{D}^m(Z).
\ee
The space of $m$ dimensional currents on $Z$, is denoted, $\curr_m(Z)$.
\end{defn}

\begin{example}  \label{basic-current}
Given an $L^1$ function $h: A \subset \R^m \to \Z$, one can define an
$m$ dimensional current $\Lbrack h \Rbrack $
as follows
\be \label{def-current-from-function}
\Lbrack h \Rbrack \left(f, \pi\right) := \int_{A \subset \R^m}  h f \det\left(\nabla \pi\right) \, d\mathcal{L}^m =\int_{A \subset \R^m} hf \, d\pi_1 \wedge \dots \wedge d\pi_m.
\ee
Given a Borel measurable set, $A\subset \R^m$, the
current $\Lbrack \One_A \Rbrack$ is defined by the indicator function $\One_A:\R^m \to \R$.  Ambrosio-Kirchheim prove $\Lbrack h \Rbrack \in \curr_m(Z)$ \cite{AK}. 
\end{example}
\begin{rmrk}
  Stronger versions of locality and continuity, as well as product and
chain rules are  proven in \cite{AK}[Theorem 3.5].  In particular, they prove
\be
T(f, \pi_{\sigma(1)},..., \pi_{\sigma(m)})= \sgn(\sigma) T(f, \pi_1,..., \pi_m)
\ee
for any permutation, $\sigma$, of $\{1,2,..., m\}$.
\end{rmrk}

The following definition will allow us to define the most
important currents explicitly:

\begin{defn}[Ambrosio-Kirchheim]
Given a Lipschitz map $\varphi:Z\to Z'$, the {\em push
forward} of a current $T\in \curr_m(Z)$
to a current $\varphi_\# T \in \curr_m(Z')$ is given in \cite{AK}[Defn 2.4] by
\be \label{def-push-forward}
\varphi_\#T(f,\pi_1,..., \pi_m):=T(f\circ \varphi, \pi_1\circ\varphi,..., \pi_m\circ\varphi)
\ee
exactly as in Federer-Flemming when everything is smooth. 
\end{defn}

\begin{example}\label{basic-current-pushed}
If one has a bi-Lipschitz map, $\varphi:\R^m \to Z$, and 
a Lebesgue function $h\in L^1(A,\Z)$ where $A\subset \R^m$,
then $\varphi_\# \Lbrack h \Rbrack \in \curr_m(Z)$ is an example of an 
$m$ dimensional current in $Z$.  Note that
\be
\varphi_\# \Lbrack h \Rbrack (f,\pi_1,..., \pi_m)=
\int_{A \subset \R^m} (h\circ \varphi )(f\circ\varphi) \, 
d(\pi_1\circ \varphi) \wedge \dots \wedge d(\pi_m\circ\varphi)
\ee
where $d(\pi_i\circ\varphi)$ is well defined almost everywhere
by Rademacher's Theorem.  All currents of importance in this
paper are built from currents of this form.
\end{example}

The following are Definition 2.3
and Definition 2.5  in \cite{AK}:

\begin{defn}[Ambrosio-Kirchheim]\label{rmrk-def-boundary}
The {\bf boundary} of $T\in \curr_{m+1}(Z)$ is defined
\be \label{def-boundary}
\partial T(f, \pi_1, ... , \pi_m):= T(1, f, \pi_1,..., \pi_m) 
\ee
since in the smooth setting 
\be
\partial T(f d\pi_1 \wedge \cdots \wedge d\pi_m)
=T(1df\wedge d\pi_1\wedge \cdots\wedge d\pi_m).
\ee
Note that $\varphi_\#(\partial T)=\partial(\varphi_\#T)$ and $\partial \partial T=0$. 
\end{defn} 
 
\begin{defn}[Ambrosio-Kirchheim] 
The {\bf restriction} $T\rstr \omega\in \curr_m(Z)$
of a current $T\in M_{m+k}(Z)$ by a $k+1$ tuple
 $\omega=(g,\tau_1,..., \tau_k)\in \mathcal{D}^k(Z)$:
\be
(T\rstr\omega)(f,\pi_1,..., \pi_m):=T(f\cdot g, \tau_1,..., \tau_k, \pi_1,..., \pi_m).
\ee
\end{defn}

The following definition of the mass of a current is technical \cite{AK}[Defn 2.6].  
A simpler
formula for mass will be given in  Lemma~\ref{lemma-weight} when we restrict
ourselves to integer rectifiable currents.

\begin{defn}[Ambrosio-Kirchheim]  \label{defn-mass}  
The {\bf mass measure} $\|T\| $
of a current $T\in \curr_m(Z)$, is the smallest Borel measure $\mu$ such that
(\ref{def-AK-current-iii}) holds for all $m+1$ tuples, $(f,\pi)$.
%
The {\bf mass} of $T$ is defined
\be \label{def-mass-from-current}
M\left(T\right) = || T || \left(Z\right) = \int_Z \, d\| T\|.
\ee
\end{defn}

In particular
\be \label{eqn-mass}
\Big| T(f,\pi_1,..., \pi_m) \Big| \le \mass(T) |f|_\infty \Lip(\pi_1) \cdots \Lip(\pi_m).
\ee

Note that the currents in $\curr_m(Z)$
defined by Ambrosio-Kirchheim have finite mass by definition.  Urs Lang 
develops a variant of Ambrosio-Kirchheim theory that does not rely
on the finite mass condition in \cite{Lang-local-currents}.

Note the integral current, $\Lbrack h\Rbrack \in \curr_m(\R^m)$,
in Example~\ref{basic-current} has mass measure,
\be
||\Lbrack h \Rbrack ||= |h| d\mathcal{L}^m
\ee
and mass
\be
\mass\left(\Lbrack h \Rbrack \right) =\int_A |h| d\mathcal{L}^m.
\ee

\begin{rmrk}
In (2.4) \cite{AK}, Ambrosio-Kirchheim show that
\be  \label{mass-push}
||\varphi_\#T|| \le [\Lip(\varphi)]^m \varphi_\# ||T||,
\ee
so that when $\varphi$ is an isometry $||\varphi_\#T||=\varphi_\#||T||$
and $\mass(T)=\mass\left(\varphi_\#T\right)$.
\end{rmrk}

Computing the mass of the push
forward current in Example~\ref{basic-current-pushed}
is a little more complicated and will be done in the next section.

\subsection{Parametrized Integer Rectifiable Currents}\label{subsect-param}\label{Subsect-Def-Current-Space-3}

Ambrosio and Kirchheim define integer rectifiable currents, $\intrectcurr_m\left(Z\right)$, on an arbitrary complete
metric space $Z$ \cite{AK}[Defn 4.2].   Rather than giving their definition, we will
use their characterization of integer rectifiable currents given in \cite{AK}[Thm 4.5]:
{\em A current $T\in \curr_m(Z)$ is an integer rectifiable current iff it has a parametrization
of the following form:}

\begin{defn}[Ambrosio-Kirchheim] \label{def-param-rep}
A {\bf parametrization} 
$\left(\{\varphi_i\}, \{\theta_i\}\right)$ of an integer rectifiable current $T\in \intrectcurr_m\left(Z\right)$ with $m\ge 1$ is a countable collection of
bi-Lipschitz maps $\varphi_i:A_i \to Z$ with $A_i\subset\R^m$ precompact
Borel measurable and with pairwise disjoint images and
weight functions $\theta_i\in L^1\left(A_i,\N\right)$ such that
\be\label{param-representation}
T = \sum_{i=1}^\infty \varphi_{i\#} \Lbrack \theta_i \Rbrack \quad\text{and}\quad \mass\left(T\right) = \sum_{i=1}^\infty \mass\left(\varphi_{i\#}\Lbrack \theta_i \Rbrack\right).
\ee
The mass measure is
\be
||T|| = \sum_{i=1}^\infty ||\varphi_{i\#}\Lbrack \theta_i \Rbrack ||.
\ee
\end{defn}

Note that the current in Example~\ref{basic-current-pushed} is an integer
rectifiable current.  

\begin{ex} \label{basic-mani}
If one has an oriented  Riemannian manifold, $M^m$, of finite volume
and a bi-Lipschitz map $\varphi:M^m\to Z$, then $T=\varphi_\#\Lbrack\One_M\Rbrack$
is an integer rectifiable current of dimension $m$ in $Z$.  If $\varphi$
is an isometry, and $Z=M$
then $\mass(T)=\vol(M^m)$.   Note further that
$||T||$ is concentrated on $\varphi(M)$ which is a set of Hausdorff dimension $m$.
\end{ex}

In \cite{AK}[Theorem 4.6] Ambrosio-Kirchheim define a canonical set associated with any
integer rectifiable current:

\begin{defn}[Ambrosio-Kirchheim] \label{defn-set}
The {\bf canonical set} of a current, $T$,
 is the collection of points in $Z$ with positive lower density:
\be \label{def-set-current}
\set\left(T\right)= \{p \in Z: \Theta_{*m}\left( \|T\|, p\right) >0\},
\ee
where the definition of lower density is given in (\ref{eqn-lower-density}).
\end{defn}

\begin{rmrk}\label{good-set}
In \cite{AK}[Thm 4.6], Ambrosio-Kirchheim prove 
given a current $T \in \intrectcurr_m\left(Z\right)$ on a 
complete metric space $Z$ with a parametrization
$\left(\{\varphi_i\}, \theta_i\right)$ of $T$, we have
\be  \label{eqn-lem-weight-1}
\mathcal{H}^m\left(\set\left(T\right) \Lambda \bigcup_{i=1}^\infty \varphi_i\left(A_i\right)\right)=0,
\ee
where $\Lambda$ is the symmetric difference, 
\be
A \Lambda B= \left(A \setminus B\right) \cup \left(B \setminus A\right).
\ee
In particular
the canonical set, $\set\left(T\right)$,
endowed with the restricted metric, $d_Z$,
is a countably $\mathcal{H}^m$ rectifiable metric 
space, $\left(\set\left(T\right), d_Z\right)$.
\end{rmrk}

\begin{example} \label{basic-mani-sing}
Note that the current in Example~\ref{basic-mani}, has 
\be \label{basic-mani-sing-1}
\set\left(\varphi_\#\Lbrack \One_M \Rbrack\right)=\varphi(M).
\ee
when $M$ is a smooth oriented Riemannian manifold.
If $M$ has a conical singularity, then (\ref{basic-mani-sing-1})
holds as well.  However if $M$ has a cusp singularity at a point $p$
then
\be \label{basic-mani-sing-1}
\set\left(\varphi_\#\Lbrack \One_M \Rbrack\right)=\varphi(M\setminus\{p\}).
\ee
\end{example}

Recall that the support of a current (c.f. \cite{AK} Definition 2.8) is
\be \label{eqn-def-support}
\spt (T) := \spt ||T|| = \{p \in Z:   \|T\|(B_p(r)) >0\,\, \forall r>0\}.
\ee
Ambrosio-Kirchheim show the closure of $\set(T)$ is $\spt(T)$.

\begin{rmrk} \label{trouble-with-support}
Note that there are integer rectifiable currents $T^m$ on $\R^n$
such that the support is all of $\R^n$.  For example, take a countable dense collection of points $p_j \in \R^3$,
then $X=\bigcup_{j\in \N} \partial B_{p_j}\left(1/2^j\right)$ is the set of the current $T\in \intcurr_m\left(\R^3\right)$ defined
by integration over $X$ and yet the support is $\R^3$.  
\end{rmrk}

\begin{rmrk} \label{def-param-rep-2}
Given a parametrization of an integer rectifiable current $T$ one may refine this parametrization by choosing Borel measurable subsets
$A_i'$ of the $A_i$ such that
$\varphi_i: A'_i \to set\left(T\right)$.  The new collection of maps $\{\varphi_i: A_i'\to Z\}$ is also a parametrization of $T$
and we will call it a settled parametrization.  Unless stated otherwise, all our parametrizations will be settled.   We may also choose precompact
$A_i'\subset A_i$ such that $\varphi_i(A_i')\cap\varphi_j(A_j')=\emptyset$.
We will call such a parametrization a preferred settled parametrization.
\end{rmrk}

Recall the definition of orientation in Definition~\ref{def-orientation} and the definition of
multiplicity in Definition~\ref{def-weight}. 
The next lemma allows one to define the orientation and multiplicity of an integer rectifiable current [Definition~\ref{def-orient-mult-T}].

\begin{lem} \label{lemma-param-equiv}
Given two currents $T,T' \in \intrectcurr_m\left(Z\right)$ on a complete metric space $Z$ and respective parametrizations
$\left(\{\varphi_i\}, \theta_i\right)$, $\left(\{\varphi'_i\}, \theta'_i\right)$ we have $T=T'$ iff the following hold:

i) The symmetric difference satisfies,
\be
\mathcal{H}^m
\left(\bigcup_{i=1}^\infty \varphi_i\left(A_i\right)  \Lambda \bigcup_{i=1}^\infty \varphi_i'\left(A'_i\right) \right)=0.
\ee

ii) The union of the atlases $\{\varphi_i\}$ and $\{\varphi_i'\}$ is an oriented atlas of
\be
X=\bigcup_{i=1}^\infty \varphi_i\left(A_i\right)  \cup \bigcup_{i=1}^\infty \varphi_i'\left(A'_i\right).
\ee

iii) The sums:
\be \label{eqn-theta-prop1}
\sum_{i=1}^\infty \theta_i\circ\varphi_i^{-1}\One_{\varphi_i\left(A_i\right)}
=
\sum_{i=1}^\infty \theta'_i\circ{\varphi'_i}^{-1}\One_{\varphi'_i\left(A'_i\right)}
\qquad
\mathcal{H}^m a.e. \textrm{ on } Z.
\ee
\end{lem} 

\begin{defn} \label{def-orient-mult-T}
Given $T$, the sum in (\ref{eqn-theta-prop1}) will be called the {\bf multiplicity} function, $\theta_T$.  This function is an $\mathcal{H}^m$ measurable
function from $Z$ to $\N \cup \{0\}$.   The uniquely defined equivalence class of oriented 
atlases of $\set\left(T\right)$
will be called the orientation of $T$.   
\end{defn}

A similar result is in \cite{AK}[Thm 9.1]
with a less Riemannian approach to the notion of orientation. The $\theta$ in
their theorem is our $\theta_T$.

\begin{proof} 
We begin by relating some equations and then prove the theorem.

Note that by restricting to $A_{i,j}:= \varphi_i\left(A_i\right) \cap {\varphi'}_j\left({A'}_j\right)$,
we can focus on one term in the parametrization at a time:
\be
T \rstr A_{i,j} = \sum_{k=1}^\infty \varphi_{k\#} \Lbrack \theta_k \Rbrack \rstr A_{i,j} =
\varphi_{i\#} \Lbrack \theta_i \Rbrack \rstr A_{i,j} =
 \varphi_{i\#} \Lbrack \theta_i \One_{\varphi_i^{-1}\left(A_{i,j}\right)} \Rbrack .
\ee
Thus $T \rstr A_{i,j}=T'\rstr A'_{i,j}$ iff
\be \label{eqn41}
\varphi_{i\#} \Lbrack \theta_i \One_{\varphi_i^{-1}\left(A_{i,j}\right)} \Rbrack
= {\varphi'_{j\#}} \Lbrack \theta'_j \One_{{\varphi'}_j^{-1}\left(A_{i,j}\right)} \Rbrack
\,\,\,\textrm{ iff } \,\,\,
\Lbrack \theta'_j \One_{{\varphi'}_j^{-1}\left(A_{i,j}\right)} \Rbrack={ \varphi'_{j\#}}^{-1}\varphi_{i\#} \Lbrack \theta_i \One_{\varphi_i^{-1}\left(A_{i,j}\right)} \Rbrack.
\ee
This is true iff for any Lipschitz function $f$ defined on $A'_j$ we have
\be
\int_{\varphi_j^{'-1}\left(A_{i,j}\right)} \theta'_j \cdot f  \, d\mathcal{L}^m = \int_{\varphi_i^{-1}\left(A_{i,j}\right)} \theta_i \cdot (f
\circ {\varphi'_j}^{-1} \circ \varphi_i )\, \det\left(\nabla\left({\varphi'_j}^{-1}\circ \varphi_i\right) \right) d\mathcal{L}^m.
\ee
By the change of variables formula,
this is true iff 
\be
\int_{\varphi_j^{'-1}\left(A_{i,j}\right)} \theta'_j \cdot f  \, d\mathcal{L}^m = 
\int_{\varphi_j^{'-1}\left(A_{i,j}\right)} (\theta_i \circ \varphi_i^{-1}\circ \varphi'_j )
\cdot f  \sgn\det\left(\nabla(\varphi_i^{-1}\circ\varphi'_j) \right)d\mathcal{L}^m
\ee
because the change of variables formula involves the absolute value of the determinant.
This is true iff the following two equations hold:
\be \label{eqn-new-iii}
\theta'_j=\theta_i\circ \varphi_i^{-1}\circ \varphi'_j  \qquad \mathcal{L}^m \textrm{ a.e. on }{\varphi'_j}^{-1}\left(A_{i,j}\right)
\ee
and
\be \label{eqn-new-ii}
\sgn\det(\nabla(\varphi_i^{-1}\circ\varphi'_j))
=1  \qquad \mathcal{L}^m \textrm{ a.e. on }{\varphi_j'}^{-1}\left(A_{i,j}\right).
\ee

Setting
\be
Y:= \bigcup_{i=1}^\infty \varphi_i(A_i) \textrm{ and } 
Y':= \bigcup_{j=1}^\infty \varphi'_j(A'_j),
\ee
we have $X= Y \cup Y'$ and $\bigcup_{i,j=1}^\infty A_{i,j} = Y \cap Y'$.
Furthermore by Remark~\ref{good-set}, we have
\be \label{eqn-set-Lambda-i}
(i) \qquad \textrm{ iff }\qquad 
\mathcal{H}^m\left(Y\Lambda Y'\right) 
\qquad\textrm{ iff  }\qquad
 \mathcal{H}^m\left(set\left(T\right) \Lambda set\left(T'\right)\right)=0.
\ee

We may now prove the theorem.
If $T=T'$, then $\set\left(T\right)=set\left(T'\right)$ and we have (i).  Furthermore
 $T\rstr A_{i,j}=T'\rstr A_{i,j}$ for all $i,j$ which implies (\ref{eqn-new-ii}) which implies (ii). 
We also have (\ref{eqn-new-iii}), which implies
\be \label{eqn-theta-on-X}
\sum_{i=1}^\infty \theta_i\circ\varphi_i^{-1}\One_{\varphi_i\left(A_i\right)}
=
\sum_{i=1}^\infty \theta'_i\circ\varphi_i^{-1}\One_{\varphi'_i\left(A'_i\right)}
\ee
holds $\mathcal{H}^m$ almost everywhere on $\bigcup_{i,j=1}^\infty A_{i,j} = Y\cap Y'$.
Since we already have (i) then (\ref{eqn-set-Lambda-i}) implies 
(\ref{eqn-theta-on-X}) holds 
 $\mathcal{H}^m$ almost everywhere on $ Y\cup Y'=X$ and we get (iii).

Conversely if (i), (ii), (iii) hold for a pair of parametrizations, then (ii) implies (\ref{eqn-new-ii}) and  (iii) implies
(\ref{eqn-new-iii}).  Thus, by (\ref{eqn41}) we have $T\rstr A_{i,j}= T'\rstr A_{i,j}$ for all $i,j$. 
Summing over $i$ and $j$ we have $T\rstr X = T'\rstr X$.  
By (i) and (\ref{eqn-set-Lambda-i}),
we have
\be
T= T \rstr \bigcup_{i=1}^\infty \varphi_i\left(A_i\right) =T \rstr Y = T'\rstr Y' = T'\rstr \bigcup_{j=1}^\infty \varphi_j\left(A'_j\right)=T'.
\ee
\end{proof}

In Proposition~\ref{prop-current-spaces} we will prove that
{\em if $T\in\intrectcurr_m(Z)$ is an integer rectifiable current, then
$(\set(T), d_Z, [\{\varphi_i\}], \theta_T)$ as defined in Definition~\ref{def-orient-mult-T}
is a completely settled weighted oriented
countably $\mathcal{H}^m$ rectifiable metric space} as in
Definitions~\ref{def-weight} and~\ref{def-settled}.
To prove this we must show $\set(T)$ is completely settled.
Thus we must better understand the relationship between the
mass measure of $T$, $||T||$, which is used to define the canonical
set and the weight $\theta_T \mathcal{H}^m$ which is used to defined
settled.  Both measures must have positive density at the same locations.

\begin{rmrk}
In the proof of \cite{AK}[Theorem 4.6], Ambrosio-Kirchheim note that
\be \label{AK4.6-measure}
||T||=\Theta_{*m}(||T||, \cdot) \mathcal{H}^m\rstr \set(T).
\ee
\end{rmrk}

\begin{example}\label{weight-match}
Suppose $T\in \intcurr_m(M^m)$ in a smooth oriented Riemannian manifold
of finite volume
is defined $T=\Lbrack \One_M\Rbrack$.  Then $\theta_T=1$
while $||T||$ is the Lebesgue measure on $M$.  
Since the Hausdorff
and Lebesgue measures agree on a smooth Riemannian manifold,
we have $\Theta_{*m}(||T||,p)=1$ as well.   
The Hausdorff and Lebesgue measures also agree on
manifolds that have point singularities
as in Example~\ref{basic-mani-sing}, so that $\set(T)$ is completely
settled with respect to $\theta_T d\mathcal{H}^m$ in both cases given in that example
as well.  In that case we again have $\theta_T=1$ everywhere, but
$\Theta_{*m}(||T||,p)=\Theta_{*m}(\theta_T\mathcal{H}^m,p)<1$ at
conical singularities and $0$ at cusp points.
\end{example}

In general, however, the lower density of $T$ need not agree with the
weight, $\theta_T$.
 To find a formula
relating the multiplicity $\theta_T$
to the lower density of $||T||$ we need a notion called the area
factor of a normed space $V$ (c.f. \cite{AK}(9.11)):
\be\label{def-area-factor}
\lambda_V:=\frac{2^m}{\omega_m} \sup \,\,\left\{\frac{\mathcal{H}^m(B_0(1))}
{\mathcal{H}^m(R)} \right\},
\ee
where the supremum is taken over all parallelepipeds $R\subset V$ which 
contain the unit ball $B_0(1)$.  

\begin{rmrk}\label{rmrk-lambda}
In \cite{AK}[Lemma 9.2], Ambrosio-Kirchheim prove that
\be \label{eqn-lem-weight-lambda-pre}
\lambda_V\in [m^{-m/2}, 2^m/\omega_m]
\ee
and 
observe that
$\lambda_V=1$ whenever $B_0(1)$ is a solid ellipsoid.  This
will occur when $V$ is the tangent space on a Riemannian
manifold because the norm is an inner product.    It is also possible that $\lambda_V=1$
when $V$ does not have an inner product norm (c.f. \cite{AK} Remark 9.3).
\end{rmrk}

The following lemma consolidates a few results in \cite{AK} and \cite{Kirchheim}:

\begin{lem}\label{lemma-weight}
Given an integer rectifiable current $T \in \intrectcurr_m(Z)$,
in a complete metric space $Z$ there is a  function 
\be \label{eqn-lem-weight-lambda}
\lambda:\set(T) \to [m^{-m/2}, 2^m/\omega_m]
\ee
satisfying
\be \label{eqn-lem-weight-new-key}
\Theta_{*m}(||T||,x)=\theta_T(x)\lambda(x),   
\ee
for $\mathcal{H}^m$ almost every $x\in \set (T)$
such that
\be \label{eqn-lem-weight-2}
||T||=\theta_T \lambda \mathcal{H}^m \rstr \set(T).
\ee
In particular $\set(T)$ with the restricted metric from $Z$
 is a completely settled 
weighted oriented countably $\mathcal{H}^m$ rectifiable metric space with respect to
the weight function $\theta_T$ defined in Definition~\ref{def-orient-mult-T}.

When $T=\varphi_\#\Lbrack \One_A\Rbrack$, with a bi-Lipschitz function, $\varphi$,
then for $x\in \varphi(A)$ we have
$\lambda(x)=\lambda_{V_x}$ where $V_x$ is $\R^m$
with the norm defined by the metric differential
$md\varphi_{\varphi^{-1}(x)}$.
\end{lem}

\begin{proof}
On the top of page 58 in \cite{AK}, Ambrosio-Kirchheim observe that
for $\mathcal{H}^m$ almost every $x\in S=\set(T)$, one can define
a approximate tangent space $\Tan^m(S,x)$ which
is $\R^m$ with a norm.    
Taking
$\lambda(x)=\lambda_{\Tan^m(S,x)}$ and applying \cite{AK}(9.10), one 
sees they have proven
(\ref{eqn-lem-weight-2}).  
We then deduce (\ref{eqn-lem-weight-new-key}) using the fact
that $\Theta_{*m}(\mathcal{H}^m \rstr \set (T),x)=1$ almost everywhere
\cite{Kirchheim}[Theorem 9].

The bounds on $\lambda$ in (\ref{eqn-lem-weight-lambda})
come from (\ref{eqn-lem-weight-lambda-pre}) and they allow us 
to conclude that the lower density of $\theta_T \mathcal{H}^m$
and the lower density of  $||T||$ are positive at the same collection
of points.

Examining the proof of \cite{AK}, Theorem 9.1, one sees that
$V_x=\Tan^m(S,x)$ in this setting.
\end{proof}



In this section we introduce the notion of an integer rectifiable current structure on a metric space and define
integer rectifiable current spaces.  We then prove Proposition~\ref{prop-current-spaces} that integer rectifiable current
spaces are completely settled weighted oriented $\mathcal{H}^m$ rectifiable metric spaces using the lemmas
from Subsection~\ref{Subsect-Def-Current-Space-2}.

\begin{defn} \label{def-current-structure} \label{defn-current-space}
An $m$-dimensional {\bf integer rectifiable
current structure} on a metric space $\left(X,d\right)$ is an integer rectifiable current $T\in\intrectcurr_m\left(\bar{X}\right)$ on the completion, $\bar{X}$,
of $X$ such that  $\set\left(T\right)=X$.  We call such a space an {\bf integer rectifiable current space} and denote it $\left(X,d,T\right)$.

Given an integer rectifiable current space $M=\left(X,d,T\right)$ , we let
$\set\left(M\right)$ and $X_M$ denote $X$,  $d_M=d$ and $\Lbrack M \Rbrack =T $. 
\end{defn}

\begin{rmrk} \label{rmrk-separable}  
By \cite{AK} Defn 4.2, any metric space with an $m$-dimensional current structure must be countably $\mathcal{H}^m$-rectifiable
because the set of an $m$ dimensional integer rectifiable current is countably $\mathcal{H}^m$ rectifiable.
By \cite{AK} Thm 4.5, there is a countably collection of
bi-Lipscitz charts with compact domains
which map onto a dense subset
of the metric space (because we only include points of 
positive density).  In particular, the space is  separable.  
\end{rmrk}

\begin{rmrk}
We do not use the support, $spt(T)$, in this definition as it
is not necessarily countably $\mathcal{H}^m$ rectifiable  
and may have a higher dimension as described in Remark~\ref{trouble-with-support}.  See Example~\ref{example-dense-support}.
\end{rmrk}

\begin{rmrk}  \label{Lip-mani-structure}
Recall that in Remark~\ref{Lip-mani-charts} we said that any $m$ dimensional oriented
connected Lipschitz or Riemannian manifold, $M$, is endowed with a standard atlas of charts 
with a fixed orientation.  We will also view these spaces as having
multiplicity or weight $1$.  If $M$ has finite volume and we've chosen an 
orientation, then we can
define an integer rectifiable current structure, $T=\Lbrack M \Rbrack \in \intrectcurr_m\left(M\right)$,
parametrized by a finite disjoint selection of charts with weight $1$.   It is easy to verify that
$\set \left(T\right)=M$.   
\end{rmrk}


\begin{lem} \label{lemma-isom-to-a-set}\label{lemma-isom-to-set}
Suppose $\left(X,d,T\right)$ is an integer rectifiable current space and $Z$ is a complete metric space.
If ${\phi}: X \to Z$ is an isometric embedding then the induced map on the completion,
$\bar{\phi}:\bar{X} \to Z$, is also an isometric embedding.  Furthermore
the pushforward
$
\bar{\phi}_\#T
$
is an integer rectifiable current on $Z$ and
\be
\phi: X \to \set\left(\bar{\phi}_\# T\right)
\ee
is an isometry.
\end{lem} 

\begin{proof}
Follows from the fact that $\set\left(\bar{\phi}_\# T\right)= \bar{\phi}\left( \set\left(T\right) \right)$ \cite{AK}.
\end{proof}

Conversely, if $T$ is an integer rectifiable current in $Z$, 
then $\left(\set\left(T\right), d_Z, T\right)$ is an an $m$ dimensional integer rectifiable current space.

\begin{prop} \label{prop-current-spaces} 
There is a one-to-one correspondence between completely settled weighted oriented countably $\mathcal{H}^m$ rectifiable metric
spaces, $\left(X, d, [\{\phi\}], \theta\right)$, and integer rectifiable current spaces $\left(X,d,T\right)$ as follows: 

Given $\left(X,d,T\right)$, we define a weight $\theta=\theta_T$
and orientation, $[\{\varphi_i\}]$ as in Definition~\ref{def-orient-mult-T}, 
so that
\be \label{eqn-theta-prop2}
\theta:= \theta_T=\sum_{i=1}^\infty \theta_i\circ\varphi_i^{-1}\One_{\varphi_i\left(A_i\right)},
\ee
and the corresponding space is $(X,d,[\{\varphi_i\}],\theta )$. 

Given $\left(X, d, [\{\varphi\}], \theta\right)$, we define a
unique induced current structure $T\in\intrectcurr_m\left(\bar{X}\right)$ given by
\be \label{eqn-T-prop1}
T\left(f,\pi\right)= \sum \varphi_{i\#}\Lbrack \theta\circ\varphi_i\Rbrack \left(f,\pi\right) = \sum \int_{A_i} \theta\circ\varphi_i f \circ \varphi_i \det\left(\nabla\left(\pi\circ\varphi_i\right)\right)\, d\mathcal{L}^m,
\ee
and the corresponding space is then $(X,d,T)$ because $\set(T)=X$.
\end{prop}

\begin{proof}
Given $\left(X, d, [\{\varphi_i\}], \theta\right)$ we first define a current on the completion $\bar{X}$ using a preferred oriented atlas
as in (\ref{eqn-T-prop1}).  This is well defined because
\be
\sum_{i=1}^\infty M \left(\varphi_{i\#} \Lbrack \theta \circ \varphi_i \Rbrack\right) \le
C_m \sum_{i=1}^\infty \int_{\varphi_i\left(A_i\right)}\theta \, d\mathcal{H}^m<\infty
\ee
where $C_m$ is a constant that may be computed using  Lemma~\ref{lemma-weight}.
The sum is then finite by Definition~\ref{def-weight}.

So we have a current with a parametrization $\left(\{\varphi_i\}, \{\theta_i\}\right)$ where $\theta_i := \theta \circ \varphi_i$.  The weight
function $\theta_T$ of the current $T$ defined below Lemma~\ref{lemma-param-equiv}
agrees with the weight function $\theta$ on $X$ because for almost every $x\in X$
there is a chart such that $x \in \varphi_i\left(A_i\right)$, and
\be \label{eqn-theta-T-theta}
\theta_T\left(x\right)=\theta_i\circ \varphi_i^{-1}\left(x\right)=\theta\left(x\right).
\ee

Furthermore $\set\left(T\right)= \{p \in \bar{X}: \Theta_{*m}\left( \|T\|, p\right) >0\}$,
so by Lemma~\ref{lemma-weight} we have
\be
\set\left(T\right)= \left\{p \in \bar{X}: \Theta_{*m}\left(\theta \, d\mathcal{H}^m \rstr \bigcup_{i=1}^\infty \varphi_i\left(A_i\right), p\right)>0\right\}
\ee
which is $X$ because $X$ is completely settled.  Since $X$ is a countably $\mathcal{H}^m$ rectifiable space, we know
$T \in \intrectcurr_m\left(\bar{X}\right)$.  Thus we have an integer rectifiable current space $\left(X,d,T\right)$.

Conversely we start with $\left(X,d,T\right)$.  Applying Lemma~\ref{lemma-param-equiv}, we have a unique well defined orientation and
weight function $\theta_T$.  Thus $\left(\set\left(T\right), d, [\{\varphi_i\}], \theta_T\right)$ is an oriented weighted countably $\mathcal{H}^m$ rectifiable
metric space.  Since $\set\left(T\right)=X$ in the definition of a current space, we have shown $\left(X, d, [\{\varphi_i\}], \theta_T\right)$
is an oriented weighted countably  
$\mathcal{H}^m$ rectifiable metric space.   As in the above paragraph,
we see that $\set\left(T\right)$ is a completely
settled subset of $\bar{X}$.  So $X$ is completely settled.

Note that since the $\{\varphi_i\}$ from the preferred atlas are the $\{\varphi_i\}$ of the parametrization and the
weights agree in (\ref{eqn-theta-T-theta}), this pair of maps is a correspondence.
\end{proof}

We may now define the mass and relate it to the weighted volume:

\begin{defn} \label{defn-space-mass}
The {\bf mass} of an integer rectifiable current space $\left(X,d,T\right)$ is defined to be the 
mass, $\mass\left(T\right)$, of the current structure, $T$. 
\end{defn}

Note that the mass is always finite by (iii) in the definition of a current.  

\bl \label{lem-push-mass}
If $\varphi: X\to Y$ is a $1$-Lipschitz map, then $\mass(\varphi_\#(T))
\le \mass(T)$.  
Thus if $\varphi: X\to Y$ is an isometric embedding, then $\mass(T)=\mass(\varphi_\#(T))$.  
\el

Recall Definition~\ref{def-weight} of the weighted volume, $Vol\left(X, \theta\right)$.
We have the following corollary of Lemma~\ref{lemma-weight} and Proposition~\ref{prop-current-spaces}:

\bl\label{lemma-mass-current-space}
The mass of an integer rectifiable current space $\left(X,d,T\right)$ with multiplicity or weight, $\theta_T$, satisfies
\be
\mass(T)=\int_X \theta_T(x)\lambda(x) d\mathcal{H}^m(x).
\ee
In particular, 
\be  \label{eqn-lem-weight-3}
M\left(T\right) \in \left[ m^{-m/2}Vol(X,\theta), \frac{2^m}{\omega_m} Vol(X,\theta)\right] ,  
\ee
where $Vol(X,\theta)$ is the weighted volume defined in Definition~\ref{def-weight}.
\el

Note that on a Riemannian manifold with multiplicity one, the mass and the weighted volume agree and are
both equal to the volume of the manifold.  On reversible Finsler spaces, 
$\lambda(x)$ depends on the norm of the tangent space at $x$.


\subsection{Integral Current Spaces} \label{Subsect-Def-Current-Space-4}

In this subsection, we define the boundaries of integer rectfiable current spaces and the notion of
an integral current space. We begin with Ambrosio-Kirchheim's
extension of Federer-Fleming's notion of an
integral current \cite{AK}[Defn 3.4 and 4.2]:

\begin{defn}[Ambrosio-Kirchheim]
An 
{\bf integral current}
is an 
integer rectifiable current,  $ T\in\intrectcurr_m(Z)$, such that $\partial T$ 
defined as
\be
\partial T \left(f, \pi_1,..., \pi_{m-1}\right) := T \left(1, f, \pi_1,..., \pi_{m-1}\right)
\ee
satisfies the requirements to be
a current.  One need only verify that $\partial T$ has finite mass as the
other conditions always hold.
We use the standard notation, $\intcurr_m\left(Z\right)$, to denote 
the space of $m$ dimensional integral currents on $Z$.
\end{defn}

\begin{rmrk} 
By the boundary rectifiability theorem of Ambrosio-Kirchheim
\cite{AK}[Theorem 8.6], $\partial T$ is then an integer rectifiable current
itself.  And in fact it is an integral current whose boundary is $0$.
\end{rmrk}


Thus we can make the following new definition:

\begin{defn} \label{def-integral-space} \label{defn-integral-current-space}
An $m$ dimensional {\bf integral current space} is an
 integer rectifiable current space, 
$\left(X,d,T\right)$, whose current structure, 
$T$, is an integral current (that is
$\partial T$ is an integer
rectifiable current in $\bar{X}$).  The boundary of $\left(X,d,T\right)$ is then the
integral current space:
\be
\partial \left(X,d_X,T\right) := \left(\set\left(\partial T\right), d_{\bar{X}}, \partial T\right).
\ee
If $\partial T=0$ then we say $\left(X,d,T\right)$ is an integral current without boundary or with zero boundary.
\end{defn}

Note that $\set\left(\partial T\right)$ is not necessarily a subset of $\set\left(T\right)=X$ but it is always a subset
of $\bar{X}$.   As in Definition~\ref{defn-current-space}, given an integer rectifiable current 
space $M=\left(X,d,T\right)$ we will use
$\set\left(M\right)$ or $X_M$ to denote $X$,  $d_M=d$ and $\Lbrack M \Rbrack =T $. 

\begin{rmrk}
On an oriented Riemannian manifold with boundary, 
$M$, the boundary $\partial M$ defined as a current space
agrees with the definition of $\partial M$ in Riemannian geometry.  
In that setting an atlas of $M$ can be
restricted to provide an atlas for $\partial M$.    It is not always possible to do this
on integer rectifiable current spaces. 
In fact the boundaries of charts need not even have finite mass for an individual chart.
If a chart $\varphi: K \subset \R^m \to Z$ with $K$ compact, then
$\partial \varphi_\# \Lbrack \One_K\Rbrack$ is an integral current iff $K$ has
finite perimeter.  
\end{rmrk}
 
 \begin{rmrk}  \label{rmrk-biLip-matching}
Suppose  $M$ and $N$ are connected  $m$-dimensional oriented 
Lipschitz manifolds 
with the standard current structures, $\Lbrack  M \Rbrack$ and $\Lbrack N \Rbrack$ 
as in Remark~\ref{Lip-mani-charts}
and  $\psi: M\to N$ a bi-Lipschitz homeomorphism. 
Then one can do a computation mapping charts on $M$ to
charts on $N$ and applying Lemma~\ref{lemma-param-equiv},  to see that
\be \label{rmrk-biLip-matching-1}
\psi_\#\Lbrack M\Rbrack = \pm\Lbrack N\Rbrack.
\ee  
That is, the bi-Lipschitz
homeomorphism is either a current preserving or a current reversing map.
When $M$ and $N$ are isometric, then the isometry is also current preserving or
current reversing.  

When $M$ and $N$ are integral current spaces, they may have multiplicity, so 
that a bi-Lipschitz homeomorphism or isometry from $\set\left(M\right)$ to $\set\left(N\right)$ does not
in general push $\Lbrack  M \Rbrack$ to $\Lbrack N \Rbrack$.  
Even with multiplicity $1$, the fact that orientations are defined using disjoint charts
can lead to different signs on different charts so that (\ref{rmrk-biLip-matching-1}) fails.
\end{rmrk}

 As in Federer, Ambrosio-Kirchheim define the total mass and we do as well:
 
 \begin{defn} \label{def-total-mass}
The {\bf total mass} of an integral current with boundary, $T$, is
\be \label{total-mass}
\nmass\left(T\right)=\mass\left(T\right) +\mass\left(\partial T\right). 
\ee
Naturally we can extend this concept to current spaces: $\nmass\left(X,d,T\right)=\nmass\left(T\right)$.
\end{defn}

Recall that by Remark~\ref{rmrk-separable}, an integral current space is separable and
has a collection of disjoint biLipshitz charts whose image is dense and the boundary of
the integral current space has the same property.  An integral current space need not be
precompact or bounded.  An integral current space is not necessarily a geodesic space.


\section{{\bf The Intrinsic Flat Distance Between Current Spaces}}\label{sect-flat-distance}

Let $\mathcal{M}^m$ be the space of $m$ dimensional integral current spaces as
defined in Definition~\ref{defn-integral-current-space}. 
Recall they have the form $M=\left(X_M,d_M,T_M\right)$ where
$T_M \in \intcurr_m\left(\bar{X}_M\right)$ and $\set(T_M)=X_M$. Note $\mathcal{M}^m$
also includes the zero current denoted ${\bf{0}}$.

Definition~\ref{def-flat1} in the introduction naturally applies to any $M,N \in \mathcal{M}^m$
so that:
\begin{equation}\label{eqn-local-defn}
d_{\Fm}\left(M,N\right):=\inf\{\mass\left(U\right)+\mass\left(V\right)\}
\end{equation}
where the infimum is taken over all complete metric spaces,
$\left(Z,d\right)$, and all integral currents,
$U\in\intcurr_m\left(Z\right), V\in\intcurr_{m+1}\left(Z\right)$,
such that there exists isometric embeddings
\be
\varphi : \left(\bar{X}_M, d_{\bar{X}_M}\right)\to \left(Z,d\right) \textrm{ and }\psi: \left(\bar{X}_N,d_{\bar{X}_N}\right)\to \left(Z,d\right)
\ee
with
\begin{equation} \label{eqn-Federer-Flat-2}
\varphi_\# T_M- \psi_\# T_N=U+\bdry V.
\end{equation}
Here we consider the ${\bf{0}}$ space to isometrically embed into any $Z$
with $\varphi_\#0=0\in \intcurr_m\left(Z\right)$.

Note that, by the definition, $d_{\Fm}$ is clearly symmetric.
In Subsection~\ref{subsect-triangle} we prove 
that $d_{\mathcal{F}}$ satisfies the triangle inequality on $\mathcal{M}^m$
[Theorem~\ref{triangle}].
As a consequence, the distance between integral
current spaces is always finite and is easy to estimate [Remark~\ref{finite}].

In Subsection~\ref{subsect-flat-dist-1}, we review the compactness theorems of 
Gromov and of Ambrosio-Kirchheim,
and present a compactness theorem for
intrinsic flat convergence [Theorem~\ref{GH-to-flat}], which
follows immediately from theirs.

In Subsection~\ref{subsect-flat-dist-2}, we prove
Theorem~\ref{inf-dist-attained} that the infimum in the definition of
the intrinsic flat distance is attained
between precompact integral current spaces.  
That is, there exists a common metric space, $Z$,
and integral currents, $U,V\in \intcurr_m(Z)$, achieving the infimum in (\ref{eqn-local-defn}).

In Subsection~\ref{subsect-flat-dist-3} we prove 
that $d_{\mathcal{F}}$ is a distance on $\mathcal{M}_0^m$.
That is, we prove that when two precompact integral current
spaces are a distance zero apart, there is a current preserving isometry between them [Theorem~\ref{zero-mani}].
Thus $d_{\Fm}$ is a distance on $\mathcal{M}_0^m$ where
\be
\mathcal{M}_0^m =\{ M \in \mathcal{M}^m: \,\, X_M \textrm{ is precompact}\}.
\ee

\begin{rmrk} \label{rmrk-distinct-flat}
Note that the flat distance $d^Z_F$ given above Definition~\ref{def-flat1} 
has an infimum that is taken over all 
$U\in\intcurr_m\left(Z\right), V\in\intcurr_{m+1}\left(Z\right)$ where the supports of $U$ and $V$
may be noncompact or even unbounded as long as they have finite mass.  
Thus  we can have unbounded 
limits [Example~\ref{ex-unbounded}] and bounded noncompact limits
[Example~\ref{ex-many-tips}].
\end{rmrk}

\subsection{The Triangle Inequality} \label{subsect-triangle}

In this section we prove the triangle inequality for the intrinsic flat distance
between integral current spaces:

\begin{thm} \label{triangle}
For all $M_1, M_2, N \in \mathcal{M}^m$, we have
\begin{equation}\label{eqn-triangle-here}
d_{\Fm}\left(M_1,M_2\right)\le d_{\Fm}\left(M_1,N\right)+d_{\Fm}\left(N, M_2\right).
\end{equation}
\end{thm}

In the proof of this theorem, we do not assume the infimum in (\ref{eqn-local-defn})
is finite.   Naturally the theorem is immediately true if the right hand side
of (\ref{eqn-triangle-here}) is infinite.  It is a consequence of the theorem that
when the right hand side is finite, the left hand side is finite as well. 
Applying the theorem with $N_1={\bf 0}$, we may then conclude the
distance is finite and estimate it using the masses of $M_1$ and $M_2$:

\begin{rmrk}\label{finite} 
Taking $U=M$ and $V=0$ in (\ref{eqn-local-defn}), we see that
$
d_{\Fm}\left(M,0\right) \le   \mass\left(M\right),
$
so  the intrinsic flat distance between any pair of integral current spaces of finite mass
is finite 
\be
d_{\Fm}\left(M_1,M_2\right) \le d_{\Fm}\left(M_1,0\right)+d_{\Fm}\left(0,M_2\right) \le \mass\left(M_1\right)+\mass\left(M_2\right).
\ee 
In particular, when $M_i$ are Riemannian manifolds, then $\mass\left(M_i\right)=\vol\left(M_i\right)$ and
we have 
\be
d_{\Fm}\left(M_1,M_2\right) \le \vol\left(M_1\right)+\vol\left(M_2\right).
\ee
\end{rmrk}

To prove Theorem~\ref{triangle} we apply the following well-known
gluing lemma (c.f. \cite{BBI}):

\begin{lem} \label{gluing-lemma}
Given three
metric spaces $\left(Z_1,d_1\right)$, $(Z_2,d_2)$ and $\left(X, d_X\right)$ and two isometric embeddings, $\varphi_i: X \to Z_i$,
we can glue $Z_1$ to $Z_2$ along the isometric images of $X$ to create a space
$Z=Z_1 \disjointunion _{X} Z_2$ where $d_Z\left(x,x'\right)= d_i\left(x,x'\right)$ when $x,x'\in Z_i$ and
\be
d_Z\left(z,z'\right)=\inf_{x\in X} \left(d_1\left(z,\varphi_1\left(x\right)\right) + d_2\left(\varphi_2\left(x\right), z'\right) \right)
\ee
whenever $z\in Z_1, z'\in Z_2$.  There exist natural isometric embeddings $f_i: Z_i \to Z$
such that $f_1\circ\varphi_1=f_2 \circ\varphi_2$ is an isometric embedding of $X$ into $Z$.
\end{lem}

We now prove Theorem~\ref{triangle}:

\begin{proof}
Let $M_i=(X_i, d_i, T_i)$ and $N=(X,d,T)$, and let $Z_1, Z_2$ be
metric spaces and let $\psi_i:\bar{X}_i \to Z_i$ and $\varphi_i:\bar{X} \to Z_i$
be isometric embeddings.  Let $U_i \in \intcurr_m(Z_i)$ and $V_i\in \intcurr_{m+1}(Z_i)$
such that 
\begin{equation} \label{triangle-1}
\varphi_{i\#} T- \psi_{i\#} T_i=U_i+\bdry V_i.
\end{equation}
Applying Lemma~\ref{gluing-lemma}, we create a metric space $Z$ with
isometric embeddings $f_i: Z_i \to Z$ such that 
$f_1 \circ \varphi_1=f_2 \circ \varphi_2$ is an isometric embedding of $X$ into $Z$.
Pushing forward the current structures to $Z$, we have 
$f_{1\#} \varphi_{1\#} T =f_{2\#}\varphi_{2\#}T$, so
\begin{eqnarray}
f_{1\#}\psi_{1\#} T_1 - f_{2\#}\psi_{2\#}T_2 &= &
f_{1\#}\psi_{1\#} T_1 - f_{1\#} \varphi_{1\#} T + f_{2\#}\varphi_{2\#}T-f_{2\#}\psi_{2\#}T_2 \\
&=& f_{1\#}(\psi_{1\#} T_1 - \varphi_{1\#} T) + f_{2\#}(\varphi_{2\#}T-\psi_{2\#}T_2 )\\
&=& f_{1\#} (-U_1 -  \partial V_1) +  f_{2\#}( U_2 +  \partial V_2)\\ 
&=& -f_{1\#} U_1 -  \partial f_{1\#}V_1 +  f_{2\#}U_2 +  \partial f_{2\#}V_2\\ 
&=& f_{2\#} U_2 -f_{1\#}U_1  + \partial (f_{2\#} V_2 -f_{1\#} V_1). 
\end{eqnarray}
So by (\ref{eqn-local-defn}) applied to the isometric embeddings 
$f_i\circ\psi_i:\bar{X}_i\to Z$, we have
\be
d_{\Fm}\left(M_1,M_2\right)
 \le \mass(f_{2\#} U_2 -f_{1\#}U_1) +  \mass (f_{2\#} V_2 -f_{1\#} V_1).
\ee
Applying the fact that mass is a norm and Lemma~\ref{lem-push-mass}
we have,
\begin{eqnarray}
d_{\Fm}\left(M_1,M_2\right)
&\le& \mass(f_{2\#} U_2) +\mass(f_{1\#}U_1) +  \mass (f_{2\#} V_2) +\mass(f_{1\#} V_1) \\
&=&\mass(U_2) +\mass(U_1) +  \mass ( V_2) + \mass (V_1).
\end{eqnarray}
Taking an infimum over all $U_i$ and $V_i$ satisfying (\ref{triangle-1}), we see
that 
\be
d_{\Fm}(M_1,M_2) \le d^{Z_1}_F(\varphi_{1\#} T,\psi_{1\#} T_1) +d^{Z_2}_F(\varphi_{2\#} T,\psi_{2\#} T_2).
\ee
Taking an infimum over all metric spaces $Z_1, Z_2$ and all isometric embeddings 
$\psi_i:\bar{X}_i \to Z_i$ and $\varphi_i:\bar{X} \to Z_i$ we obtain the triangle inequality.
\end{proof}

\subsection{A Brief Review of Existing Compactness Theorems} \label{subsect-flat-dist-1}

Gromov defined the following distance between metric spaces in \cite{Gromov-metric}:

\begin{defn}[Gromov]\label{defn-GH} 
Recall that the Gromov-Hausdorff distance between two metric spaces $\left(X, d_X\right)$ and $\left(Y, d_Y\right)$
is defined as
\be \label{eqn-GH-def}
d_{GH}\left(X,Y\right) := \inf  \, d^Z_H\left(\varphi\left(X\right), \psi\left(Y\right)\right)
\ee
where $Z$ is a complete metric space, and $\varphi: X \to Z$ and $\psi:Y\to Z$ are
isometric embeddings and where the Hausdorff distance in $Z$ is defined as
\be
d_{H}^Z\left(A,B\right) = \inf\{ \epsilon>0: A \subset T_\epsilon\left(B\right) \textrm{ and } B \subset T_\epsilon\left(A\right)\}.
\ee
\end{defn}

Gromov proved that this is indeed a distance on compact metric spaces: $d_{GH}\left(X,Y\right)=0$
iff there is an isometry between $X$ and $Y$.   There are many equivalent definitions
of this distance, we choose to state this version because it inspired our definition
of the intrinsic flat distance.  Gromov also introduced the following notions:

\begin{defn}[Gromov]\label{defn-equibounded} 
A collection of metric spaces is said to be equibounded or
uniformly bounded if there is a uniform upper bound
on the diameter of the spaces.  
\end{defn}

\begin{rmrk}\label{rmrk-NXr}
 We will write $N\left(X,r\right)$ to denote the maximal
number of disjoint balls of radius $r$ in a space $X$.   Note that $X$ can always be covered
by $N\left(X,r\right)$ balls of radius $2r$.
\end{rmrk}

\begin{defn}[Gromov] \label{defn-equicompact} 
A collection of spaces is said to
be equicompact or uniformly compact if they 
have a common upper bound $N\left(r\right)$
such that $N\left(X,r\right) \le N\left(r\right)$ for all spaces $X$ in the collection.
\end{defn}

Note that Ilmanen's Example depicted in Figure~\ref{fig-hairy-sphere} is not
equicompact, as the number of balls centered on the tips approaches infinity
[Example~\ref{ex-hairy-sphere}].

Gromov's Compactness Theorem states that sequences of equibounded and equicompact
metric spaces have a Gromov-Hausdorff converging subsequence \cite{Gromov-French}.  
In fact,
Gromov proves a stronger version of this statement in 
a subsequent work, \cite{Gromov-poly}{p 65}, which 
we state here so that we may apply it:

\begin{thm}[Gromov's Compactness Theorem] \label{Thm-Gromov} 
If a sequence of compact metric spaces, $X_j$, is equibounded and equicompact, then
there is a pair of compact metric spaces, $Y \subset Z$, and a subsequence $X_{j_i}$ which
isometrically embed into $Z$: $\varphi_{j_i}: X_{j_i} \to Z$ such that 
\be
\lim_{i\to \infty} d_H^Z\left(\varphi_{j_i}\left(X_{j_i}\right), Y\right)  =0.
\ee
So $\left(Y,d_Z\right)$ is the Gromov-Hausdorff limit of the $X_{i_j}$.
\end{thm}

Gromov's proof of the stronger statement  involves a construction of a metric
on the disjoint union of the sequence of spaces.  This method of proving the Gromov compactness theorem relies on
the fact that infimum in (\ref{defn-GH}) can be estimated arbitrarily well by taking $Z$ to be a disjoint union of
the spaces and choosing a clever metric on $Z$. 

The reason we have stated this stronger version of Gromov's Compactness Theorem
is because it can be applied in combination with Ambrosio-Kirchheim's compactness
theorem to prove our first compactness theorem for integral current spaces [Theorem~\ref{GH-to-flat}].

Recall the notion of total mass [Definition~\ref{def-total-mass}].
Ambrosio Kirchheim's Compactness Theorem, which extends Federer-Fleming's Flat Norm 
Compactness Theorem, is stated in terms of weak convergence of
currents.  See Definition 3.6 in \cite{AK} which extends Federer-Fleming's notion of weak convergence:

\begin{defn}[Weak Convergence] \label{def-weak}
A sequence of integral currents $T_j \in \intcurr_m\left(Z\right)$ is said to converge weakly to
a current $T$ iff the pointwise limits satisfy
\be
\lim_{j\to \infty}  T_j\left(f, \pi_1,..., \pi_m\right) = T\left(f, \pi_1,..., \pi_m\right) 
\ee
for all bounded Lipschitz $f$ and Lipschitz $\pi_i$.
\end{defn}

\begin{rmrk} \label{rmrk-unif-to-weak}
If we suppose one has a sequence of isometric embeddings, $\varphi_i: X \to Z$,
which converge uniformly to $\varphi: X\to Z$, and $T\in \intcurr_m(X)$,
then $\varphi_{i\#}T$ converges to $\varphi_\#T$.  This can be seen by
applying properties (ii) and (iii) in the definition of a current
as follows:
\begin{eqnarray*}
\lim_{i\to\infty} \varphi_{i\#}T(f, \pi_1,..., \pi_m) 
&=&\lim_{i\to\infty} T(f\circ \varphi_i, \pi_1\circ\varphi_i,..., \pi_m\circ\varphi_i)\\
&=& T(f\circ \varphi, \pi_1\circ\varphi,..., \pi_m\circ\varphi)=\varphi_\#T(f, \pi_1,..., \pi_m).
\end{eqnarray*}
\end{rmrk}

\begin{rmrk}
If $T_j\in \intcurr_m(Z)$ has
$\mass(T_j) \to 0$, then by (\ref{eqn-mass}),
\be 
\Big| T_j(f,\pi_1,..., \pi_m) \Big| \le \mass(T_j) |f|_\infty \Lip(\pi_1) \cdots \Lip(\pi_m)\to 0,
\ee
so $T_j$ converges weakly to $0$.
\end{rmrk}

\begin{rmrk}\label{rmrk-flat-implies-weak}
Note that flat convergence implies weak convergence because $T_j \Fto T$
implies there exists $U_j,V_j$ with $\mass(U_j)+\mass(V_j)\to 0$
such that $T_j-T=U_j +\partial V_j$.  This implies that $U_j$ and
$V_j$ must converge weakly to $0$ and $\partial V_j$ must as well.
So $T_j -T\weaklyto 0$ and $T_j \weaklyto T$.
\end{rmrk}

\begin{rmrk} \label{rmrk-lower-mass}
Immediately below the definition of weak convergence \cite{AK} Defn 3.6,
Ambrosio-Kirchheim prove 
the lower semicontinuity of mass.  In particular, 
if $T_j$ converges weakly to $T$, then $\liminf_{j\to\infty} \mass(T_j) \ge \mass(T)$.  
\end{rmrk}

\begin{rmrk} It should be noted here that weak convergence as defined in Federer \cite{Federer}
is tested only with differential forms of compact support while weak convergence in
Ambrosio-Kirchheim does not require the test tuples to have compact support.  Sequences
of unit spheres in Euclidean space whose centers diverge to infinity converge weakly to
$0$ in the sense of Federer but not in the sense of Ambrosio-Kirchheim.  
\end{rmrk}

\begin{thm}[Ambrosio-Kirchheim Compactness]\label{AK-compact}
Given any complete metric space 
$Z$, a compact set $K \subset Z$ and any sequence of integral currents  $T_j \in \intcurr_m \left(Z\right)$
with a uniform upper bound on their total mass $\nmass\left(T_j\right)=\mass\left(T_j\right) +\mass\left(\partial T_j\right) \le M_0$,
such that $\set\left(T_j\right) \subset K$, there exists a subsequence, $T_{j_i}$, and a limit current $T \in \intcurr_m\left(Z\right)$
such that $T_{j_i}$ converges weakly to $T$.
\end{thm}

The key point of this theorem is that the limit current is an integral current and has a rectifiable
set with finite mass and rectifiable boundary with bounded mass.

In order to apply Ambrosio-Kirchheim's result we need a result of the second author from
\cite{Wenger-flat}[Theorem 1.4] which generalizes a theorem of Federer-Fleming relating the
weak and flat norms.  As in Federer-Fleming one needs a uniform bound on total mass
to have the relationship.  To simplify the statement of \cite{Wenger-flat}[Theorem 1.4],
we restrict the setting to Banach spaces although his result is far more general:

\begin{thm}[Wenger Flat=Weak Convergence] \label{weak=flat} 
Let $E$ be a Banach space and $m \ge 1$.  
If we assume a sequence of integral currents,
$T_j \in \intcurr_m\left(E\right)$, has a uniform upper
bound on total mass $\mass\left(T_j\right)+\mass\left(\partial T_j\right)$, then $T_j$ converges weakly to 
$T\in \intcurr_m\left(E\right)$ iff $T_j$ converges to $T$ in the flat sense.
\end{thm}

For our purposes, it suffices to have a Banach space, because we may apply Kuratowski's embedding theorem
to embed any complete metric space into a Banach space:

\begin{thm}[Kuratowski Embedding Theorem] \label{Kuratowski}
Let $Z$ be a complete metric space, and $\ell^\infty\left(Z\right)$ be the space of
bounded real valued functions on $Z$ endowed with the sup norm.  Then the
map $\iota: Z \to \ell^\infty\left(Z\right)$ defined by fixing a basepoint $z_0\in Z$
and letting $\iota\left(z\right)= d_Z\left(z_0,\cdot\right)-d_Z\left(z, \cdot\right)$ is an isometric embedding.
\end{thm}

\begin{rmrk} \label{rmrk-Kuratowski}
By the Kuratowski embedding theorem,
the infimum in (\ref{eqn-local-defn})
can be taken over Banach spaces, $Z$. 
\end{rmrk}

Combining Kuratowski's Embedding Theorem with 
Gromov and Ambrosio-Kirchheim's Compactness Theorems we immediately obtain:

\begin{thm} \label{GH-to-flat}
Given a sequence of $m$ dimensional integral current spaces $M_j=\left(X_j, d_j, T_j\right)$ such that $X_j$ are equibounded and
equicompact and such that $\nmass\left(T_j\right)$ is uniformly bounded above, then a subsequence
converges in the
Gromov-Hausdorff sense $\left(X_{j_i}, d_{j_i}\right) \to \left(Y,d_Y\right)$ and in the 
intrinsic flat sense 
$\left(X_{j_i}, d_{j_i}, T_{j_i}\right) \to \left(X,d,T\right)$
where either $\left(X,d,T\right)$ is an $m$ dimensional integral current space
with $X \subset Y$
or it is the ${\bf 0}$ current space.
\end{thm}

Note that $X$ might be a strict subset of $Y$ as seen in Example~\ref{example-not-length} depicted
in Figure~\ref{figure-not-length}. 

\begin{proof}
By Gromov's Compactness Theorem, there exists a compact space $Z$ and isometric embeddings $\varphi_j:X_j \to Z$
such that a subsequence of the $\varphi_j\left(X_j\right)$, still denoted $\varphi_j(X_j)$,  
converges in the Hausdorff sense to a closed subset, $Y' \subset Z$.  
We then apply Kuratowski's Theorem
to define isometric embeddings 
$\varphi'_j=\iota\circ \varphi_j: X_j \to \ell^\infty\left(Z\right)$.  Note that $K=\iota\left(Z\right)\subset \ell^\infty\left(Z\right)$ is compact and
\be
\spt \varphi'_{j\#}\left(T_j\right) \subset Cl\left(\varphi'_j\left(X_j\right)\right) \subset \iota\left(Z\right) = K.
\ee 
Let $Y=\iota(Y')$ with the restricted metric.  

We now apply the Ambrosio-Kirchheim Compactness Theorem to see that there exists a further
subsequence $\varphi'_{j_i \#} T_{j_i}$ converging weakly to an integral current 
$S \in \intcurr_m\left(\ell^\infty\left(Z\right)\right)$.  We claim $\spt S \subset Y$.
If not then there exists $x \in \spt S\setminus Y$, and there exists $r>0$
such that $B(x,r)\cap Y=0$.   By definition of support,
$||S||(B(x,r/2)) >0$.  By weak convergence, there is an $i$
sufficiently large that $||S_{j_i}||(B(x,r))>0$.  In particular $x \in T_{r/2}(\spt S_{j_i})$.
Taking $i\to\infty$, we see that $x\in T_r(Y)$ because $Y$ is the Hausdorff
limit of the $\spt S_{j_i}$.

Since $E=\ell^\infty\left(Z\right)$ is a Banach space and there
is a uniform upper bound on the total mass, we apply Wenger's Flat=Weak Convergence
Theorem to see that
\be
d^E_F\left(\varphi'_{j_i \#} T_{j_i}, S\right) \to 0.
\ee
We now define our limit current space $\left(X,d,T\right)$ by taking $X=\set(S)$, $d=d_E$
and $T=S$.  The identity map isometrically embeds $X$ into $E$ and takes $T$ to $S$.
Since $\set(S) \subset \spt(S) \subset Y$, we are done.
\end{proof}

We have the following immediate corollary of Theorem~\ref{GH-to-flat}:

\begin{cor} \label{lower-dim}
Given a sequence of precompact
$m$ dimensional integral current spaces, $M_j=\left(X_j, d_j, T_j\right)$,
with a uniform upper bound on their total mass  such that $X_j$ converge in the
Gromov-Hausdorff sense to a compact limit space,$Y$,
of lower Hausdorff dimension, $dim_\mathcal{H}(Y)<m$, then $M_j$
converges in the intrinsic flat sense 
to the $\bf{0}$ current space because the zero current is the only $m$ dimensional
integral current whose canonical set has Hausdorff dimension less than $m$.  
\end{cor}

\begin{rmrk} \label{vol-to-zero}
Note that by Remark~\ref{finite} any collapsing sequence of Riemannian manifolds, $M_j^m$ such that
$\vol\left(M_j\right) \to 0$, converges in the intrinsic flat sense to the 
${\bf{0}}$ integral current
space.  Thus even when the Gromov-Hausdorff limit is higher
dimensional as in Example~\ref{ex-jungle-gym} 
the intrinsic flat limit may collapse to the ${\bf 0}$ current space.
\end{rmrk}

\subsection{The Infimum is Achieved} \label{subsect-flat-dist-2} 

In this subsection we prove the infimum in the definition of the intrinsic flat
distance (\ref{eqn-local-defn}) is achieved for precompact integral current
spaces.  

\begin{thm}\label{inf-dist-attained}\label{achieved}
Given a pair of precompact integral current spaces, $M=(X,d,T)$
and $M'=(X',d',T')$, there exists a compact metric space, $Z$,
integral
currents $U\in\intcurr_m\left(Z\right)$ and  $V\in\intcurr_{m+1}\left(Z\right)$,
and isometric embeddings
$\varphi : \bar{X}\to Z$ and $\varphi':\bar{X}' \to Z$
with
\begin{equation} \label{eqn-Federer-Flat-3}
\varphi_\# T- \varphi'_\# T'=U+\bdry V
\end{equation}
such that
\begin{equation}\label{eqn-local-defn-2}
d_{\Fm}\left(M,M'\right)=\mass\left(U\right)+\mass\left(V\right).
\end{equation}
In fact, we can take $Z=\spt\left(U\right)\cup \spt(V)$.
\end{thm}

This theorem also holds for $M'=\bf{0}$, where we just take $T'=0$ and do not concern
ourselves with embedding $X'$ into $Z$.

In our proof of this theorem, we use the notion of an injective metric space and
Isbell's theorem regarding the existence of an injective envelope of a metric space
\cite{Isbell}:

\begin{defn}
A metric space $W$ is said to be injective iff it has the following property:
given any pair of metric spaces, $Y\subset Z$, and
any $1$ Lipschitz function, $f: Y\subset Z \to W$, we can extend $f$ to
a $1$ Lipschitz function $\bar{f}: Z \to W$.
\end{defn}

\begin{thm} [Isbell] 
Given any metric space $X$, there is a smallest injective space, 
which contains $X$, called the injective envelope.
Furthermore, when $X$ is compact, its injective envelope
is compact as well.
\end{thm}


We now prove Theorem~\ref{inf-dist-attained}.
 
\begin{proof}
Let $Z_n$ and $U_n\in\intcurr_m\left(Z_n\right)$ and $V_n\in\intcurr_{m+1}\left(Z_n\right)$ 
approach the infimum in the definition of the flat distance
between current spaces (\ref{eqn-local-defn}).  That is there exists
isometric embeddings $\varphi_n: \bar{X}\to Z_n$ and $\varphi'_n: \bar{X}'\to Z_n$ 
such that
\be
\varphi_{n\#}T  - \varphi'_{n\#}T' = U_n +\bdry V_n
\ee
where
 \begin{equation} 
  \mass\left(U_n\right)+\mass\left(V_n\right) \leq d_{\Fm}\left(M,M'\right) + \frac{1}{n}.
 \end{equation}
 
 We would like to apply Ambrosio-Kirchheim's Compactness Theorem, so we
 need to find a common compact metric space, $Z$, and push $U_n$ and $V_n$
 into this common space and then take their limits to find $U$ and $V$.   We will
 build $Z$ in a few stages using Gromov's Compactness Theorem and Isbell's
 Theorem.   The $Z_n$ we have right now need not be equicompact or
 equibounded.
 
 We first claim that $\varphi_n, \varphi_n'$ and $Z_n$ may be chosen so that
  \be
 \diam(Z_n) \le 3 \diam(\varphi_n(\bar{X}))+ 3\diam(\varphi'_n(\bar{X}'))
 =3 \diam(X) + 3 \diam(X').
 \ee
 If not, then there exists $p_n\in \varphi_n(\bar{X})$ and $p_n'\in \varphi'_n(\bar{X}')$
 such that the closed balls
 \be
 \bar{B}(p_n,2\diam(X)) \cap \bar{B}(p_n',2\diam(X')) =\emptyset.
 \ee
 Taking $A_n=Z_n\setminus (\bar{B}(p_n,2\diam(X)) \cup \bar{B}(p'_n,2\diam(X')))$,
 we would then define 
 $
 Z_n':=Z_n / A_n 
 $
 with the quotient metric 
 \be
 d_{Z_n'}\left([z_1],[z_2]\right) := \inf \left\{ d_{Z_n}(x_1,a_1)+d_{Z_n}(a_2,x_2): \,\, 
                                                        x_i \in [z_i], a_i \in A_n\right\}. 
 \ee
Then $Z_n'$ has the required bound on diameter and we need only construct
the embeddings.
 
 Let $p: Z_n \to Z_n/A$ be the projection.  Then $p$ is an isometric 
 embedding when restricted to 
  $\varphi_n(X)\subset \bar{B}(p_n,\diam(X))$  or to 
  $\varphi_n(X')\subset \bar{B}(p'_n,\diam(X'))$.   Thus 
  $p\circ\varphi_n:\bar{X} \to Z_n/A$ and 
  $p\circ\varphi'_n:\bar{X}' \to Z_n/A$ are isometric embeddings.
 Furthermore $p$ is $1$-Lipschitz on $Z_n$, so
 \be
p_\#\varphi_{n\#}T  -p_\# \varphi'_{n\#}T' = p_\#U_n +\bdry p_\#V_n
\ee
and, by Lemma~\ref{lem-push-mass},
 \begin{equation} 
  \mass\left(p_\#U_n\right)+\mass\left(p_\#V_n\right) \leq
\mass\left(U_n\right)+\mass\left(V_n\right) 
\end{equation}
 So our first claim is proven.
 
Now let $Y_n:=\varphi_n(\bar{X}) \cup \varphi'_n(\bar{X}')\subset Z_n$ with the restricted
 metric from $Z_n$.  By our first claim, the diameters of the $Y_n$ are uniformly bounded.
 In fact the $Y_n$ are equicompact because the number of disjoint
 balls of radius $r$ may easily be estimated:
 \be
 N(Y_n,r) \le N(\varphi_n(X),r) + N(\varphi_n'(X'),r)=N(X,r) + N(X',r).
 \ee
 Thus, by Gromov's Compactness Theorem, there exists a compact metric
 space, $Z'$, and isometric embeddings $\psi_n: Y_n \to Z'$.
 
 Recall that $U_n\in\intcurr_m\left(Z_n\right)$ and $V_n\in\intcurr_{m+1}\left(Z_n\right)$,
 so we need to extend $\psi_n$ to $Z_n$ in order to push forward
 these currents into the common compact metric space, $Z$, and 
 take their limits.
 
 By Isbell's Theorem, we may take $Z$ to be the injective envelope of $Z'$.   
 Since $Z$ is injective, we can extend the $1$-Lipschitz
 maps, $\psi_n$, to $1$-Lipschitz maps, $\bar{\psi}_n: Z_n \to Z$.
 So now we have a common compact metric space, $Z$,
 isometric embeddings $\bar{\psi}_n\circ \varphi_n: \bar{X} \to Z$
 and $\bar{\psi}_n\circ \varphi'_n:\bar{X}'\to Z$,
 such that
 \be \label{four-currents}
\bar{\psi}_{n\#} \varphi_{n\#}T  - \bar{\psi}_{n\#}\varphi'_{n\#}T' = 
\bar{\psi}_{n\#}U_n +\bdry \bar{\psi}_{n\#}V_n
\ee
where
 \begin{equation} 
  \mass\left(\bar{\psi}_{n\#}U_n\right)+\mass\left(\bar{\psi}_{n\#}V_n\right) \leq 
  d_{\Fm}\left(M,M'\right) + \frac{1}{n}.
 \end{equation}

By Arzela-Ascoli's Theorem, after taking a subsequence, the isometric embeddings
$\bar{\psi}\circ\varphi_n: X\to Z$ and $\bar{\psi}\circ\varphi'_n: X'\to Z$ 
converge uniformly to isometric embeddings
$\varphi: X \to Z$ and $\varphi': X'\to Z$.  As in Remark~\ref{rmrk-unif-to-weak},
we then have weak convergence:
\be
\bar{\psi}_{n\#} \varphi_{n\#}T   \weaklyto \varphi_\#T \textrm{ and }
\bar{\psi}_{n\#}\varphi'_{n\#}T' \weaklyto \varphi'_\# T'.
\ee 

By Ambrosio-Kirchheim's Compactness Theorem, after possibly taking
a further subsequence, there exists $U\in\intcurr_m\left(Z\right), 
V\in\intcurr_{m+1}\left(Z\right)$ such that
\be 
\bar{\psi}_\#U_n \weaklyto U \textrm{ and } \bar{\psi}_\#V_n \weaklyto V.
\ee
In particular, $\varphi_\#T - \varphi'_\# T'= U- \partial V$.

By the lower semicontinuity of mass (c.f. Remark~\ref{rmrk-lower-mass}), 
\begin{equation} 
  \mass(U)+\mass(V) \leq 
  d_{\Fm}\left(M,M'\right) + \frac{1}{n} \qquad \forall n \in \N
 \end{equation}
and we are done.
\end{proof}



\subsection{Current Preserving Isometries} \label{subsect-flat-dist-3}

We can now prove that the intrinsic flat distance is a distance
on the space of precompact oriented Riemannian manifolds with boundary
and, more generally, on precompact integral current spaces in $\mathcal{M}_0^m$.

\begin{defn}
Given $M,N \in \mathcal{M}^m$, an isometry $f: X_M \to X_N$ is called
a current preserving isometry between $M$ and $N$, if its
extension $\bar{f}: \bar{X}_M \to \bar{X}_N$ pushes forward the
current structure on $M$ to the current structure on $N$:
$
\bar{f}_\# T_M= T_N
$
\end{defn}

When $M$ and $N$ are oriented Riemannian manifolds 
or other Lipschitz manifolds with the standard
current structures as described in Remark~\ref{Lip-mani-charts}
then orientation preserving isometries are exactly current preserving 
isometries.  {See Remark~\ref{rmrk-biLip-matching}.}

\begin{thm} \label{zero-mani}
If $M,N$ are precompact integral current spaces such that $d_{\Fm}\left(M,N\right)=0$
then there is a current preserving isometry from $M$ to $N$.
Thus $d_{\mathcal{F}}$ is a distance on $\mathcal{M}^m_0$.
\end{thm}

It should be noted that a pair of precompact metric
spaces, $X,Y$ such that $d_{GH}(X,Y)=0$ need not be isometric (e.g.  
the Gromov-Hausdorff distance between a Riemannian manifold, and
the same manifold with one point removed is $0$).  However, if $X$ and $Y$
are compact, then Gromov proved
$d_{GH}(X,Y)=0$ implies they are isometric \cite{Gromov-metric}.

While we do not require that our spaces be complete, the definition
of an integral current space requires that the spaces be completely
settled [Defn~\ref{def-settled}] so that $X=\set (T)$
[Defn~\ref{defn-integral-current-space}].  This is as essential to
the proof of Theorem~\ref{zero-mani} as the compactness is 
essential in Gromov's theorem.
Precompactness on the other hand, is not a necessary condition.
Theorem~\ref{zero-mani} can be extended to noncompact 
integral current spaces applying Theorem 6.1 in the second author's
compactness paper \cite{Wenger-compactness}.

\begin{proof}
By Theorem~\ref{inf-dist-attained} and the fact that an integral current
has zero mass iff it is $0$,
we know there exists a compact
space $Z$ and isometric embeddings,
$\varphi : \left(\bar{X}_M, d_{\bar{X}}\right) \to \left(Z,d\right)$ and $\psi: \left(\bar{X}_N , d_{\bar{X}_N}\right)\to \left(Z,d\right)$,
with
\begin{equation} \label{eqn-zero-mani-1}
{\varphi}_\# T_M- {\psi}_\# T_N=0 \in \intcurr_m\left(Z\right).
\end{equation}
Thus
\be
\set\left({\varphi}_\# T_M\right) =\set \left({\psi}_\# T_N\right).
\ee
By Lemma~\ref{lemma-isom-to-set}, we know $\varphi: X_M \to \set\left({\varphi}_\# T_M\right)$
and $\psi: X_N \to \set\left({\psi}_\# T_N\right)$ are isometries.

We define our isometry $f: X_M \to X_N$ to be $f=\psi^{-1}\circ \varphi$.
Then $\bar{f}: \bar{X}_M \to \bar{X}_N$, pushes $T_M\in \intcurr_m\left(\bar{X}_M\right)$
to $\bar{f}_\# T_M \in \intcurr_m\left(\bar{X}_N\right)$,
so that with (\ref{eqn-zero-mani-1}) we have,
\be
{\psi}_\# \bar{f}_\# T_M= {\varphi}_\# T_M = {\psi}_\# T_N.
\ee
Since ${\psi}_\# \,\left(f_\#T_M - T_N\right) =0 \in \intcurr_m\left(Z\right)$ and ${\psi}$
is an isometry, we have
 $f_\#T_M - T_N=0 \in \intcurr_m\left(\bar{X}_N\right)$.
\end{proof}

The following is an immediate consequence of Theorem~\ref{zero-mani}: 

\begin{cor}
If $M^m$ and $N^m$ are precompact oriented Riemannian manifolds with finite volume,
then $d_{\mathcal{F}}(M^m,N^m)=0$ iff there is an orientation preserving isometry,
$\psi: M^m \to N^m$.  Thus $d_{\mathcal{F}}$ is a distance on the space of precompact
oriented Riemannian manifolds with finite volume.
\end{cor}

\begin{rmrk}
Initially we were hoping to prove that if the intrinsic flat distance between 
two Riemannian manifolds
is zero then the manifolds are isometric.  This is false
unless the manifold has an orientation reversing isometry
as we prove in Theorem~\ref{zero-mani}.  We thought we
might use a $\Z_2$ notion of integral currents to avoid the issue of
orientation.  However, at the time there was no such theory, so we
settled on this version of the theorem with this notion of intrinsic
flat distance.  Very recently Ambrosio-Katz \cite{Ambrosio-Katz}
and Ambrosio-Wenger \cite{Ambrosio-Wenger} 
completed work covering this theory and one expects this will
lead to interesting new ideas.
Alternatively one could try to use
the even more recent theory of DePauw-Hardt \cite{DH-chains}.
\end{rmrk}


\section{{\bf Sequences of Integral Current Spaces}} 

In this section we describe the properties of sequences of integral
current spaces which converge in the intrinsic flat sense.

In Subsection~\ref{subsectseq} we take a
Cauchy or converging sequence of precompact integral current
spaces and construct a common metric space, $Z$, into which the entire sequence
embeds [Theorem~\ref{cauchy} and Theorem~\ref{converge}].  Note that
$Z$ need not be compact even when the spaces are.  Relevant examples
are given and an open question appears in Remark~\ref{rmrk-cauchy}.

In Subsection~\ref{subsect-properties} we remark on the properties
of converging sequences of integral current spaces.
We prove the lower semicontinuity
of mass [Theorem~\ref{mass-drops}] which is a direct consequence of
Ambrosio-Kirchheim \cite{AK}.  We remark on the continuity of filling
volume which follows from work of the second author \cite{Wenger-flat}.

In Subsection~\ref{subsect-SorWen1} we state consequences of
the authors' first paper \cite{SorWen1} concerning limits of sequences
of Riemannian manifolds with contractibility conditions as in work
of Greene-Petersen \cite{Greene-Petersen}.  We discuss how to
avoid the kind of cancellation depicted in Example~\ref{example-cancels}
depicted in Figure~\ref{figure-cancels} using Gromov's filling volume \cite{Gromov-filling}.

In Subsection~\ref{subsect-Ricci} we discuss noncollapsing sequences
of manifolds with nonnegative Ricci or positive scalar curvature particularly
Theorem~\ref{thm-ricci} and Conjecture~\ref{conj-scalar}
which appear in our first paper \cite{SorWen1}.

In Subsection~\ref{subsect-Wenger-compactness} we state the second
author's compactness theorem [Theorem~\ref{thm-Wenger-compactness}]
which is proven in \cite{Wenger-compactness}.  We then prove
Theorem~\ref{Cauchy-to-Converge} which provides a common metric
space $Z$ for a Cauchy sequence bounded as in the compactness theorem.
In particular, any Cauchy sequence of integral current spaces with a uniform
upper bound on diameter and total mass converges to an integral current
space.

\subsection{Embeddings into a Common Metric Space}\label{subsectseq}\label{subsect-flat-dist-4}

In this subsection we prove Theorems~\ref{cauchy},~\ref{converge}
and~\ref{convergeto0} which describe how
Cauchy and converging
sequences of integral current spaces, $M_i$, can be isometrically embedded into a common 
separable complete metric space $Z$ as a flat Cauchy or converging sequence. 
These theorems are essential to understanding sequences of manifolds which
do not have Gromov-Hausdorff limits.
We will also apply them
to prove Theorem~\ref{Cauchy-to-Converge}.  

\begin{thm}\label{cauchy}
Given an intrinsic  flat  Cauchy sequence of
integral current spaces, $M_j=\left(X_j, d_j, T_j\right) \in \mathcal{M}^m$,
there exists a separable complete
metric space
$Z$, and a sequence
of isometric embeddings $\varphi_j: X_j \to Z$
such that $\varphi_{j\#}T_j$ is a flat Cauchy
sequence of integral currents in $Z$.
\end{thm}

The classic example of a Cauchy sequence of integral currents converging to
Gabriel's horn shows that a uniform upper bound on mass is required to
have a limit space which is an integral current space [Example~\ref{ex-Gabriel's-horn}].
So the Cauchy sequence in this theorem need not converge without an
additional assumption on total mass.
In Example~\ref{ex-unbounded} we see that even with the uniform bound on
total mass, the sequence may have a limit which is unbounded.
In Example~\ref{ex-many-tips} depicted in Figure~\ref{fig-many-tips} we
see that even with a uniform bound on total mass and diameter, the limit
space need not be precompact.  See also Remark~\ref{rmrk-cauchy} and Theorem~\ref{Cauchy-to-Converge}.

If we assume that the Cauchy sequence of integral current spaces converges
to a given integral current space, than we can embed the entire sequence
including the limit into a common metric space $Z$:

\begin{thm}\label{converge}\label{converges}
If a sequence of 
integral current spaces,
$M_{j}=\left(X_j, d_j, T_j\right)$, converges  
to an 
integral current space,
 $M_0=\left(X_0,d_0,T_0\right)$
 in the intrinsic flat sense , then
there is a separable
complete metric space, $Z$, and isometric embeddings  $\varphi_j: X_j \to Z$ such that
$\varphi_{j\#}T_j$ flat converges to $\varphi_{0\#} T_0$ in $Z$
and thus converge weakly as well.
\end{thm}

Note that one cannot construct a compact $Z$ as Gromov did in \cite{Gromov-groups}
even when one knows the sequence converges in the intrinsic flat sense to a compact
limit space and that the sequence has a uniform bound on total mass.   The sequence of hairy spheres
in Example~\ref{ex-hairy-sphere} converges to a sphere in the flat norm but cannot be isometrically
embedded into a common compact space because the sequence is not equicompact.

The special case of Theorem~\ref{converges} where $M_j$ converges to the
$\bf{0}$ space can have prescribed pointed isometries:

\begin{thm}\label{convergeto0}
If a sequence of 
integral current spaces
$M_{j}=\left(X_j, d_j, T_j\right)$ converges 
in the intrinsic flat sense to the 
zero integral current space, $\bf{0}$, then
we may choose points $p_j\in X_j$ and a
separable
complete metric space, $Z$, and isometric embeddings  
$\varphi_j: \bar{X}_j \to Z$ such that
$\varphi_j(p_j)=z_0\in Z$ and
$\varphi_{j\#}T_j$ flat converges to $0$ in $Z$ and thus converges weakly as well.
\end{thm}

We prove this theorem first since it is the simplest.

\begin{proof}
By the definition of the flat distance, we know there exists a complete 
metric space $Z_j$ and $U_j\in \intcurr_m(Z_j)$ and $V_j\in \intcurr_{m+1}(Z_j)$
and an isometry $\varphi_j: X_j\to Z_j$ such that
$\varphi_{j\#}T_j=U_j +\partial V_j$ and 
\be
d_\mathcal{F}(M_j, {\bf{0}}) \le\mass(U_j)+\mass(V_j) \to 0.
\ee
We may choose $Z_j=\spt U_j \cup \spt V_j$, so it is separable.

We then create a common complete separable metric space $Z$ by gluing
all the $Z_j$ together at the common point $\varphi_j(p_j)$:
\be
Z=Z_1 \disjointunion Z_2 \disjointunion\cdots
\ee
where $d_Z(z_1,z_2)= d_{Z_i}(z_1,z_2)$ when there exists an $i$ with $z_1,z_2\in Z_i$
and 
\be
d_Z(z_i,z_j)=d_{Z_i}(z_i,\varphi_i(p_i))+d_{Z_j}(z_j,\varphi_j(p_j)).
\ee
We then identify all the $\varphi_i(p_i)=\varphi_j(p_j)\in Z$ so that this is a metric.
Since mass is preserved under isometric embeddings, we have
$d^Z_F(\varphi_{j\#}T_j,0) \le \mass(U_j)+\mass(V_j) \to 0$.
\end{proof}

To prove Theorems~\ref{cauchy} and~\ref{converge}, we need to glue together
our spaces $Z$ in a much more complicated way.  So we first prove the
following two lemmas and then prove the theorems.
 We close the section with Remark~\ref{rmrk-cauchy}
which discusses a related open problem.

Recall the well known gluing lemma [Lemma~\ref{gluing-lemma}]
that we applied to prove the triangle inequality in Subsection~\ref{subsect-triangle}.
One may apply this gluing of metric spaces countably many times, to glue together countably
many distinct metric spaces: 

\begin{lem} \label{lemma-tree-graph} \label{lemma-shrub}
We are given 
a connected tree with countable vertices $\{V_i : i \in A \subset \N\}$ and edges 
$\{E_{i,j} : (i,j)\in B \}$ where $B\subset \{(i,j): \, i<j, i,j \in A\}$,
and a corresponding countable collection of metric spaces 
$\{X_i: i \in A\}$ and $\{ Z_{i,j} : \, \left(i,j\right) \in B\}$ 
and isometric embeddings 
\be
\varphi_{i, (i,j)}: X_i \to Z_{i,j} \textrm{ and }
\varphi_{j, (i,j)}: X_j \to Z_{i,j} \qquad \forall (i,j)\in B.
\ee
Then there is a unique metric space $Z$ defined by gluing the $Z_{i,j}$ along the isometric
images of the $X_i$.
In particular there exists isometric embeddings $f_{i,j}: Z_{i,j} \to Z$ for all
$(i,j)\in B$ such that for all $(i,j),(j,k) \in B$ we have
isometric embeddings
\be \label{eqn-f-i-j}
f_{i,j} \circ \varphi_{j,(i,j)} = f_{j,k} \circ \varphi_{j, (j,k)}: \, X_{j} \to Z.
\ee
If $Z_{i,j}$ are separable then so is $Z$.
\end{lem}

\begin{proof}
Let $Z$ be the disjoint union of the $Z_{i,j}$.  We define a quasimetric on $Z$ and then identify the
images of the $X_i$ so that the quasimetric becomes a metric.  Let $z,z'\in Z$, so
each lies in one of the $Z_{i,j}$ and thus has a corresponding edge $E\left(z\right), E\left(z'\right)\in \{E_{i,j}: (i,j)\in B\}$.

If $E\left(z\right)=E\left(z'\right)$ then they lie in the same $Z_{i,j}$ and we
let $d_Z\left(z,z'\right):=d_{Z_{i,j}}\left(z,z'\right)$ which we denote as $d_{i,j}$ to avoid excessive subscripts below.

If $E(z)\neq E(z')$, then because the graph is a connected tree there is a unique sequence of distinct edges
$\{E_{i_0,i_1}, E_{i_1,i_2}, ..., E_{i_{n},i_{n+1}}\}$ where $E\left(z\right)=E_{i_0,i_1}$ and $E\left(z'\right)=E_{i_n,i_{n+1}}$.
We define   
\begin{equation*}
\begin{split}
d_Z\left(z,z'\right) & = \inf \Bigg\{ \,\, d_{i_0,i_1}\left(z, \varphi_{i_1, (i_0,i_1)}\left(y_1\right)\right) \\
  & \qquad \qquad
       + \sum_{j=1}^{n-1} d_{i_j,i_{j+1}}\left(\varphi_{i_j, (i_j,i_{j+1})}\left(y_j\right), 
       \varphi_{i_{j+1}, (i_j, i_{j+1}) }\left(y_{j+1}\right) \right)\\
  & \qquad \qquad
   + \, d_{i_n, i_{n+1}}\left(\varphi_{i_n, (i_n,i_{n+1})}\left(y_n\right), z' \right)
   \, :\,\left(y_1,..., y_n\right)\in X_{i_1}\times \cdots \times X_{i_n}  \Bigg\}.
\end{split}
\end{equation*}
One may then easily verify the triangle inequality $d_Z\left(a,b\right)+ d_Z\left(b,c\right) \ge d_Z\left(a,c\right)$ by breaking into cases
regarding the location of $E\left(b\right)$ relative to $E\left(a\right)$ and $E\left(c\right)$.
Finally we identify points $z$ and $z'$ such that $d_Z\left(z,z'\right)=0$.
\end{proof}

We can now prove Theorem~\ref{cauchy}:

\begin{proof}
Recall that we have a Cauchy sequence of current spaces, so for all $\epsilon>0$, there exists $N_\epsilon \in \N$
such that 
\be
r_{i,j}= d_{\Fm}\left(M_i,M_j\right) < \epsilon \qquad \forall i,j \ge N_\epsilon.
\ee

By the definition of the intrinsic flat distance between $M_i$ and $M_j$ in 
(\ref{eqn-local-defn}),
 there
exist metric spaces $Z_{i,j}$ and isometric embeddings
$\varphi_{i, (i,j)}: \bar{X}_i\to Z_{i,j}$ and $\varphi_{j,(i,j)}: \bar{X}_j\to Z_{i,j}$
and integral currents
$U_{i,j} \in \intcurr_m\left(Z_{i,j}\right)$ and $V_{i,j} \in \intcurr_{m+1}\left(Z_{i,j}\right)$
with
\be \label{cauchy-1}
\varphi_{i, (i,j)\#}T_i-\varphi_{j, (i,j)\#}T_j= 
U_{i,j} + \partial V_{i,j} \in \intcurr_m\left(Z_{i,j}\right)
\ee
such that
\be \label{cauchy-exact-1}
r_{i,j}:= d_{\Fm}\left(M_i,M_j\right) =
d_F^{Z_{i,j}}\left(\varphi_{i,(i,j)\#}T_i, \varphi_{j,(i,j)\#}T_j\right)\le \mass\left(U_{i,j}\right)+\mass\left(V_{i,j}\right) \le 3r_{i,j}/2.
\ee
We choose $Z_{i,j}=\spt U_j \cup \spt V_j$ and so it is separable.

Since the sequence is Cauchy, we know 
there exists a subsequence $j_k \in \Bbb{N}$ such that 
$j_1=1$ and when $k\ge 2$ we have $r_{j_k,i}\le 1/2^k$
$\forall i \ge j_k$. In particular $r_{j_k,j_{k+1}}\le 1/2^k$ when $k\ge 2$.
We call this special subsequence, a {\em geometric subsequence}.

We now apply Lemma~\ref{lemma-tree-graph} to the graph whose
vertices are $\{V_i : i \in A = \N \}$ and edges 
$\{E_{i,j} : \left(i,j\right) \in B \subset \N\times \N \}$ where
\be
B=\{\left(j_k,j_{k+1}\right): k \in \N\} \cup \{ \left(j_k,i\right) : i=j_k, ..., j_{k+1} -1\}.
\ee
Intuitively this is a tree whose trunk is the geometric subsequence and whose branches consist
of single edges attached to the nearest vertex on the trunk.

As a result we have a complete metric space $Z$  and isometric embeddings 
$f_{i,j}:Z_{i,j} \to Z$ such that
\be \label{cauchy2}
f_{i,j} \circ \varphi_{j,(i,j)} = f_{j,i'} \circ \varphi_{j, (j,i')}: \, X_{j} \to Z
\ee
are isometric embeddings for all $(i,j), (j,i') \in B$.  In particular
each current space $M_j$ has been mapped to a unique current in $Z$:
\be \label{cauchy-4}
T'_j:= f_{i,j\#} \varphi_{j,(i,j)\#} T_j = f_{j,i'\#} \varphi_{j, (j,i')\#} T_j
\in \intcurr_m\left(Z\right)
\ee
So $f_{i,j} \circ \varphi_{j,(i,j)}$
is a current preserving isometry from
$M_j=\left(X_j, d_j, T_j\right)$ 
to $\left(\set (T'_j), d_Z, T'_j\right)$.

Applying (\ref{cauchy-1}), we have for any $\left(i,j\right)\in B$:
\be
T'_i-T'_j = f_{i,j \#}\varphi_{i, (i,j)\#}T_i-f_{i,j\#}\varphi_{j, (i,j)\#}T_j= f_{i,j\#}U_{i,j} + \partial f_{i,j\#}V_{i,j} \in \intcurr_m\left(Z\right).
\ee
Since mass is conserved under isometries (c.f. Lemma~\ref{lem-push-mass}) we have 
\be \label{add-error-1}
d^Z_F\left(T'_i, T'_j\right) \le \mass\left(f_{i,j\#}U_{i,j}\right) + \mass\left( f_{i,j\#}V_{i,j}\right) = \mass\left(U_{i,j}\right) + \mass \left(V_{i,j}\right) = 3r_{i,j}/2.
\ee
In particular by our choice of $B$ in (\ref{cauchy2}), we have for the geometric subsequence:
\be
d^Z_F\left(T'_{j_k}, T'_{j_{k+1}}\right) \le 3/2^k \,\,\forall k\ge 2.
\ee
For $i,i' \ge j_2$ we have 
$k,k' \ge 2$ respectively such that $\left(i,j_k\right), \left(i', j_{k'}\right) \in B$ such that
\be
d^Z_F\left(T'_{j_k}, T'_{i}\right) \le 3/2^k \textrm{ and }
d^Z_F\left(T'_{j_{k'}}, T'_{i'}\right) \le 3/2^{k'}.
\ee
So we have
\begin{eqnarray}
d^Z_F\left(T'_i, T'_{i'}\right) &\le& d^Z_F\left(T'_{j_k}, T'_{i}\right) + \sum_{h=k}^{k'-1}  d^Z_F\left(T'_{j_h}, T'_{j_{h+1}}\right)+ d^Z_F\left(T'_{j_{k'}}, T'_{i'}\right) \\
&\le& 3/2^k + \left(3/2^k + 3/2^{k+1} + \cdots +3/2^{k'}\right) < 9/2^k.
\end{eqnarray}
and thus our sequence of integral current spaces has been 
mapped into a Cauchy sequence of integral currents.
\end{proof}


We now prove Theorem~\ref{converge}.  Since we have
already proven Theorem~\ref{convergeto0}, we will assume
we have a nonzero limit in this proof:

\begin{proof}
As in the proof of Theorem~\ref{cauchy},  we take a geometrically converging 
subsequence of the converging sequence of current spaces.  
This time we apply Lemma~\ref{lemma-tree-graph} to the tree whose
vertices are $\{V_i : i \in A = 0 \cup \N\}$ and edges $\{E_{i,j} : \left(i,j\right) \in B \subset \N\times \N \}$ where
\be
B=\{\left(j_k,0\right): k \in \N\} \cup \{ \left(j_k,i\right) : i=j_k, ..., j_{k+1} -1\}.
\ee
so that all the terms in the geometric subsequence will be directly attached to the limit, and everything else
will be attached to the subsequence as before.  As in (\ref{cauchy-4}) we obtain unique currents
$T_j' \in \intcurr_m\left(Z\right)$ such that $\left(\set (T'_j), d_Z, T'_j\right)$ has a current preserving isometry with $\left(X_j, d_j, T_j\right)$.  This time our currents flat converge, because for any $i \in [j_k, j_{k+1}-1]$ we have
\be
d^Z_F\left(T'_i, T'_{0}\right) \le d^Z_F\left(T'_{j_k}, T'_{0}\right) + d^Z_F\left(T'_i, T'_{j_{k}}\right) \le 3/2^k + 3/2^k.
\ee
Weak convergence then follows by Remark~\ref{rmrk-flat-implies-weak}.
\end{proof}

\begin{rmrk} \label{rmrk-cauchy} 
We do not know if the sequence $\varphi_{j\#}T_j$ in Theorem~\ref{cauchy}
when given a uniform bound on total mass
converges in the flat sense to an integral current in $Z$.  Without a uniform
bound on total mass it is possible there is no limit integral current space [Example~\ref{ex-Gabriel's-horn}].

It is an open question
whether flat Cauchy sequences with uniform upper bounds on total
mass have flat converging subsequences which converge to an integral current
in the sense of Ambrosio-Kirchheim.  In Federer-Fleming, one needs to add
a diameter bound because integral currents in Federer-Fleming have compact support.  
In Ambrosio-Kirchheim compactness is never assumed so an unbounded limit
like the one in Example~\ref{ex-unbounded}  is not a counter example here.

In Theorem~\ref{Cauchy-to-Converge} we prove 
that adding a uniform bound
on diameter as well as the bound on total mass,
we can find a common metric space $Z$ 
where where $\varphi_{j\#}T_j$  do converge.  The metric space $Z$ in that theorem
may not be the metric space constructed in Theorem~\ref{cauchy}.   
To prove that theorem we need Theorem~\ref{converge} as well as
the second author's compactness theorem, Theorem~\ref{thm-Wenger-compactness}.  
It would be of interest to eliminate the bound on diameter or find a counter example.
\end{rmrk}

\subsection{Properties of Intrinsic Flat Convergence} \label{subsect-properties}
As a consequence of Theorems~\ref{converge} and~\ref{convergeto0}
and Kuratowski's Embedding Theorem,
we may now observe that sequences of integral current spaces that
converge in the intrinsic flat sense have all the same properties
Ambrosio-Kirchheim have proven for 
sequences of integral currents that converge weakly in a Banach space.
Most importantly, one has the lower semicontinuity of mass.  Applying work
of the second author in \cite{Wenger-flat} [Theorem 1.4], one also observes that one has 
continuity of the filling volume.  Here we only give the details on lower
semicontinuity of mass and leave it to the reader to extend the ideas to
other properties of integral currents.

\begin{thm}\label{mass-drops}  
If a sequence of 
integral current spaces
$M_{j}=\left(X_j, d_j, T_j\right)$ converges in the intrinsic flat sense
to 
$M_0=\left(X_0,d_0,T_0\right)$ 
then $\partial M_j$ converges to $\partial M_0$ in the intrinsic flat sense. 
\be
\liminf_{j\to\infty} \mass\left(M_j\right) \ge \mass \left(M_0\right)
\textrm{ and }
\liminf_{j\to\infty} \mass\left(\partial M_j\right) \ge \mass \left(\partial M_0\right).
\ee
\end{thm}

In Example~\ref{example-cancels} depicted in Figure~\ref{figure-cancels}
we see that the mass of the limit
space may be $0$ despite a uniform lower bound on the mass of the sequence.  

\begin{proof}
First we isometrically embed the converging sequence into a common metric
space, $Z$, applying Theorem~\ref{converge} and Theorem~\ref{convergeto0}:
$\varphi_j: \bar{X}_j \to Z$ such that $\varphi_{j\#}T_j$ converges in the
flat sense in $Z$ to $\varphi_{0\#}T_0$.  Note that
$$
d^Z_F(\partial \varphi_{j\#} T_j,\partial \varphi_{0\#} T_0) \le
d^Z_F(\varphi_{j\#} T_j,\varphi_{0\#} T_0) \to 0.
$$
By the definition of $\partial M=(\set (\partial T), d, \partial T)$
and the fact that $\partial \varphi_{j\#} T=\varphi_{j\#}\partial T$,
we have
\be
d_{\mathcal{F}}(\partial M_j, \partial M_0) \le d^Z_F(\varphi_{j\#}\partial T_j,
\varphi_{0\#}T_0)\to 0.
\ee

Immediately below the definition of weak convergence of currents in a metric
space $Z$ in \cite{AK}[Defn 3.6], Ambrosio
Kirchheim remark that the mapping $T \mapsto ||T||(A)$ is lower semicontinuous
with respect to weak convergence for any open set $A \subset Z$.   
Since $\varphi_{j\#} T_j$ converge weakly to $\varphi_{0\#}T_0$, we may take 
$A=Z$ and apply Lemma~\ref{lem-push-mass}, to see that
\be
\liminf_{j\to\infty} \mass(M_j)=\liminf_{j\to\infty}\mass(\varphi_{j\#}T_j)
\ge \mass(\varphi_{0\#} T_0) = \mass(M_0).
\ee
The same may be done to the boundaries to conclude
that 
$$
\liminf_{j\to\infty} \mass\left(\partial M_j\right) \ge \mass \left(\partial M_0\right).
$$
\end{proof}

\begin{rmrk}
Note that there are also local versions of the lower semicontinuity of mass
which can be seen by taking $A$ in the proof above to be a ball $B_{\varphi_0(x_0)}(r)$.  
These local versions require an application of Ambrosio-Kirchheim's
Slicing Theorem \cite{AK} Thm 5.6,
 which implies that $\varphi_{j\#}T_j \rstr B_{\varphi_0(x_0)}(r)$ is 
an integral current for almost all values of $r$.  
The reader is referred to \cite{SorWen1} where local versions of lower
semicontinuity of mass and continuity of filling volume are applied.
\end{rmrk}

\subsection{Cancellation and Intrinsic Flat Convergence} \label{subsect-SorWen1}

When a sequence of integral currents converges to the ${0}$ current
due to the effect
of two sheets of opposing orientation coming together, this is referred to as
cancellation.  In Example~\ref{example-cancels} depicted in Figure~\ref{figure-cancels},
we see that the same effect can occur causing  a sequence of Riemannian
manifolds to converge in the intrinsic flat sense to the ${\bf 0}$ current space.  Naturally 
it is of great importance to avoid this situation.

In \cite{SorWen1}, the authors proved a few theorems providing conditions
that prevent cancellation of certain weakly converging sequences of integral currents.  
These theorems immediately
apply to prevent the cancellation of certain sequences of Riemannian manifolds 
although they do not extend to arbitrary integral current spaces.   The reader is
referred to \cite{SorWen1} for the most general statements of these results.

In this section we give some of the intuition that led to these results,
then review Greene-Petersen's compactness theorem and finally
review a result  of \cite{SorWen1}, Theorem~\ref{thm-contractibility-no-cancellation},
which states that under the conditions of Greene-Petersen's theorem,
there is no cancellation and, in fact, 
the intrinsic flat and Gromov-Hausdorff limits agree. 

\begin{rmrk}
The initial observation that lead to the results in \cite{SorWen1} was that the
sequence in Example~\ref{example-cancels} depicted in
Figure~\ref{figure-cancels} has increasing topological type.
The only way to bring two sheets together with an intrinsic distance on a smooth
Riemannian manifold, was to create many small tubes between the two sheets,
and all these tubes lead to increasing local topology.  
\end{rmrk}

\begin{rmrk} \label{rmrk-avoid-cancellation}
The second observation was that, in order to avoid cancellation, one needed
to locally bound the filling volume of spheres away from $0$.  More precisely
the filling volumes of distance spheres of radius $r$ had to be bounded below by
$Cr^m$, so that the filling volumes in the limit would have the same bound.
Since the volume of a ball is larger than the filling volume of the sphere,
we could then prove the limit points had positive density.  
\end{rmrk}

Note that
if a sequence of Riemannian manifolds converges to a Riemannian
manifold with a cusp singularity as in Example~\ref{ex-to-cusp}
depicted in Figure~\ref{fig-to-cusp}, the cusp point disappears
in the limit because it does not have positive density [Example~\ref{example-settled},
Example~\ref{basic-mani-sing}].    To avoid cancellation,
we need to prevent points from disappearing.

\begin{figure}[h] 
   \centering
   \includegraphics[width=4.5in]{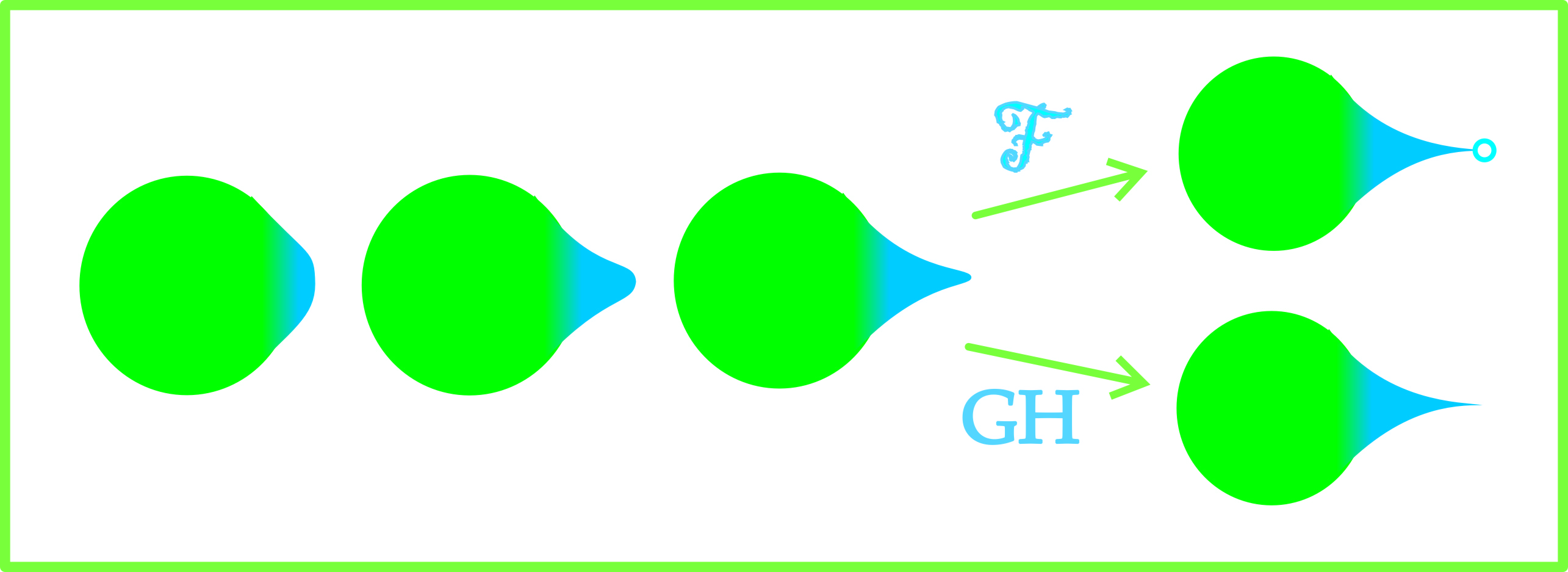} 
   \caption{The intrinsic flat limit does not include the tip of the cusp.}
   \label{fig-to-cusp}
\end{figure}

In Gromov's
initial paper defining filling volume, he proved the filling volume could be
bounded from below by the filling radius and the filling radius could be bounded
from below by applying contractibility estimates \cite{Gromov-filling}.
Greene-Petersen applied Gromov's technique to estimate the filling
volumes of balls and consequently prove the following compactness
theorem \cite{Greene-Petersen}.  They needed a uniform estimate
on contractibility to prove their theorem:

\begin{defn}
On a Riemannian manifold, $M^m$,
a geometric contractibility function, $\rho:(0,r_0] \to (0,\infty)$, is a function
such that $\lim_{r\to 0}\rho(r)=0$ and such that
any ball $B_p(r)\subset M^m$ is contractible in 
$B_p\left(\rho\left(r\right)\right)\subset M^m$.
\end{defn}           

\begin{thm}[Greene-Petersen] \label{thm-Greene-Petersen} If a
sequence of Riemannian manifolds  $M^m_j$ 
without boundary have  a uniform geometric contractibility function, 
$\rho:(0,r_0] \to (0,\infty)$
then one can construct uniform lower bound $\nu_{\rho,m}: (0,D] \to (0,\infty)$ such that 
\be
\vol\left(B_p\left(r\right) \right) \ge \fillvol(\partial B_p(r)) \ge \nu_{\rho,m}(r)   
\ee
for all balls $B_p(r)$ in all the manifolds.
If, in addition, there is a uniform upper bound on volume $\vol(M^m_j)\le V$,
then a subsequence $M_j^m \GHto Y$.
\end{thm}

Immediately below the statement of this theorem, Greene-Petersen mention
that if $\rho$ is linear, $\rho(r)=\lambda r$, then there exists a constant
$C_m>0$ such that $\nu_{\rho,m}(r)\ge C_m r^m$.  This is exactly the bound
needed to avoid cancellation.

If the geometric contractibility function $\rho$ is not linear then one can have a sequence
of Riemannian manifolds which converge to a Riemannian
manifold with a cusp singularity as in Example~\ref{ex-to-cusp}
depicted in Figure~\ref{fig-to-cusp}.  
The lack of a uniform linear geometric contractibility function for that sequence of
is depicted in Figure~\ref{fig-GHFcusp2}.  

 \begin{figure}[h] 
   \centering
   \includegraphics[width=4.5in]{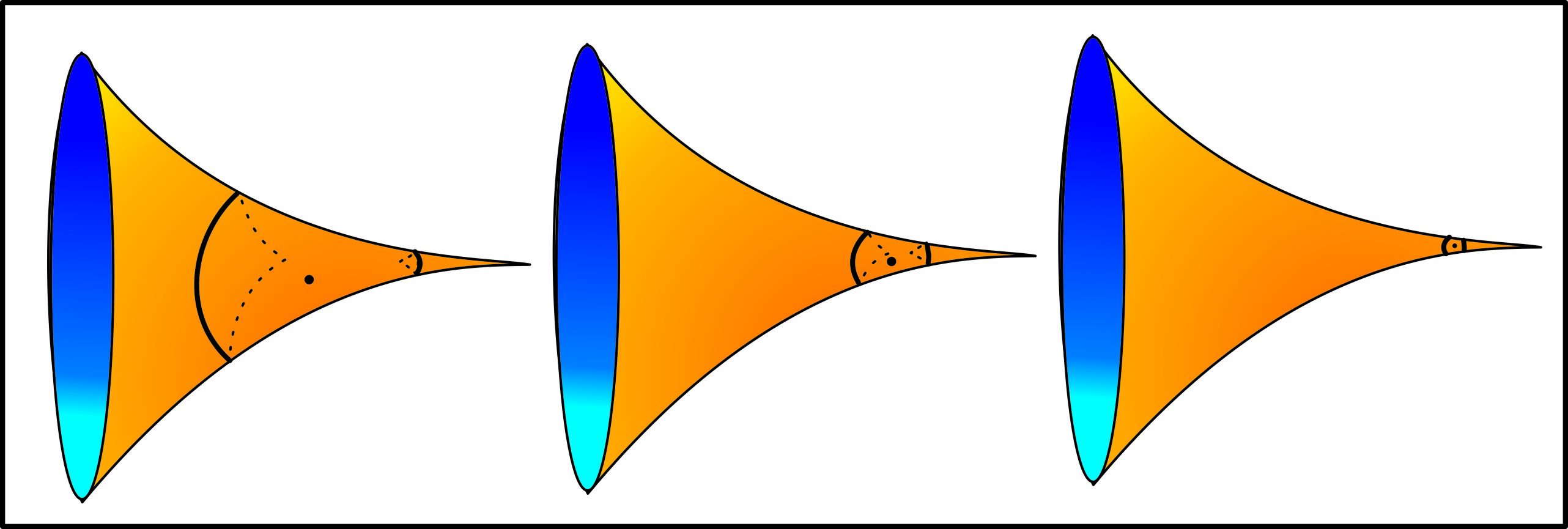} 
   \caption{The first ball contracts in a ball of twice its radius, the second in a ball of
   3 times its radius, the next in a ball of five times its radius...}
   \label{fig-GHFcusp2} 
\end{figure}

Cones have linear geometric contractibility functions (as seen in 
Figure~\ref{fig-GHFcusp3}).  Riemannian manifolds with conical
singularities viewed as integral current spaces include their
conical singularities [Example~\ref{example-settled},
Example~\ref{basic-mani-sing}]. 

 \begin{figure}[h] 
   \centering
   \includegraphics[width=4.5in]{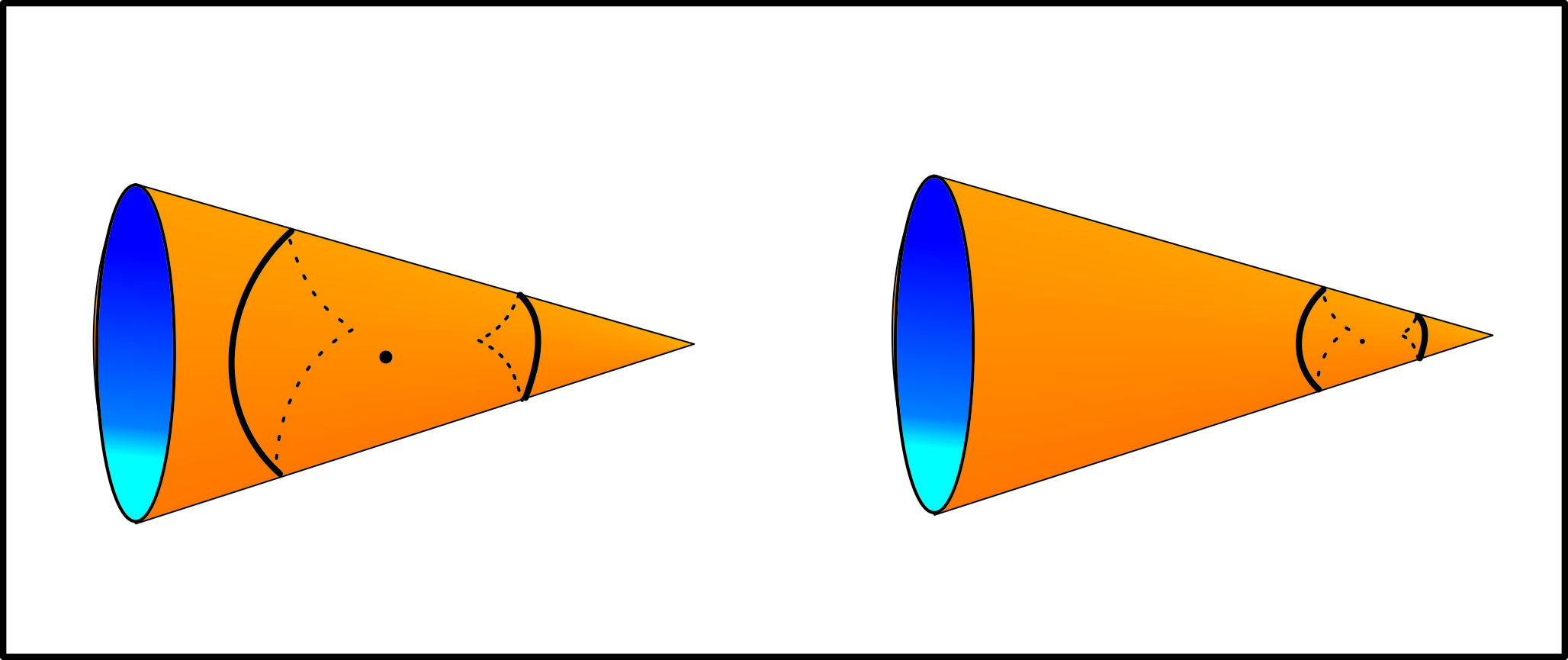} 
   \caption{The contractibility function is $\rho(r)=2r$ here.}
   \label{fig-GHFcusp3} 
\end{figure}


In \cite{SorWen1}, we dealt with a far more general class of integral current
spaces than Riemannian manifolds.  We began by applying Gromov's
compactness theorem to isometrically embed the sequence into a
common metric space where 
we used a notion of integral filling volume (c.f. \cite{Wenger-isoper}), 
which is well defined for integral currents without boundary.
We did not use Greene-Petersen's
smoothing arguments applying Ambrosio-Kirchheim's Slicing Theorem instead.
We needed to adapt everything to integral filling
volumes, so we applied a new
Lipschitz extension theorem akin to that of Lang-Schlichenmaier  \cite{Lang-Schlichenmaier}.
This lead to the following local theorem we could apply to avoid cancellation.
The following is a simplified restatement of \cite{SorWen1} Theorem 4.1:

\begin{thm} \cite{SorWen1}\label{thm-contractible}\label{contractibility-no-cancellation}
If $M^m$ is an oriented Lipschitz manifold of finite volume
with integral current structure, $T$, and if there is a ball,
 $B_x(r)\subset M^m$, that has $\partial T \rstr B_x(r)=0$ and if
$B_x(r)$ has a uniform linear geometric contractibility function,
$\rho:[0,2r]\to [0,\infty)$, with $\rho(r)=\lambda r$, then
\be \label{sorwen1.estimate}
||T||(B_x(s))\ge \textrm{Fillvol}_\infty(\partial (T\rstr \bar{B}_x(r)) \ge C_\lambda s^m \,\,
a.e. \,s\in [0, r/(2^{m+6}\lambda^{m+1})].
\ee
\end{thm}

\begin{example}
Note that the condition here that $\partial T \rstr B_x(r)=0$ is necessary.
If $M^m$ were a thin flat strip $[0,1]\times[0,\epsilon]$
, all balls in $M^m$ would have $\rho(r)=r$,
but the volumes of the balls would be less than $2r\epsilon$.
\end{example}

This theorem combined with the ideas described in 
Remark~\ref{rmrk-avoid-cancellation} leads to the the following theorem
demonstrating that the limits occurring in Greene-Petersen's compactness
theorem have no cancellation:

\begin{thm} \cite{SorWen1}\label{thm-contractibility-no-cancellation}
If a sequence of connected oriented
Lipschitz manifolds 
without boundary, $M^m_j=(X_j,d_j,T_j)$ has
a uniform linear geometric contractibility function,
$\rho:[0,r_0]\to [0,\infty)$, with $\rho(r)=\lambda r$, 
and a uniform upper bound on volume, then a subsequence converges in
both the intrinsic flat sense and the Gromov-Hausdorff sense to
the same space $M^m=(X,d,T)$.  In particular, $M^m$ is a countably
$\mathcal{H}^m$ rectifiable metric space.
\end{thm}

A more general version of Theorem~\ref{thm-contractibility-no-cancellation}
which allows for boundaries, is stated as Corollary 1.6 in our paper \cite{SorWen1}.

\begin{rmrk}\label{rmrk-Schul-Wenger}
If the contractibility function is not linear, Schul-Wenger
have shown the limit space need not be countably $\mathcal{H}^m$ rectifiable
in the Appendix of
\cite{SorWen1}.  Note that Ferry-Okun have shown that without a uniform
upper bound on volume, these sequences can converge to an infinite
dimensional space \cite{Ferry-Okun}.
\end{rmrk}

\subsection{Ricci and Scalar Curvature}\label{subsect-Ricci}

Gromov proved that a sequence of manifolds, $M^m_j$, with nonnegative Ricci curvature
and a uniform upper bound on diameter, have a subsequence which converges in
the Gromov-Hausdorff sense to a compact geodesic space, $Y$ \cite{Gromov-metric}.   
Cheeger-Colding
proved that in the noncollapsed setting, where the volumes are uniformly bounded
below, the manifolds converge in the metric measure sense to $Y$ with the
Hausdorff measure, $\mathcal{H}^m$.  In particular, if $p_j \in M_j$ converge to
$y\in Y$ then $\vol(B_{p_j}(r))$ converges to $\mathcal{H}^m(B_y(r))$.
Furthermore $Y$ is countably
$\mathcal{H}^m$ rectifiable with Euclidean tangent cones almost everywhere.
Points with Euclidean tangent cones are called {\em regular points} and, at such
points, the density of the Hausdorff measure is $1$.  In fact
$\lim_{r\to 0} \mathcal{H}^m(B_y(r))/r^m=\omega_m$. \cite{ChCo-PartI}.

Such sequences
do not have uniform geometric contractibility functions as seen by Perelman's example
in \cite{Perelman-example}.  In fact Menguy proved the limit space could
have infinite topological type \cite{Menguy-inf-top-type}.  Nevertheless, in \cite{SorWen1},
the authors proved that the Gromov-Hausdorff and intrinsic flat distances agree in
this setting:

\begin{thm} \label{thm-ricci} \cite{SorWen1}
If a noncollapsing
sequence of oriented Riemannian manifolds 
without boundary, $M^m_j=(X_j,d_j,T_j)$ has
nonnegative Ricci curvature and a uniform upper bound on diameter,
then a subsequence converges in
both the intrinsic flat sense and the Gromov-Hausdorff sense to
the same space $M^m=(X,d,T)$.  
\end{thm}

This theorem can be viewed as an example of a noncancellation theorem.
The proof is based on Theorem~\ref{thm-contractible} and the fact
that Perelman proved that balls of large volume in a manifold with
nonnegative Ricci curvature are contractible \cite{Perelman-max-vol}.
We also applied the work of Cheeger-Colding \cite{ChCo-PartI}, which
states that in this setting the volumes of balls converge and that almost
every point in the Gromov-Hausdorff limit is a regular point.  Regular
points have Euclidean tangent cones and 
$\lim_{r\to 0} \mathcal{H}^m(B_y(r))/r^m=\omega_m$.

\begin{rmrk}
It would be interesting if one could prove this theorem directly without
resorting to the powerful theory of Cheeger-Colding.  That would give new
insight perhaps allowing one to extend this result to situations with weaker
conditions on the curvature.
\end{rmrk}

In \cite{SorWen1} we presented an example of a sequence of three
dimensional Riemannian manifolds with positive scalar curvature that
converge in the intrinsic flat sense to the $0$ integral current space.
Example~\ref{example-cancels} depicted in Figure~\ref{figure-cancels}
is a 2 dimensional version of this example.  The example with positive
scalar curvature is constructed by connecting a pair of standard three
dimensional spheres by an increasingly dense collection of tunnels.
Each tunnel is constructed using Schoen-Yau or Gromov-Lawson's
method \cite{MR535700} \cite{MR569070}.    
This
sequence has increasingly negative Ricci and sectional curvatures
within the tunnels but the scalar curvature remains positive.  
Note that each tunnel has a minimal two sphere inside.  It is natural
in the study of general relativity, to require that a manifold have positive
scalar curvature and no interior minimal surfaces.  The boundary is
allowed to consist of minimal surfaces.

The following conjecture is based upon discussions with Ilmanen:

\begin{conjecture}  \label{conj-scalar}
A converging sequence of three dimensional
Riemannian manifolds with positive
scalar curvature, a uniform lower bound on volume,
 and no interior minimal surfaces converges without
cancellation to a nonzero integral current space.
\end{conjecture}

A solution to this 
conjecture would have applications in general relativity and is essential to
solving Ilmanen's 2004 proposal that a new weak form of convergence needs
to be developed to better understand manifolds with positive scalar curvature.

\vspace{.1cm}

\subsection{Wenger's Compactness Theorem}\label{subsect-Wenger-compactness}


In \cite{Wenger-compactness}, the second author has proven the key
compactness theorem for the intrinsic flat distance:

\begin{thm}\label{thm-Wenger-compactness}  {\em{\cite{Wenger-compactness} [Theorem 1.2]}}  
Let $m, C, D > 0$ and let $\bar{X}_j$ be a sequence of complete metric spaces. 
Given $T_j \in \intcurr_m(\bar{X}_j)$ with uniform bounds on total mass and diameter:
\be \label{eqn-compact-1}
\mass(T_j) +\mass(\partial T_j) \le C
\ee 
and 
\be \label{eqn-compact-2}
\diam(\spt(T_j) ) \le D
\ee
then 
there exists a subsequence $T_{j_i}$, a complete metric space $Z$, an integral current $T \in \intcurr_m(Z)$
and isometric embeddings $\varphi_{j_i}: \bar{X}_{j_i} \to Z$ such that 
\be \label{eqn-compact-3}
d^Z_F\left(\varphi_{j_i\#}T_{j_i},T\right) \to 0.
\ee 

In particular, if $M_n=(X_n, d_n,T_n)$ is a sequence of integral current spaces satisfying  (\ref{eqn-compact-1})
and (\ref{eqn-compact-2}), then a subsequence converges in the intrinsic flat sense to an integral
current space of the same dimension.  The limit space is in fact $M=(\set(T), d_Z, T)$.
\end{thm}

In particular, sequences of oriented Riemannian manifolds with boundary
with a uniform upper bound on volume, on the volume of the boundary and
on diameter have a subsequence which converges in the intrinsic flat sense
to an integral current space.  Note that even when the sequence of manifolds 
is compact, the limit space need not be precompact as seen in Example~\ref{ex-many-tips}
depicted in Figure~\ref{fig-many-tips}.


We now apply this compactness theorem combined
with techniques from the proof of
Theorem~\ref{converge} to prove Theorem~\ref{Cauchy-to-Converge}.
We do not apply this compactness theorem anywhere else in this paper.

Contrast this with Theorem~\ref{cauchy} and see  Remark~\ref{rmrk-cauchy}.

\begin{thm} \label{Cauchy-to-Converge}  
Given an intrinisic flat Cauchy sequence of
integral current spaces, 
$M^m_j=\left(X_j, d_j, T_j\right)$
with a uniform bound on total mass, $\nmass(M_j) \le V_0$, and a uniform bound
on diameter, $\diam(M_j) \le D$, there exists a complete metric space
$Z$, and a sequence
of isometric embeddings $\varphi_j: X_j \to Z$
such that $\varphi_{j\#}T_j$ is a flat Cauchy
sequence of integer rectifiable currents in $Z$
which converges in the flat sense to an integral current $T\in \intcurr_m(Z)$.

Thus $M^m_j$ converges in the intrinsic flat sense to an integral current
space $(\set(T), d_Z, T)$.
\end{thm}

\begin{proof}
First there is a subsequence
$(X_{j_i}, d_{j_i}, T_{j_i})$ which converges in the intrinsic flat sense
to an integral current space $(X,d,T)$, by Wenger's compactness theorem.  
Since $(X_j,d_j,T_j)$ is
Cauchy, it also converges to $(X,d,T)$.  Then Theorem~\ref{converge}
then yields the claim.
\end{proof}


\section{{\bf Lipschitz Maps and Convergence}} \label{Sect-reln}

We review Lipschitz convergence and
prove that when sequences of manifolds
converge in the Lipschitz sense, then they converge in the intrinsic flat
sense.  As a consequence, sequences of manifolds which converge in
the $C^{k,\alpha}$ sense or the $C^\infty$ sense, also converge in the
intrinsic flat sense.  Lemmas in this section will also be useful when proving
the examples in the final section of the paper.

\subsection{Lipschitz Maps}  \label{subsect-lip-maps}


The purpose of this subsection is to list some basic properties of the intrinsic 
flat norm of an integral current space. 
Some of the lemmas will be used later on for the construction 
of examples in the Appendix. 
Others will be used to relate the Lipschitz convergence to intrinisic flat convergence
[Theorem~\ref{thm-lip-to-flat}].

Recall that a metric space $X$ is called injective if for every metric space $Y$, every subset $A\subset Y$ and every Lipschitz map $\varphi: A\to X$ there exists a Lipschitz extension $\bar{\varphi}: Y\to X$ of $\varphi$ with the same Lipschitz constant. It is not difficult to check that given a set $Z$, the Banach space $l^\infty\left(Z\right)$ of bounded functions, endowed with the supremum norm, is injective
(c.f. \cite{BenLinText} p 12-13).

Given a complete metric space $X$ and $T\in\intcurr_m\left(X\right)$ we define
\begin{equation}
 \flatnorm_X\left(T\right):= \inf\left\{\mass\left(U\right) + \mass\left(V\right): U\in\intcurr_m\left(X\right), V\in\intcurr_{m+1}\left(X\right), T = U+\bdry V\right\}
\end{equation}
whereas
\begin{equation}
 \flatnorm\left(T\right):= \inf\left\{\flatnorm_Z\left(\varphi_{\#}T\right): \text{ $Z$ metric space, $\varphi: X\hookrightarrow Z$ isometric embedding}\right\}.
\end{equation}

\bl\label{lemma:injective-spaces-flatnorm} \label{lem-wen-1.1}
Given $X$ an injective metric space and $T\in\intcurr_m\left(X\right)$ we have $\flatnorm\left(T\right) = \flatnorm_X\left(T\right).$
\el

\begin{proof}
 Let $Z$ be a metric space and $\varphi: X\hookrightarrow Z$ an isometric embedding. Since $X$ is injective there exists a $1$-Lipschitz extension $\psi: Z\to X$ of $\varphi^{-1}: \varphi\left(X\right) \to X$. Let $U\in\intcurr_m\left(Z\right)$ and $V\in\intcurr_{m+1}\left(Z\right)$ with $\varphi_{\#}T=U+\bdry V$ and observe that $U':= \psi_{\#}U$ and $V':= \psi_{\#}V$ satisfy $T = U' + V'$ and
\begin{equation}
  \mass\left(U'\right)+\mass\left(V'\right) \leq \mass\left(U\right) + \mass\left(V\right).
 \end{equation}
 Since $U$ and $V$ were arbitrary, it follows that $\flatnorm_X\left(T\right)\leq\flatnorm\left(T\right)$.
\end{proof}

\bl\label{lemma-Lipschitz-fillvol} \label{lem-wen-1.2}
Let $X$ and $Y$ be complete metric spaces and let $\varphi: X \to Y$ be a $\lambda$-bi-Lipschitz map.
Then for each
$T\in\intcurr_m\left(X\right)$  we have
$$\flatnorm\left(T\right)\leq\lambda^{m+1}\flatnorm_Y\left(\varphi_\# T\right).$$
\el

\begin{proof}
Let $\iota: X\to l^\infty\left(X\right)$ be the Kuratowski embedding and let $\bar{\varphi}: Y\to l^\infty\left(X\right)$ be a $\lambda$-Lipschitz extension of $\iota\circ\varphi^{-1}$. Given $U\in\intcurr_m\left(Y\right)$ and  $V\in\intcurr_{m+1}\left(Y\right)$ with $\varphi_\#T= U+\bdry V$ then $\iota_\#T=\bar{\varphi}_{\#}U + \bdry\left(\bar{\varphi}_{\#}V\right)$ and thus
\begin{equation}
\flatnorm\left(T\right)= \flatnorm_{l^\infty\left(X\right)}\left(\iota_\# T\right) \le \mass\left(\bar{\varphi}_\#U\right) + \mass\left(\bar{\varphi}_\#V\right)\leq \lambda^{m}\mass\left(U\right) + \lambda^{m+1}\mass\left(V\right).
\end{equation}
Minimizing over all $U$ and $V$ selected as above completes the proof.
\end{proof}


%

\bl  \label{Euclidean-fillvol}  \label{lem-wen-1.3}
  Let $X$ be a complete metric space and $\varphi: X\to\R^N$ a $\lambda$-lipschitz map
  where $\lambda \ge 1$. For $T\in\intcurr_m\left(X\right)$ we have
 \begin{equation}
  \flatnorm\left(T\right)\geq \left(\sqrt{N}\lambda\right)^{-\left(m+1\right)}\flatnorm_{\R^N}\left(\varphi_\#T\right).
 \end{equation}
\el

We illustrate the use of the lemma by a simple example: Let $M$ be an $m$-dimensional oriented submanifold of $\R^N$ of finite volume and finite boundary volume. Endow $M$ with the length metric and call the so defined metric space $X$. Clearly, the inclusion $\varphi: X\to\R^N$ is $1$-Lipschitz. Let $T$  be the integral current in $X$ induced by integration over $M$. The above lemma thus implies
\begin{equation}
 \flatnorm\left(T\right)\geq N^{-\frac{m+1}{2}}\flatnorm_{\R^N}\left(\Lbrack M\Rbrack\right)
\end{equation}
where $\Lbrack M\Rbrack$ is the current in $\R^N$ induced by integration over $M$.

\begin{proof}
\CS{Let $A=\iota\left(X\right) \subset l^\infty\left(X\right)$ where $\iota: X\to l^\infty\left(X\right)$ denotes the Kuratowski embedding. 
 Then $\varphi\circ \iota^{-1}: A \to \R^N$ is a $\lambda$-Lipschitz map.}
 By McShane's extension theorem there exists a $\sqrt{N}\lambda$-Lipschitz extension $\psi: l^\infty\left(X\right)\to\R^N$ of $\varphi\circ\iota^{-1}: A \to \R^N$ \cite{McShane-34}.
 
 Thus, if $U\in\intcurr_m\left(l^\infty\left(X\right)\right)$ and $V\in\intcurr_{m+1}\left(l^\infty\left(X\right)\right)$ are such that $\iota_\# T = U + \bdry V$ then
\be
\varphi_\#T= \psi_\# \iota_\# T= \psi_\# U + \psi_\# \left(\partial V\right)=\psi_\#U + \bdry\left(\psi_\#V\right)
\ee 
and
\begin{equation}
 \flatnorm_{\R^N}\left(\varphi_\#T\right)\leq \mass\left(\psi_\#U\right) + \mass\left(\psi_\#V\right) \leq \left(\sqrt{N} \lambda\right)^{m+1}[\mass\left(U\right) + \mass\left(V\right)].
\end{equation}
We now obtain the claim by minimizing over all $U$ and $V$ and using \lemref{lemma:injective-spaces-flatnorm}.
\end{proof}
 
\CS{ In the following lemma we bound the intrinsic flat distance between an integral
current space and its image under a bi-Lipschitz map.   Recall the total mass
$\nmass\left(T\right)=\mass\left(T\right)+\mass\left(\partial T\right)$ [Definition~\ref{def-total-mass}].}

\bl \label{Lip-In} 
Let $X$ and $Y$ be complete metric spaces and let $\varphi: X \to Y$ be a $\lambda$-bi-Lipschitz map for some $\lambda>1$.
Then for $T\in\intcurr_m\left(X\right)$ viewed as an integral current space
$T=\left(\set\left(T\right), d_X,T\right)$ and $\varphi_\#T=\left(\set\left(\varphi_\#T\right), d_Y, \varphi_\# T\right)$
we have
\begin{equation}
d_{\mathcal{F}}\left(T, \varphi_\#T\right) \le 
k_{\lambda,m}
 \max\{\diam\left(\spt T\right), \diam\left(\varphi\left(\spt T\right)\right)\}\,\nmass\left(T\right)
\end{equation}
where $k_{\lambda,m}:=\frac{1}{2}\left(m+1\right)\lambda^{m-1}\left(\lambda - 1\right)$.
\el

 
\begin{proof}
 Let $C_0:= \spt T$, $C_1:= \varphi\left(C_0\right)$, and denote by $d_0$ and $d_1$ the metric on $C_0$ and $C_1$, respectively. Let $D:= \max\{\diam C_0, \diam C_1\}$. Let $d_Z$ be the metric on $Z:= C_0\sqcup C_1$ which extends $d_0$ on $C_0$ and $d_1$ on $C_1$ and which satisfies
\begin{equation}
 d_Z\left(x, x'\right)= \inf\{d_0\left(x,\bar{x}\right)+d_1\left(\varphi\left(\bar{x}\right), x'\right): \bar{x}\in C_0\} + \lambda'D,
\end{equation}
whenever $x\in C_0$ and $x'\in C_1$ and where $\lambda':= \frac{1}{2}\lambda^{-1}\left(\lambda - 1\right)$.
It is not difficult to verify that $d_Z$ is in fact a metric.

Let $\varphi_i: C_i\to l^\infty\left(Z\right)$ 
be the composition of the inclusion map with the Kuratowski embedding. 
Note that these are isometric embeddings. Define a map
$\psi: [0,1]\times C_0\to l^\infty\left(Z\right)$ using linear interpolation:
\begin{equation} \label{psi-def-1}
\psi\left(t,x\right):= \left(1-t\right)\varphi_0\left(x\right)+t\varphi_1\left(\varphi\left(x\right)\right).
\end{equation}
 It is then clear that 
 \be \label{eqn-lip1}
 \Lip\left(\psi\left(\cdot,x\right)\right)= \lambda' D \quad \forall x\in C_0
 \quad \textrm{and} \quad \Lip\left(\psi\left(t,\cdot\right)\right)\leq\lambda \quad \forall t\in[0,1].
 \ee  
We now apply the linear interpolation to define two currents,
\be
\begin{split}
U & :=\psi_\#\left([0,1]\times \bdry T\right)\in  \intcurr_m\left(l^\infty\left(Z\right)\right)
\textrm{ and } \\
 V & := \psi_{\#}\left([0,1]\times T\right) \in \intcurr_{m+1}\left(l^\infty\left(Z\right)\right),
\end{split}
 \ee
  where the product of currents is defined as in \cite{Wenger-isoper} Defn 2.8.   
  By Theorem 2.9 in \cite{Wenger-isoper},
\begin{equation}
 \bdry\left([0,1]\times T\right) = [1]\times T - [0]\times T - [0,1]\times \bdry T.
\end{equation}
So if we push forward by $\psi$ applying (\ref{psi-def-1})
we get
\begin{eqnarray*}
\bdry V &=& \psi_\# \left([1]\times T\right) -\psi_\#\left( [0]\times T\right) -\psi_\#\left( [0,1]\times \bdry T\right)\\
&=&\varphi_{1\#}\varphi_{0\#}T - \varphi_{0\#}T - U.
\end{eqnarray*}
Since $\varphi_0$ is an isometric embedding  we have
\begin{equation}
d_{\mathcal{F}}\left(\varphi_\# T, T\right) \le d_F^Z\left(\varphi_{0\#}\varphi_\# T, \varphi_{0\#}T\right) \le
\mass\left(U\right) + \mass\left(V\right).
\end{equation}
By Proposition 2.10 in \cite{Wenger-isoper}, we have
\begin{equation} \label{mass-est}
\mass\left(U\right) + \mass\left(V\right) \le m\lambda^{m-1}\lambda' D\, \mass\left(\bdry T\right) + \left(m+1\right)\lambda^m\lambda' D\, \mass\left(T\right).
\end{equation}
Thus we obtain the lemma.
\end{proof}

\subsection{Lipschitz and Smooth Convergence}\label{subsect-lip-conv}


Over the years various notions
of smooth convergence and compactness theorems have been proven.  We recommend
Petersen's textbook \cite{Petersen-text}  for a survey of these various notions of
convergence progressing from $C^{1,\alpha}$ to $C^\infty$ convergence.  All these notions
involve maps $f_j:M_j\to M_\infty$ and the push forward of the metric tensors $g_j$
from $M_j$ to positive definite tensors $f_{j*}g_j$ on $M$ and then studying the
appropriate convergence of these tensors to $g$.  

A weaker notion than these
notions is Gromov's Lipschitz convergence introduced in 1979
which does not require one to  examine
the metric tensors but rather just the distances on the spaces \cite{Gromov-metric}[Defn 1.1 and Defn 1.3].  
In this section we will briefly review Lipschitz convergence and prove that
whenever a sequence of manifolds converges in the Lipschitz sense then it
converges in the intrinsic flat sense [Theorem~\ref{thm-lip-to-flat}].  As a consequence,
$C^{1,\alpha}$ convergence and all other smooth forms of convergence are
stronger than intrinsic flat convergence as well.  That is, any sequence of manifolds
converging in the smooth sense to a manifold, converges in the intrinsic flat sense
as well.

\begin{defn}[Gromov] \label{defn-lip-conv} 
The Lipschitz distance between two metric spaces $X,Y$, is defined as
\be
d_L\left(X,Y\right)= \inf \{ \,|\log \dil \left(f\right)| + |\log \dil\left( f^{-1} \right)|:   \textrm{ bi-Lipschitz }f:X\to Y \, \}
\ee
where 
\be
\dil\left(f\right)=\sup \left\{ \frac {d\left(f\left(x\right),f\left(y\right)\right)}{d\left(x,y\right)}: \,\, x,y\in X \textrm{ s.t. } x\neq y\, \right\}.
\ee
When there is no bi-Lipschitz map from $X$ to $Y$ one says $d_L\left(X,Y\right)=\infty$.
\end{defn}

Note that if a sequence of orientable Riemannian manifolds $M_j$ converges in the Lipschitz
sense to a metric space $M$, then $M$ is bi-Lipschitz to an
orientable  Riemannian manifold.
In particular $M$ is an orientable  Lipschitz manifold and by Remarks~\ref{rmrk-biLip-matching}
and~\ref{Lip-mani-structure}, it has a natural structure as an integral current space
determined completely by choosing an orientation on the space.

\begin{thm} \label{thm-lip-to-flat}
If $M_j$ are orientable Lipschitz manifolds converging in the Lipschitz sense
to an oriented Lipschitz manifold $M$, then after matching orientations of
the $M_j$ to the limit manifold, $M$,  the oriented Lipschitz manifolds $\Lbrack M_j\Rbrack$ 
converge in the intrinsic flat sense to $\Lbrack M \Rbrack$.

In fact, whenever $M$ and $N$ are Lipschitz manifolds with matching orientations,
\begin{equation} \label{mani-flat-lip-1}
d_{\mathcal{F}}\left(M, N\right) < k_{\lambda,m} \max \{\diam\left(M\right), \diam\left(N\right)   \} \,   ( \vol\left(M\right) +\vol(\partial M))
\end{equation}
where $k_{\lambda,m}:=\frac{1}{2}\left(m+1\right)\lambda^{m-1}\left(\lambda - 1\right)$
and where 
$\lambda=e^{d_L\left(M,N\right)}$.  
\end{thm}

Gromov has proved that Lipschitz convergence implies
Gromov-Hausdorff convergence \cite{Gromov-metric}[Prop 3.7].  So that in
this setting the Gromov-Hausdorff limits and intrinsic flat limits agree.
Gromov's proof applies to any sequence of metric spaces.  We cannot
extend our theorem to arbitrary integral current spaces because, in general,
one cannot just reverse orientations to match the orientations between
a pair of bi-Lipschitz homeomorphic integral current spaces.



\begin{proof}

Recall Remarks~\ref{rmrk-biLip-matching} and~\ref{Lip-mani-structure}, 
that when   $\psi: M^m\to N^m$ is a  bi-Lipschitz homeomorphism
between connected oriented Lipschitz
manifolds then $\psi_\#\Lbrack M\Rbrack = \pm\Lbrack N\Rbrack$.
Once the orientations have been fixed to match, the sign becomes positive.

Lemma~\ref{Lip-In} implies that
\begin{equation} \label{mani-flat-lip-1}
d_{\mathcal{F}}\left(M, N\right) \le \frac{1}{2}\left(m+1\right)\lambda^{m-1}\left(\lambda - 1\right) \max\{\diam\left(M\right), \diam\left(N\right)\}\,(\vol\left(M\right)+ \vol(\partial M))
\end{equation}
where $\lambda>1$ is the bi-Lipschitz constant for $\psi$.
Note further that
\be
\log \lambda \le |\log dil\left(\psi\right)|+|\log dil\left(\psi^{-1}\right)|\le 2\log \lambda.
\ee
Taking the infimum of this sum over all $\psi$
and applying (\ref{mani-flat-lip-1}), we see that 
\begin{equation} \label{mani-flat-lip-2}
d_{\mathcal{F}}\left(M, N\right) \le k_{\lambda.m}
\max\{\diam\left(M\right), \diam\left(N\right)\}\,(\vol\left(M\right)+ \vol(\partial M))
\end{equation}
where $\lambda=e^{d_L\left(M,N\right)}$.  

Now whenever a sequence of
Lipschitz manifolds, $M_j$, converges in the Lipschitz sense 
to a Lipschitz manifold, $M$,  then
\be
\lambda_j= e^{d_L\left(M_j,M\right)} \to 1
\textrm{ and  }\diam\left(M_j\right) \to \diam\left(M\right).
\ee  
Thus $d_{\mathcal{F}}\left(M_j, M\right)$ is less than or equal to
\begin{equation} \label{mani-flat-lip-3}
k_{\lambda_j,m} \max\left\{\diam\left(M_j\right), \diam\left(M\right)\right\}\,\left( \vol\left(M\right)_{\,}^{\,}+\vol(\partial M)\right)
\end{equation}
which converges to $0$ as $j\to \infty$.
\end{proof}



\appendix
\section{\bf{Examples by C. Sormani}} \label{sect-examples}

In this section we present proofs of all the examples referred
to throughout the paper.
In order to prove our examples converge in the intrinsic flat sense,
we need convenient ways to isometrically embed our Riemannian
manifolds into a common metric space, $Z$.   In most examples we explicitly
construct $Z$.  Two major techniques we develop are the {\em bridge
construction} [Lemma~\ref{lem-bridge-Z} and Proposition~\ref{prop-bridge-filling}]
and the {\em pipe filling construction}  [Remark~\ref{rmrk-pipe-filling}].  
In all examples in this section, the common metric space $Z$
is an integral current spaces whose 
tangent spaces are Euclidean almost everywhere so that the weighted volume
and mass agree [Lemma~\ref{lemma-weight} and Remark~\ref{rmrk-lambda}].  We also have multiplicity one (so that the volume and
mass agree) enabling us to use volumes to estimate the intrinsic flat distance.

\subsection{Isometric Embeddings}\label{subsect-isom-embed}

Recall that a metric space is a geodesic or length space if the metric is determined
by taking an infimum over the lengths of all rectifiable curves.   In Riemannian
manifolds, the lengths of curves are defined by integrating the curve using
the metric tensor.
Given a connected subset, $X$,
of a metric space, $Z$, one has the restricted metric, $d_Z$, on $X$ as well as an
induced length metric on $X$, $d_X$, which is found by taking the infimum of
all lengths of rectifiable curves lying within $X$ where the lengths of the curves are computed
locally using $d_Z$:
\be
L\left(C\right) = \sup_{0=t_0<t_1<\cdots t_k=1} \sum_{i=1}^k d_Z\left(c\left(t_i\right), c\left(t_{i-1}\right)\right).
\ee
When one uses this induced length metric on $X$, then
$X$ may no longer isometrically embed into $Z$.

In our first lemma, we describe a process of attaching one geodesic metric space,
$Y$, to another metric space, $Z$, along a closed subset, 
$X\subset Z$, to form a metric space, $Z'$, into
which $Z$ isometrically embeds.  This lemma is one sided, as $Y$ need not isometrically
embed into $Z'$ [see Figure~\ref{fig-hemi-bump}].  

\begin{lem} \label{lem-local-filling}  
Let $\left(Z, d_Z\right)$ and $\left(Y, d_Y\right)$ be geodesic metric spaces and let $X \subset Z$
be a closed subset.  
Suppose $\psi:\left(X,d_X\right)\to \left(Y, d_Y\right)$ is an isometric embedding. 

Then we can create a metric space $Z'=Z\disjointunion Y/ \sim \,\,$
where $z\sim y$ iff $z\in X\subset Z$ and $y=\psi\left(z\right)$.  We endow $Z'$ with the induced
length metric where lengths of curves are measured by $d_Z$ between points in $Z$
and by $d_Y$ between points in $Y$.  The natural map $\varphi_Z: Z \to Z'$
is an isometric embedding.

If we assume further that $Y\setminus \psi\left(X\right)$ is locally convex then
 the natural map $f: Y \to Z'$ is a bijection onto its image which is a local
isometry on $Y\setminus \psi\left(X\right)$.
\end{lem}

We will say that $Z'$ is created by {\em attaching} $Y$ to $Z$
along $X$.  Note that  $f:Y \to Z'$ need not be an isometry.
This can be seen, for example, when $Z$ is the flat Euclidean plane, $X$ is the unit circle
in $Z$ and $Y$ is a hemisphere.  See Figure~\ref{fig-hemi-bump}.

\begin{figure}[h] 
   \centering
   \includegraphics[width=4.5in]{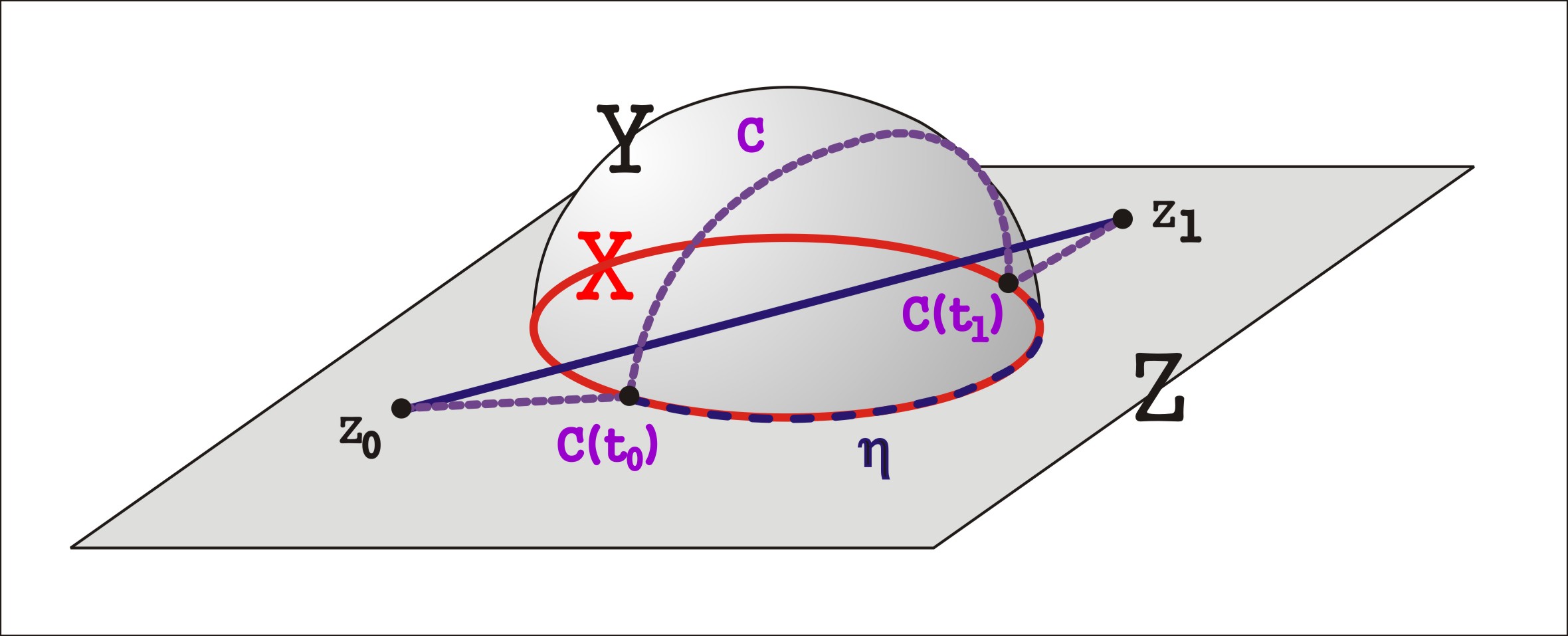} 
   \caption{Lemma~\ref{lem-local-filling} .}
   \label{fig-hemi-bump}
\end{figure}

\begin{proof}
First we show $\varphi_Z$ is an isometry. 
Let $z_0, z_1 \in Z$, so $d_Z\left(z_0, z_1\right)=L_Y\left(\gamma\right)$ where $\gamma:[0,1]\to Z$,
$\gamma\left(0\right)=z_0$ and $\gamma\left(1\right)=z_1$.   Since $\varphi_Z\circ \gamma$ runs
from $\varphi_Z\left(z_0\right)$ to $\varphi_Z\left(z_1\right)$ and has the same length, we know 
$d_{Z'}\left(\varphi_Z\left(z_0\right),\varphi_Z\left(z_1\right)\right) \le d_Z\left(z_0,z_1\right)$.  
Now suppose there is a shorter curve $C:[0,1]\to Z'$ running from
 $\varphi_Z\left(z_0\right)$ to $\varphi_Z\left(z_1\right)$.  If $C$ were the image of a curve in $Z$ under
 $\varphi_Z$, then $C$ would not be shorter than $\gamma$, so $C$ passes through
 $\varphi_Z\left(X\right)$ into $f\left(Y\right) \subset Z'$.  
 
 We claim there is
 a curve $C': [0,1] \to \varphi_Z\left(Z\right)$ running from $C\left(0\right)$ to $C\left(1\right)$
 with $L\left(C'\right)\le L\left(C\right)$, contradicting the fact that $\gamma$ is the shortest such curve.
 
 Since $Z\setminus X$ is open, $U=Z'\setminus \varphi_Z\left(Z\right)$ is open, and $C^{-1}\left(U\right)$
 is a collection of open intervals in $[0,1]$.  Let $t_0, t_1$ be any endpoints of a pair of
 such intervals so that $C:[t_0,t_1]\to f\left(Y\right)\subset Z'$ and $C\left(t_0\right), C\left(t_1\right)\in \varphi_Z\left(X\right) \subset Z$.
 Since $X$ isometrically embeds into $Y$, the shortest curve $\eta$ from
 $C\left(t_0\right)$ to $C\left(t_1\right)$ lies in $\varphi\left(X\right)$.  Thus we can replace this segment of $C$ with $\eta$
 without increasing the length.  We do this for all segments passing into $f\left(Y\right)$ and
 we have created $C'$ proving our claim.  Thus $\varphi_Z$ is an isometric embedding.
 
 Assuming now that $Y \setminus \psi(X)$ is locally convex, we know that
  $\forall p\in Y\setminus \psi\left(X\right)$ there exists a convex ball
$B_p\left(r_p\right)$.  We claim $f$ is an isometry on $B_p\left(r_p/2\right)$.  If $y_1, y_2 \in B_p\left(r_p/2\right)$
then the shortest curve between them, $\gamma$ has $L\left(\gamma\right)< r_p$ 
and lies in $B_p\left(r_p\right)$.  If there were a shorter curve, $C$, between $f\left(y_1\right)$ and
$f\left(y_2\right)$ in $Z$, then it could not be restricted to $f\left(Y\right)$ and in particular it would 
have to be long enough to reach $\partial B_p\left(r_p\right)$ and would thus have length
$L\left(C\right)\ge 2\left(r_p/2\right)$ which is a contradiction.
 \end{proof}
 
When we wish to isometrically embed two spaces with isometric subdomains into a common space $Z'$, 
we may attach them using an isometric product as a bridge between them.
Recall that the isometric product $Z\times [a,b]$ of a geodesic space, $Z$,
has a metric defined by
\be \label{defn-isom-product}
d\left(\left(z_1,s_1\right), \left(z_2,s_2\right)\right):=\sqrt{ \left(d_Z\left(z_1,z_2\right)\right)^2 + \left(s_1-s_2\right)^2},
\ee
and it is a geodesic metric space with this metric and a geodesic, $\gamma$,
projects to a geodesic, $\pi \circ \gamma$, in $Z$.  

\begin{lem}\label{lem-bridge-Z}
Suppose there exists an isometry, $\psi:U_1\subset M_1 \to U_2\subset M_2$,
between smooth connected open domains, $U_i$, in a pair of geodesic spaces, $M_i$,
each endowed with their own induced length metrics, $d_{U_i}$.
Let  
\be \label{choice-of-h}
h_i=\sqrt{\diam_{M_i}\left(\partial U_i\right)\left(2\diam_{M_i} \left(U_i\right)+ \diam_{M_i}\left(\partial U_i\right)\right)}.
\ee
Then there exist isometric embeddings $\varphi_i$ from each $M_i$ into a common
complete geodesic metric space,
\be
Z= M_1 \disjointunion \left(U_1 \times [-h_1, h_2] \right) \disjointunion M_2 \, / \sim \,,
\ee
where $z_1 \sim z_2$ if and only if  one of the following holds:
\be
z_1 \in U_1 \,\,\,\ and  \,\,\, z_2=\left(z_1,-h_1\right)\in U_1\times [-h_1,h_2] 
\ee
or visa versa or
\be
z_1 \in U_2 \,\,\,  and \,\,\,z_2=\left(\psi\left(z_1\right),h_2\right)\in U_1\times [-h_1,h_2] 
\ee
or visa versa.  The length metric on $Z$ is computed by taking the
lengths of segments from each region using $d_{M_i}$
and the product metric on $U_1\times [-h_1, h_2]$.  The isometries, $\varphi_i$
are mapped bijectively onto to the copies of $M_i$ lying in $Z$.
\end{lem}

We will say that we have joined $M_1$ and $M_2$ with the bridge $U_1\times [-h_1,h_2]$
and refer to the $h_i$ as the heights of the bridge.    See Figure~\ref{fig-isom-ex}.

\begin{figure}[h] 
   \centering
   \includegraphics[width=4.7in]{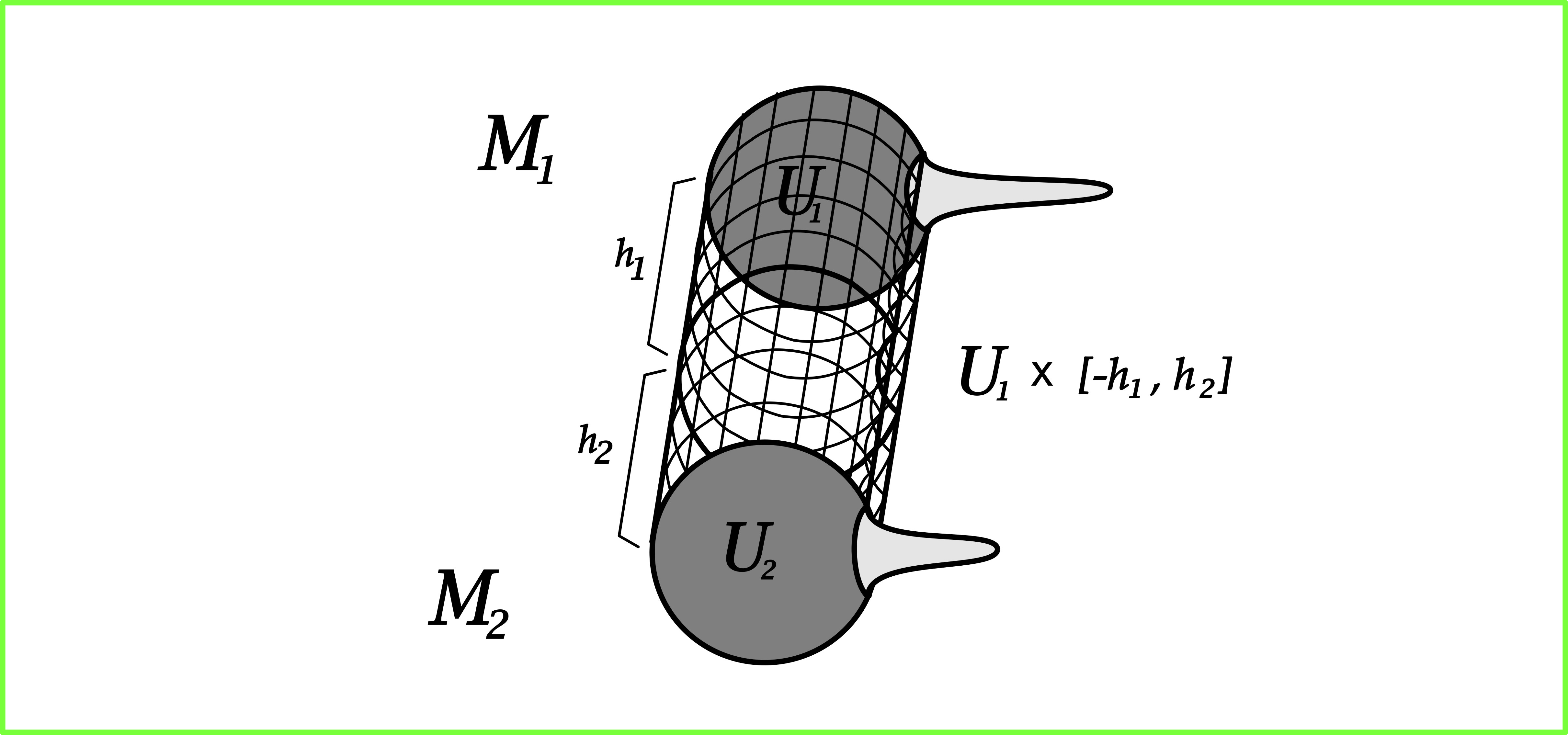} 
   \caption{The Bridge Construction [Lemma~\ref{lem-bridge-Z}].}
   \label{fig-isom-ex}
\end{figure}

\begin{proof}
Suppose $x,y\in M_1$, then there exists a geodesic $\gamma$ running
from $x$ to $y$ achieving the length between them, and clearly $\varphi_1\circ \gamma$
has the same length, so $d_{M_1}\left(x,y\right) \ge d_Z\left(\varphi_1\left(x\right),\varphi_1\left(y\right)\right)$.  Suppose on the contrary that $\varphi_1$ is not an 
isometric embedding.  So there is a curve $c:[0,1]\to Z$ 
running from $c\left(0\right)=\varphi_1\left(x\right)$ to $c\left(1\right)=\varphi_1\left(y\right)$ which is shorter than any curve running
from $x$ to $y$ in $M_1$.

If the image of $c$ lies in $M_1 \disjointunion \left(U_1 \times [-h_1, h_2] \right)\subset Z$, then the projection
of $c$ to $M_1$, $\pi\circ c$ would be shorter than $c$ and lie in $M_1$ and we would
have a contradiction.  Thus $c$ must pass into $M_2\setminus U_2 \subset Z$.

We divide $c$ into parts, $c_1$ runs from $\varphi_1\left(x\right)$ to $x'\in \partial \left(M_2\setminus U_2\right)$ 
and $c_3$ runs from a point $y'\in \partial \left(M_2\setminus U_2\right)$ 
to $\varphi\left(y\right)$ and $c_2$ lies between these.   Note that the projections
$\pi\left(x'\right)=\varphi_1\left(x"\right)$ and $\pi\left(y'\right)=\varphi_2\left(y"\right)$ where $x", y" \subset \partial U_1$.
Then
\begin{eqnarray}
\,\, L\left(c\right) &=& L\left(c_1\right) + L\left(c_2\right) + L\left(c_3\right)\\
&= & \sqrt{ L\left(\pi\circ c_1\right)^2 + \left(h_1+h_2\right)^2} + L\left(c_2\right) + \sqrt{ L\left(\pi\circ c_2\right)^2 + \left(h_1+h_2\right)^2}\\
& \ge & \sqrt{ L\left(\gamma_1\right)^2 + \left(h_1\right)^2}  + \sqrt{ L\left(\gamma_2\right)^2 + \left(h_1\right)^2}
\end{eqnarray}
where $\gamma_1$ is the shortest curve from $x$ to $x''$ and
$\gamma_2$ is the shortest curve from $y$ to $y''$ in $U_1\subset M_1$.  

By the definition of $h_i$ we know
\begin{eqnarray*}
L\left(\gamma_i\right)^2 +h_i^2 &=& L\left(\gamma_i\right)^2 +\diam\left(\partial U_i\right)\left(2\diam \left(U_i\right)+ \diam\left(\partial U_i\right)\right)\\
&\ge &L\left(\gamma_i\right)^2 +\diam\left(\partial U_i\right)\left(2L\left(\gamma_i\right)+ \diam\left(\partial U_i\right)\right)\\
&=& \left(L\left(\gamma_i\right) + \diam\left(\partial U_i\right) \right)^2.
\end{eqnarray*}
Thus
\be
L\left(c\right)  \ge  L\left(\gamma_1\right) + \diam\left(\partial U_1\right) + L\left(\gamma_2\right) + \diam\left(\partial U_1\right)
>L\left(\gamma_1\right)+L\left(\gamma_2\right)+\diam\left(\partial\left(U_i\right)\right).
\ee
Thus $c$ is longer than a curve lying in $M_1$ which runs from $x$ to $y$
via $x", y"\in \partial U_1$.  This is a contradiction.  We can similarly prove $\varphi_2$
is an isometric embedding.
\end{proof}

The difficulty with applying Lemma~\ref{lem-bridge-Z}, is that often $M_1$ and $M_2$
do not end up close together in the flat norm on $Z'$.  This can occur when 
$M_i\setminus U_i$
have large volume.  In the next proposition we combine this lemma with the prior
lemma to create a better $Z'$.

\begin{prop}\label{prop-bridge-filling}  
Suppose two oriented Riemannian manifolds with boundary, $M_i^m=\left(M_i,d_i,T_i\right)$
have connected open subregions, $U_i \subset M_i$, such that $T_i\rstr U_i \in \intcurr_m\left(M_i\right)$
and there exists an orientation preserving isometry, $\psi:U_1 \to U_2$.  Taking  
$V_i=M_i \setminus U_i$, 
and  geodesic metric spaces $X_i$ such that
\be
\psi_i:\left(V_i, d_{V_i}\right) \to \left(X_i, d_{X_i}\right)
\ee
are isometric embeddings and $X_i\setminus \psi_i\left(V_i\right)$ are locally convex.
Then if $B_i\in \intcurr_{m+1}\left(X_i\right)$ and $A_i\in \intcurr_m\left(X_i\right)$
with $\set\left(B_i\right), \set\left(A_i\right) \subset X_i \setminus \psi_i\left(V_i\right)$
satisfy
\be
\psi_{i\#}\left(T_i\rstr V_i\right)=A_i+\partial B_i. 
\ee
we have
\be
d_{\Fm}\left(M_1, M_2\right) \le \vol\left(U_1\right) \left(h_1+h_2\right) + \mass\left(B_1\right) 
+ \mass\left(B_2\right) +  \mass\left(A_1\right) +\mass\left(A_2\right)
\ee
where $h_i$ is as in (\ref{choice-of-h}) and
\be
d_{GH}\left(M_1, M_2\right) \le \left(h_1+h_2\right) + \diam\left(M_1\setminus U_1\right) + \diam\left(M_2\setminus U_2\right).
\ee
Note that when $X_i=V_i$, taking $B_i=0$ and $A_i=T\rstr V_i$ we have
\be
d_{\Fm}\left(M_1, M_2\right) \le \vol\left(U_1\right) \left(h_1+h_2\right) +  \vol\left(V_1\right) +\vol\left(V_2\right).
\ee
\end{prop}

See Figure~\ref{fig-flat-isom-ex}.  

\begin{figure}[h] 
   \centering
   \includegraphics[width=4.7in]{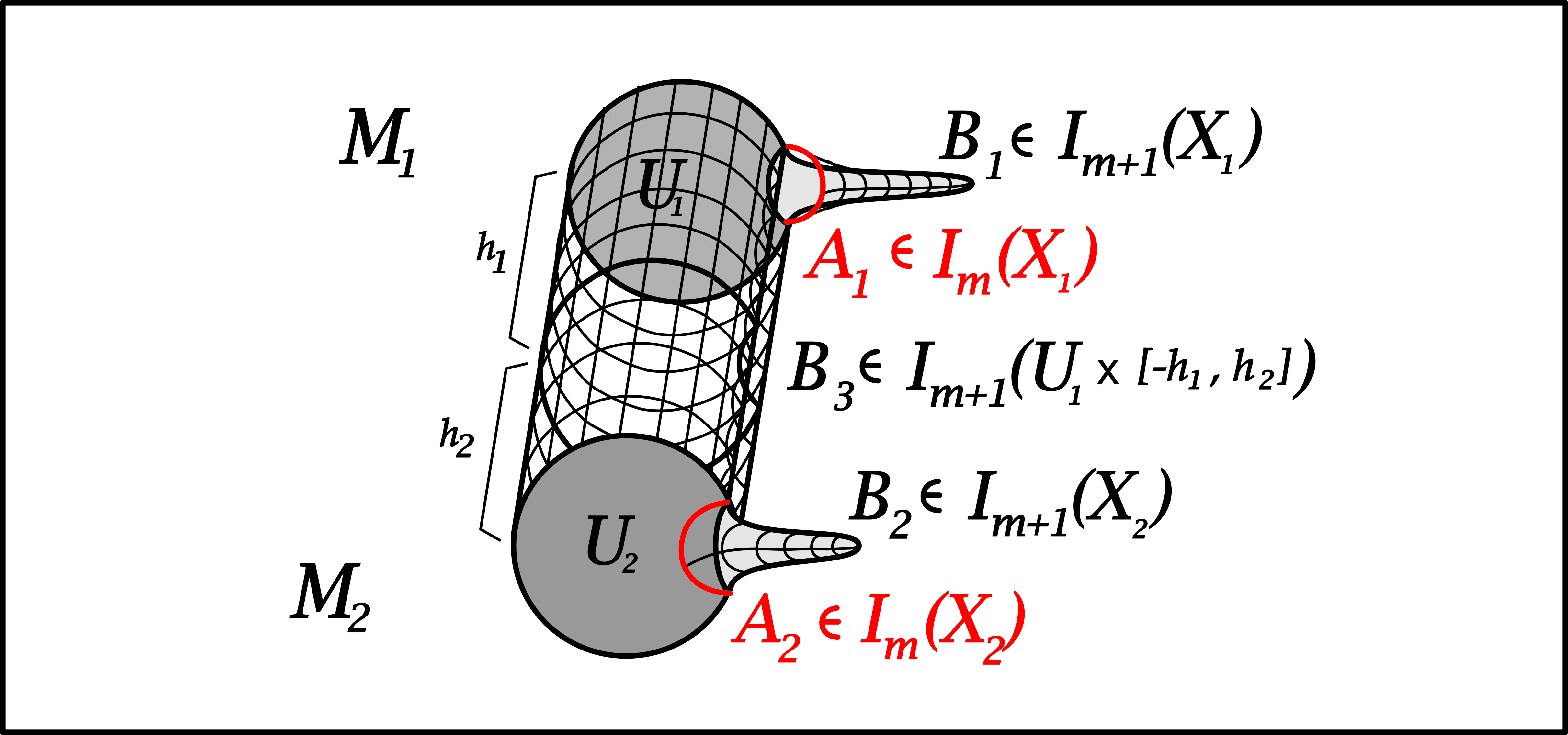} 
   \caption{Proposition~\ref{prop-bridge-filling}}
   \label{fig-flat-isom-ex}
\end{figure}

\begin{proof}
First we construct $Z$ exactly as in Lemma~\ref{lem-bridge-Z}.  
We obtain the estimate on the Gromov-Hausdorff
distance by observing that 
\be
d_H^{Z}\left(\varphi_1\left(M_1\right), \varphi_2\left(M_2\right)\right) \le \left(b_1+b_2\right) + \diam\left(M_1\setminus U_1\right) + \diam\left(M_2\setminus U_2\right).
\ee

To estimate the flat distance we construct
$Z'$ by applying Lemma~\ref{lem-local-filling}   to attach both $X_i$ to $Z$.
Note that $f_{i\#}B_i\in \intcurr_{m+1}\left(Z\right)$ and $f_{i\#}A_i\in \intcurr_m\left(Z\right)$
have the same mass as $B_i$ and $A_i$ respectively because $f_i: X_i -\psi_i\left(V_i\right)$
are locally isometries on $\set\left(B_i\right)$ and $\set\left(A_i\right)$.  Since $Z$ isometrically
embeds in $Z'$, the manifolds, $M_i$ 
are isometrically embedded and we will call the embeddings $\varphi'_i$.
Furthermore
\begin{eqnarray*}
 \varphi'_{1\#}T_1 - \varphi'_{2\#}T_2 
 &=& \varphi'_{i\#}\left(T_1\rstr V_1\right)-\varphi'_{2\#}\left(T_2\rstr V_2\right)\\
    & & +                            \varphi'_{i\#}\left(T_1\rstr U_1\right)-\varphi'_{2\#}\left(T_2\rstr U_2\right)\\
 &=& f_{1\#}A_1-f_{2\#}A_2 + f_{1\#}\partial B_1- f_{2\#}\partial B_2 + \partial B_3
 \end{eqnarray*}
 where $B_3\in \intcurr_{m+1}\left(Z\right)$ is defined as integration over
 $U_1 \times [-h_1, h_2]$ with the correct orientation.  Thus
 \be
 d_F^Z\left(\varphi_{1\#}T_1, \varphi_{2\#}\left(T_2\right) \right)\le \mass\left(B_3\right) 
 +\mass\left(B_1\right) +\mass\left(B_2\right)+\mass\left(A_1\right) +\mass\left(A_2\right) 
 \ee
 and we obtain the required estimate.
  \end{proof}


\subsection{Disappearing Tips and Ilmanen's Example} \label{subsect-example-ilmanen}

In this subsection we apply the bridge and filling techniques from the last subsection
to prove a few key examples.  We remark upon Gromov's square convergence
[Figure~\ref{figure-bumpy-hair}, Remark~\ref{rmrk-square}].  We
close with a proof that Ilmanen's Example depicted in Figure~\ref{fig-hairy-sphere}
does in fact converge in the intrinsic flat
sense [Example~\ref{ex-hairy-sphere}].  
Each example is written as a statement followed by a proof.

\begin{ex}\label{example-one-hair}
Let $M^m_j$ be spheres which have one increasingly thin tip 
as in Figure~\ref{figure-one-hair}.  In each $M_j$ there is a subdomain,
$U_j$, which is isometric to $U'_j= M_0\setminus B_p\left(r_j\right)$ where $M_0$
is the round sphere.  We further assume that $V_j=M_j\setminus U_j$
have $\vol\left(V_j\right) \to 0$.   We claim $M_j$ converges to $M_0$ in the intrinsic
flat sense.
\end{ex}

We prove this example converges with an explicit construction:

\begin{proof}
Since there is an isometry $\psi: U_j \to U_j'$ 
we join $M_j$ to the sphere $M_0$ with a bridge $U_j \times [-h_j, h_j']$
creating a metric space $Z$
as in Lemma~\ref{lem-bridge-Z} where $h_j, h_j' \to 0$ as $j \to \infty$.
Furthermore the isometric embeddings $\varphi_j: M_j \to Z$ and $\varphi_j':M_0\to Z$
push forward the current structures $T_j$ on $M_j$ and $T_0$ on $M_0$
so that
\be
\varphi_{j\#} T_j - \varphi'_j T_0 = \varphi_{j\#} \left(T_j \rstr U_j\right) - \varphi'_j \left(T_0 \rstr U'_j\right)
              + \varphi_{j\#} \left(T_j \rstr V_j\right) - \varphi'_j \left(T_0 \rstr V'_j\right)
\ee
where $V_j=M_j\setminus U_j$ and $V'_j=M_0\setminus U'_j$. 
We define $B_j \in \intcurr_3\left(Z\right)$
by integration over the bridge $U_j \times [-h_j, h_j']$ we have
\be
\varphi_{j\#} T_j - \varphi'_j T_0 = \varphi_{j\#} \left(T_j \rstr U_j\right)  - \varphi'_j \left(T_0 \rstr U'_j\right)+  \partial B_j 
\ee
 Note that  $\mass\left(T_j \rstr V_j\right)= \vol\left(V_j\right) \le 2/j^2$
and $\mass\left(T_0\rstr V_j'\right) $ both converge to $0$ as $j\to\infty$ and
$\mass\left(B_j\right) \le \vol\left(U_j\right)\left(h_j+h_j'\right)$ does as well because $\diam\left(\partial U_j\right)\to 0$
while $\diam\left(M_j\right)\le 4\pi$.   Thus $M_j$ converge to the sphere $M_0$ in the
intrinsic flat sense.
\end{proof}

\begin{rmrk} \label{rmrk-square}\label{rmrk-square-lambda}
The above example is similar to Gromov's Example on page 118 in \cite{Gromov-metric}.
The $\square_\lambda$ limit agrees with the flat limit for $\lambda>0$.  The
Gromov-Hausdorff limit of this sequence is the sphere with a unit length segment
attached.
Gromov points out that if $M_j \GHto M_\infty$ and $p_j \in M_j \to p_\infty\in M_\infty$
have a uniform positive lower bound on the measure of $B_{p_j}\left(1\right)$
the $\square_1$ limit of the $M_j$ which is a subset of $M_\infty$
includes $p_\infty$.  This is not true for the intrinsic flat limit
as can be seen in the following example.
\end{rmrk}

\begin{figure}[h] 
   \centering
   \includegraphics[width=4.7in]{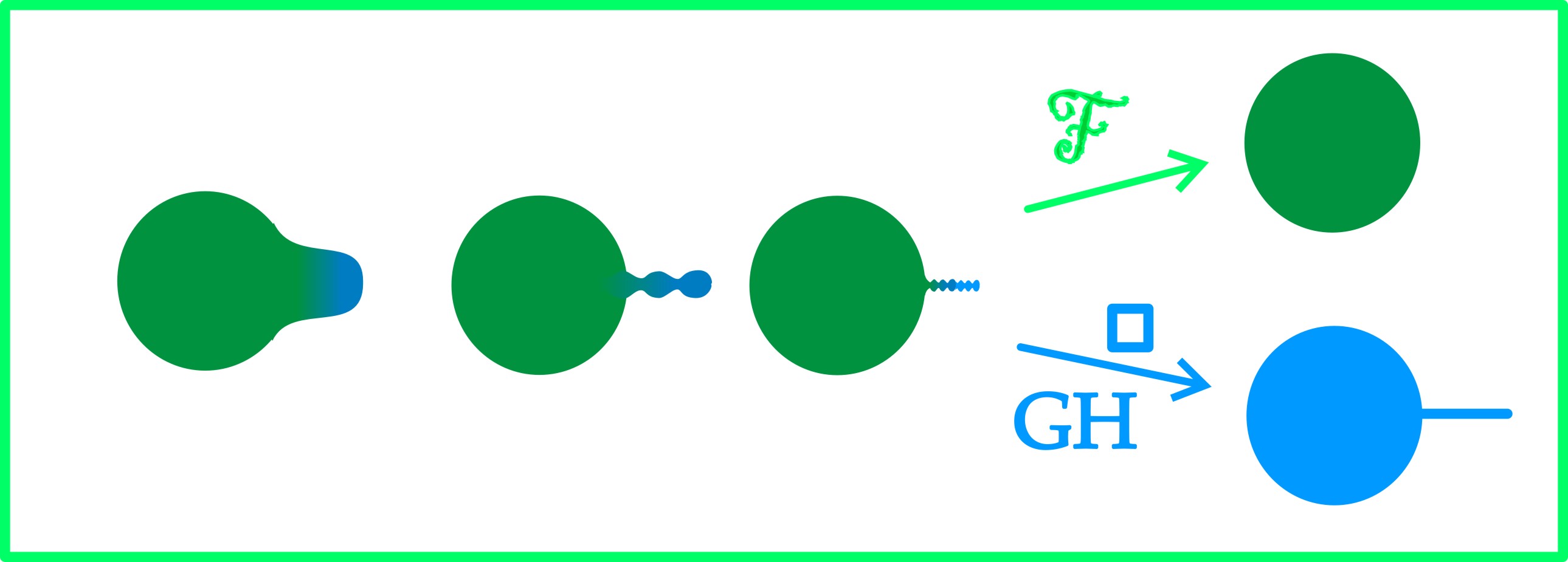} 
   \caption{Contrasting with Gromov's square limit.}
   \label{figure-bumpy-hair}
\end{figure}

\begin{ex}\label{ex-bumpy-hair}
Let $M^m_j$ be spheres which have one increasingly thin tip with uniformly
bounded volume  as in Figure~\ref{figure-bumpy-hair}.  
In each $M_j$ there is a subdomain
$U_j$ which is isometric to $U'_j= M_0\setminus B_p\left(r_j\right)$ where $M_0$
is the round sphere.  We further assume that $V_j=M_j\setminus U_j$
have $\vol\left(V_j\right)$ decreasing but $\ge  V_0 >0$ while $V_j$ 
converge in the Gromov-Hausdorff sense to a line segment.   
Then $M_j $ converges to $M_0$ in the intrinsic flat sense.
\end{ex}

\begin{proof}  
Since there is an isometry $\psi$ from $U_j$ to  $U'_j$,
we join $M_j$ to the sphere $M_0$ with a bridge $U_j \times [-h_j, h_j']$
creating a metric space $Z$
as in Lemma~\ref{lem-bridge-Z} where $h_j, h_j' \to 0$ as $j \to \infty$.
Furthermore the isometric embeddings $\varphi_j: M_j \to Z$ and $\varphi_j':M_0\to Z$
push forward the current structures $T_j$ on $M_j$ and $T_0$ on $M_0$

By Corollary~\ref{lower-dim} and $V_j\GHto [0,1]$, we know $\left(V_j, d_j,T_j \rstr V_j\right)$ 
converges to $0$ as an integral current space.
By Theorem~\ref{inf-dist-attained}, there is a metric space $X_j$ with an isometry $\phi_j: V_j \to X_j$
and integral currents $A_j, B_j$ such that $\phi_{j\#}\left(T\rstr V_j\right) = A_j +\partial B_j$
such that $\mass\left(A_j\right) +\mass\left(B_j\right) \to 0$.  We now apply Proposition~\ref{prop-bridge-filling},
attaching $X_j$ to $Z$ to create $Z'$, and we have
\be
d_{\Fm}\left(M_j, M_0\right) \le \vol\left(U_j\right) \left(h_j+h'_j\right) + \mass\left(B_j\right) + \mass\left(A_j\right) +\vol\left(M_0\setminus U'_j\right) \to 0.
\ee
Note that here we did not bother with two fillings as in the proposition.
\end{proof}

We now prove Ilmanen's example in Figure~\ref{fig-hairy-sphere} converges
to a standard sphere in the intrinsic flat sense.  Although Ilmanen's sequence
of examples have positive scalar curvature and are three dimensional, here
we show convergence in any dimension including two.

\begin{ex}\label{ex-hairy-sphere}
We assume $M_j$ are diffeomorphic to
spheres with a uniform upper bound on volume and  that each $M_j$
contains a connected open domain $U_j$ which is isometric to a domain
$U'_j= M_0 \setminus \bigcup_{i=1}^{N_j} B_{p_{j,i}}\left(R_j\right)$ where 
$M_0$ is the round sphere and $B_{p_{j,i}}\left(R_j\right)$ are pairwise disjoint.
We assume  that  each connected component, $U_{j,i}$
of $V_j =M_j\setminus U_j$ 
and each ball $B_{p_{j,i}}\left(R_j\right)$ has volume $\le v_j/N_j$
where $v_j \to 0$.   Then $M_j$ converges to a round sphere
in the intrinsic flat sense as long as $N_j \sqrt{R_j}\to 0$.
\end{ex}

\begin{proof}
We cannot directly apply Proposition~\ref{prop-bridge-filling}   in this setting because
$\diam\left(\partial U_j\right)$ are not converging to $0$.  So instead of building a
bridge $Z$ directly from $M_j$ to $M_0$, we build bridges from
$M_0=M_{j,0}$ to $M_{j,1}$ to $M_{j,2}$ and up to $M_{j, N_j}=M_j$ by adding
one bump at a time.  Each pair has only one new bump and so we can show
\be
d_{\Fm}\left(M_{j,i}, M_{j,i+1}\right) 
\le  \vol\left(U_{j,i}\right) \left(h_{j,i}+h'_{j,i}\right) + 2 v_j/N_j
\ee
where 
\begin{eqnarray*}
h_{i,j}, h'_{i,j} &\le&  \sqrt{\diam\left(\partial U_{j,i+1}\right)\left(\diam\left(M_{j,i}\right) +\diam\left(\partial U_{j,I+1}\right)\right) } \\
&\le& \sqrt{\pi R_j\left( \diam\left(M_j\right) +\pi R_j\right)}.
\end{eqnarray*}
Summing from $i=1$ to $N_j$ we see that 
\be
d_{\Fm}\left(M_0, M_j\right)\le \vol\left(U_j\right) 2 \sqrt{\pi R_j\left( \diam\left(M_j\right) +\pi R_j\right)} + 2v_j \to 0.
\ee
\end{proof}


\subsection{Limits with Point Singularities}

Recall that when defining an integral current space, $(X,d,T)$,
we required that $\set(T)=X$ so that all points in the space have
positive density [Defn~\ref{defn-set}  , Defn~\ref{defn-current-space}].  
In this subsection, we present two related examples.

\begin{ex}\label{ex-to-cone}
In  Figure~\ref{fig-to-cone} we have a sequence of Riemannian surfaces, $M_j$,
diffeomorphic to the sphere converging in the intrinsic flat sense
to a Lipschitz manifold, $M_0$, with a conical singularity.  
Since this sequence clearly converges in the Lipschitz sense
to $M_0$, this is proven by applying Theorem~\ref{thm-lip-to-flat}.
\end{ex}

\begin{figure}[h] 
   \centering
   \includegraphics[width=4.7in]{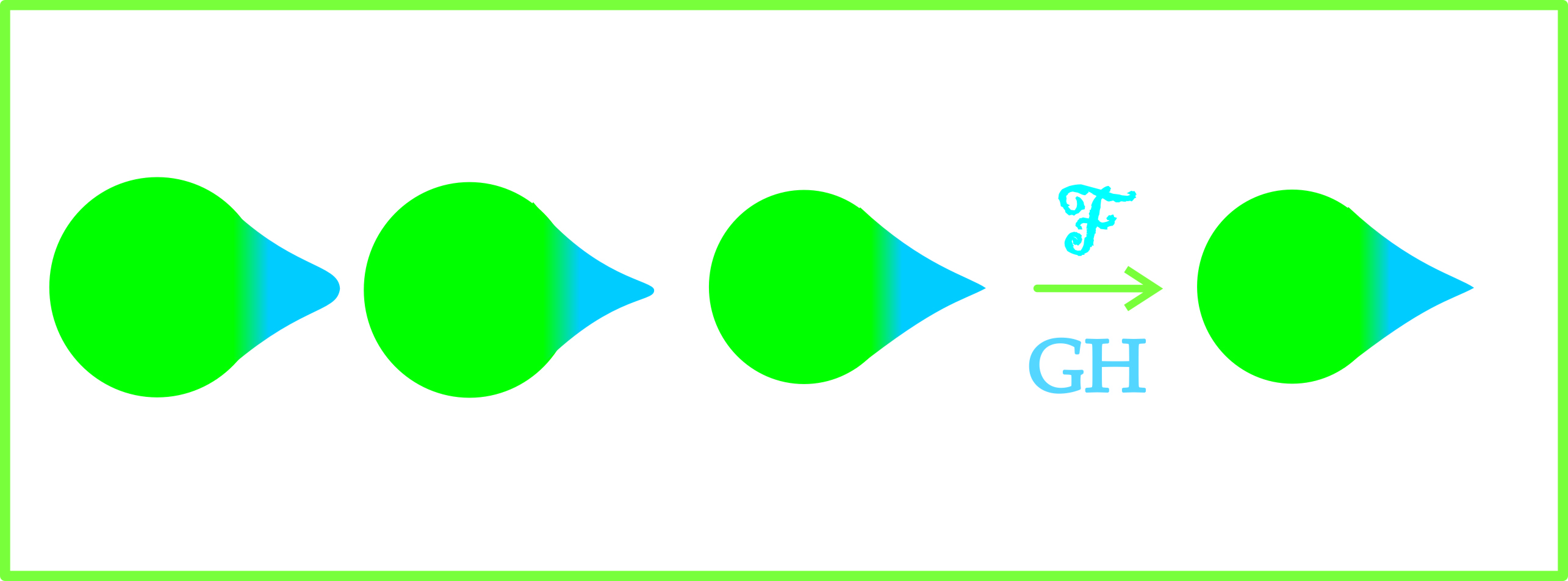} 
   \caption{The intrinsic flat limit does include the tip of the cone.}
   \label{fig-to-cone}
\end{figure}

\begin{ex}\label{ex-to-cusp}
In  Figure~\ref{fig-to-cusp} we see a sequence of Riemannian surfaces, $M_j$, 
diffeomorphic to the sphere converging
in the intrinsic flat sense
to a Riemannian manifold, $M_\infty$, with a cusp singularity.  The cusp singularity is
not included in the limit current space because we only include points of
positive lower density.
\end{ex}

 There is no Lipschitz convergence here
even if we were to include the cusp point so we prove this example:

\begin{proof} 
Note that $M_\infty$ is a geodesic space because no minimizing curves pass
over a cusp point.  So we apply Lemma~\ref{lem-bridge-Z} to build a bridge
$Z$ between $M_j$ and $M_\infty$ removing small balls, $V_j$ near their tips
so that $U_j=M_j\setminus V_j$ are locally isometric.
Now we apply Proposition~\ref{prop-bridge-filling} with $X_i=V_i$
which works even though $M_\infty$ has a point singularity because 
$\mass \left(V_\infty\right)=\vol\left(V_\infty\right)$.  So we have:
\be
d_{\Fm}\left(M_j, M_\infty\right) \le \vol\left(U_j\right) h_j +\vol\left(V_j\right) +\vol\left(V_\infty\right)
\ee
where $h_j=\sqrt{\diam\left(\partial U_{j}\right)\left(\diam\left(M_{j}\right) +\diam\left(\partial U_{j}\right) \right)} $
\end{proof}

\subsection{Limits need not be Precompact}

In this subsection,
we present a pair of integral current spaces which are not precompact and yet are the limits
of a sequence of Riemannian surfaces diffeomorphic to the sphere with a uniform upper
bound on volume.  Example~\ref{ex-unbounded} is not bounded and is a classic surface
of revolution of finite area.  Example~\ref{ex-many-tips} depicted in Figure~\ref{fig-many-tips}
is the limit of a sequence with a uniform upper bound on diameter and is bounded but has
infinitely many tips. 

\begin{ex}  \label{ex-unbounded}
Let $M_0$ be the surface of revolution in Euclidean space defined by
\be
M_0=\{(x,y,z): \,\,x^2+y^2= 1/(1-z)^4, \, z\ge 0\} \subset E^3
\ee
with the outward orientation and the induced Riemannian length metric.
Since $M_0$ has finite area and its boundary has finite length,
it is an integral current space.  

Let 
\be
M_j=\{(x,y,z): \,\,x^2+y^2= f_j(z)/(1-z)^4, \, z \ge 0\} \subset E^3
\ee
where $f_j(z)=1$ for $z\le j$ and such that $f_j(z)=0$ for $z\ge j+1/j$ and smoothly
decreasing between these values so that $M_j$ is smooth at $z=j+1/j$.  We
also orient $M_j$ outward and give it the induced Riemannian length metric.
Note that $\diam(M_j) \to \infty$ so $M_j$ is not Cauchy in the Gromov-Hausdorff
sense.  However $M_j$ converges to $M_0$ in the intrinsic flat sense.  
\end{ex}

\begin{proof}
Note that $U_j\in M_j$ be defined as $M_j \cap \{z\in[0,j]\}$ is locally isometric
to $U'_j \in M_0$ defined by $M_0\cap \{z\in [0,j]\}$.  
We join $M_j$ to the sphere $M_0$ with a bridge $U_j \times [-h_j, h_j']$
creating a metric space $Z$ where $h_j, h_j' $ are bounded by
\be 
\sqrt{ \frac{\pi}{(1-j)^4}  \left(2  (2j)+   \frac{\pi}{(1-j)^4}  \right)} \to 0 \textrm{ as } j\to \infty.
\ee
From here onward we may apply Proposition~\ref{prop-bridge-filling}
using the fact that $V_j=M_j\setminus U_j$
and $V'_j=M_0\setminus U'_j$ both have area converging to $0$.
\end{proof}

\begin{ex}\label{ex-many-tips}
The sequence of Riemannian manifolds $M_j$ in Figure~\ref{fig-many-tips} is defined by taking
a sequence of $p_j$ lying on a geodesic in the sphere $M_0$ 
converging to a point $p_\infty$
and choosing balls $B_{p_j}\left(r_j\right)$ that are disjoint.  The tips
are Riemannian manifolds, $N_j$ with boundary such that $\partial N_j$ is
isometric to $\partial B_{p_j}\left(r_j\right)$ and $N_j$ can be glued smoothly to
$M_0\setminus B_{p_j}\left(r_j\right)$.  We further require that $\diam\left(N_j\right)\le 2$
and $\vol\left(N_j\right) \le \left(1/2\right)^j$.  Then $M_j$ is formed by removing
the first $j$ balls from $M_0$ and gluing in the first $j$ tips, $N_1, N_2,...N_j$, with the
usual induced Riemannian length metric:
\be
M_j := \left(M_0 \setminus \bigcup_{i=1}^j B_{p_i}(r_i)\right) \disjointunion N_1 \disjointunion 
N_2 \disjointunion \cdots
\disjointunion N_j.
\ee
So the diameter and volume
of $M_j$ are uniformly bounded above.
  
The intrinsic flat limit $M_\infty$ is defined by removing all the balls and gluing in
all the tips:
\be
M_\infty := \left(M_0 \setminus \bigcup_{i=1}^j B_{p_i}(r_i)\right) \disjointunion N_1 \disjointunion 
N_2 \disjointunion \cdots
\ee
so that $M_\infty$ is not smooth at $p_\infty$ but it is a countably
$\mathcal{H}^m$ rectifiable space.  There are natural current structures
$T_j$ and $T_\infty$ on these spaces with weight 1 and orientation
defined by the orientation on $M_0$.   Note that $M_\infty$ has finite volume
and diameter but is not precompact
because it contains infinitely many disjoint balls of radius $1$.
\end{ex}

\begin{proof}
Let $\epsilon_j=d_{M_0}\left(p_j,p_\infty\right)$.  Then 
there is an isometry $\psi: U_j \to U_j'$ where
$U_j =M_j \setminus B_{p_\infty}\left(\epsilon_j -r_j\right) $
and $U_j'\subset M_\infty$.  So we
join $M_j$ to  $M_\infty$ with a bridge $U_j \times [-h_j, h_j']$
creating a metric space $Z$
as in Lemma~\ref{lem-bridge-Z} where $h_j, h_j' \to 0$ as $j \to \infty$.
Furthermore the isometric embeddings $\varphi_j: M_j \to Z$ and $\varphi_j':M_\infty\to Z$
push forward the current structures $T_j$ on $M_j$ and $T_\infty$ on $M_\infty$
so that
\be
\varphi_{j\#} T_j - \varphi'_j T_\infty = \varphi_{j\#} \left(T_j \rstr U_j\right) - \varphi'_j \left(T_\infty \rstr U'_j\right)
              + \varphi_{j\#} \left(T_j \rstr V_j\right) - \varphi'_j \left(T_\infty \rstr V'_j\right)
\ee
where $V_j=M_j\setminus U_j$ and $V'_j=M_\infty\setminus U'_j$. 
 Letting $B \in \intcurr_{m+1}\left(Z\right)$
to be defined by integration over the bridge $U_j \times [-h_j, h_j']$ we have
\be
\varphi_{j\#} T_j - \varphi'_j T_0 = \partial B_j+ \varphi_{j\#} \left(T_j \rstr V_j\right)  - \varphi'_j \left(T_\infty \rstr V'_j\right) 
\ee
However
\be
\begin{split}
&\mass\left(T_j \rstr V_j\right)= \vol\left(V_j\right) \le \omega_m\left(\epsilon_j-r_j\right)^m \to 0 \textrm{ and}
\\ 
&\mass\left(T_\infty\rstr V_j'\right) \le \sum_{i=j+1}^\infty \frac{1}{2^j} \to 0 \textrm{ and}
\\
&\mass\left(B_j\right) \le \vol\left(U_j\right)\left(h_j+h_j'\right) \to 0
\end{split}
\ee
because $\diam\left(\partial U_j\right)\to 0$
while $\diam\left(M_j\right)\le \pi+2$.   
Thus $M_j$ converge to the sphere $M_\infty$ in the intrinsic
flat sense.
\end{proof}

\subsection{Pipe Filling and Disconnected Limits}


In this subsection we study sequences of Riemannian manifolds which
converge to spaces which are not geodesic spaces.  Our examples
consist of spheres joined by cylinders where the cylinders disappear
in the intrinsic flat limit.    For these examples
we cannot just apply Lemma~\ref{lem-bridge-Z} because we do not
have connected isometric domains.    

We develop a new concept
called "pipe filling" [See  Remark~\ref{rmrk-pipe-filling}] .  Note that
a cylinder, $S^{m-1} \times [0,1]$, does not 
isometrically embed into a solid Euclidean cylinder, $D^m \times [0,1]$,
but that it does isometrically embed into a cylinder of higher dimension
$S^m \times [0,1]$.
We prove Example~\ref{example-not-length} depicted in Figure~\ref{figure-not-length},
and Example~\ref{example-gym2}  depicted in Figure~\ref{example-gym2}.  

\begin{example}\label{example-not-length}  
The sequence of manifolds in Figure~\ref{figure-not-length} are
smooth manifolds, $M'_j$, which are
bi-Lipschitz close to Lipschitz manifolds,
\be
M_j=\{\left(x,y,z\right): \,\, x^2+z^2 = f_j^2\left(y\right), \, y\in [-3, 3]\}, 
\ee
where $f_j\left(y\right)$ is a smooth function such that 
\be\label{eqn-pipe1}
f_j\left(y\right):= \sqrt{ 1-\left(y+2\right)^2} \textrm{ for } 
y \in [-3, -2+\sqrt{1-\left(1/j\right)^2}],
\ee
\be \label{eqn-pipe2}
f_j\left(y\right):= \sqrt{ 1-\left(y-2\right)^2} \textrm{ for } y \in [2-\sqrt{1-\left(1/j\right)^2}, 3] 
\ee
and $f_j\left(y\right)=1/j$ between these two intervals.
For $j=\infty$ we let $f_\infty\left(y\right)$ satisfy (\ref{eqn-pipe1}) and (\ref{eqn-pipe2})
and $f_\infty(y):=0$ between the two intervals so that $M_\infty$ is two spheres joined 
by a line segment.  

All $M_j$ for $j=1,2,3...$ are endowed
with geodesic metrics and outward orientations.  
Then $M_j$ Gromov-Hausdorff converges to the connected geodesic space
$M_\infty$ but converges in the intrinsic flat sense to two disjoint spheres,
 $N_\infty=\left(\set\left(T_\infty\right), d_{M_\infty}, T_\infty\right)$ where $T_\infty\in \intcurr_2\left(M_\infty\right)$ 
 is integration over the spheres.  Since $d_{lip}\left(M_j, M'_j\right)\to 0$
we also have a sequence of Riemannian manifolds converging to 
this disconnected limit space.   
\end{example}

\begin{proof}
We construct a common metric space $Z_j$ as in Figure~\ref{fig-pipe-filling}.  More precisely,
\be
Z_j=\{\left(x,y,z,w\right): \,\, x^2+z^2 = \bar{f}_j^2\left(y,w\right), \, y\in [-3, 3], \,\, w\in [0,1/j]\} 
\ee
where 
\be
\bar{f}_j\left(y,w\right)= \max\left\{ f_j\left(y\right)\sqrt{1-j^2w^2},  f_\infty\left(y\right)\right\} 
\ee
with the induced length metric from four dimensional Euclidean space.
$Z_j$ is roughly two spheres of radius 1 crossed with intervals, $S^2\times [0,1/j]$,
with a thin half cylinder, $S^2_{+,1/j} \times [-1,1]$, between them.   This
half cylinder is filling in the thin cylinder in $M_j$ and is the key
step in the pipe filling construction.

\begin{figure}[h] 
   \centering
   \includegraphics[width=4.7in]{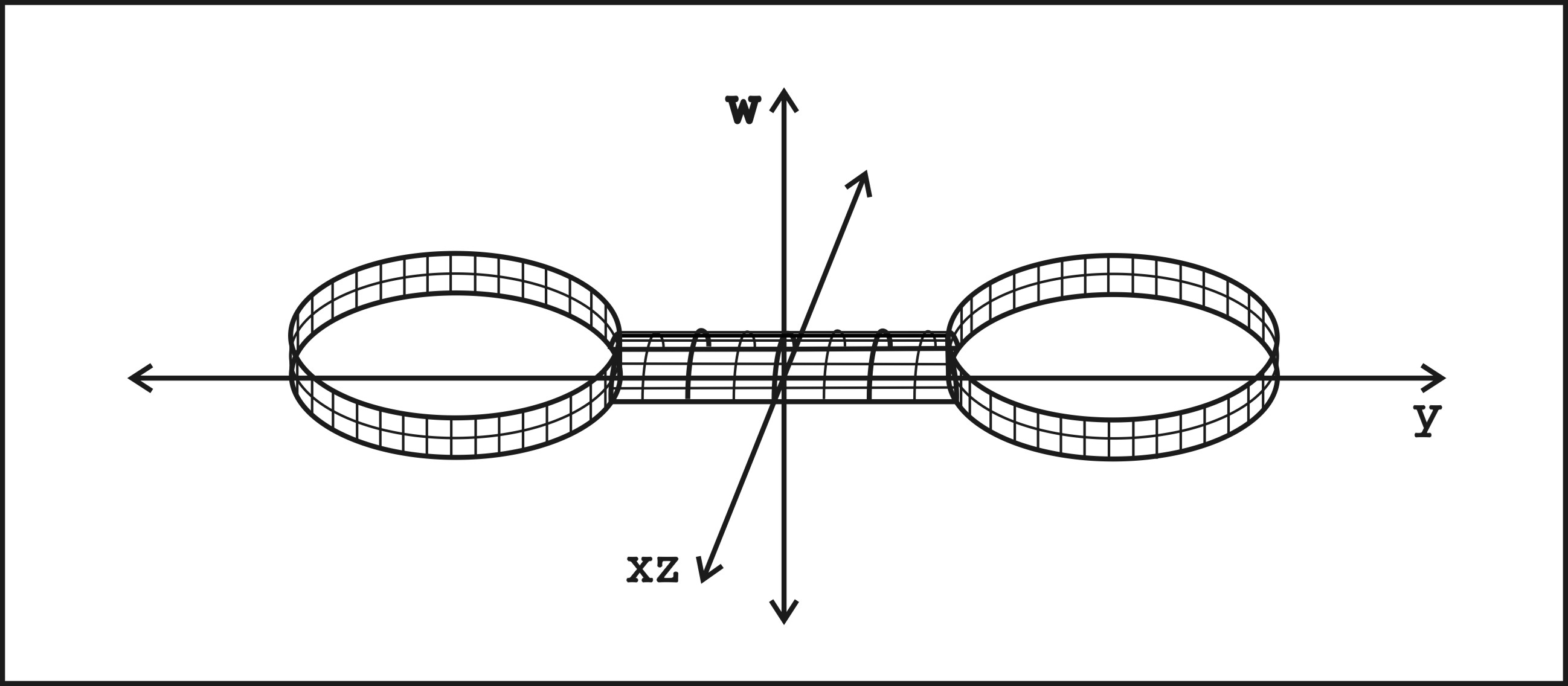} 
   \caption{Here the cylinder in the xzy plane is filled in by a half cylinder.}
   \label{fig-pipe-filling}
\end{figure}

It is easy to see that $\varphi_\infty: M_\infty \to Z_j$ such that
$\varphi_\infty\left(p\right)=\left(p, 1/j\right)$ is an isometric embedding because there
is a distance nonincreasing retraction from $Z_j$ to the level set $w=1/j$.   It
is also an isometric embedding when restricted to $N_\infty$.

Regretably $\varphi_j:M_j\to Z_j$ with $\varphi_j\left(p\right)=\left(p,0\right)$ is not 
isometric embedding.  It preserves distances between points which both lie in
one of the balls or both lie in the thin cylinder, but not necessarily between points
in different regions.  Let
\be
h_j=\sqrt{\pi/j + (\pi/(2j))^2}
\ee
and glue $M_j \times [-h_j,0]$ to $Z_j$ to create $Z_j'$ which can be viewed
as a metric space lying in four dimensional Euclidean space with the induced 
intrinsic length metric.  Then $\varphi'_j: M_j \to Z_j'$ where
$\varphi_j'(p)=(p,-h_j)$ is an isometry.  Any minimizing curve in $Z_j'$ between
points $(p,-h_j)$ and $(q, -h_j)$ can either be retracted down to the $w=-h_j$
level, or it must travel up at least to the $w=0$ level.  So the curve has length
$\sqrt{l_1^2+h_j^2}$ before reaching $w=0$ and then travels some distance, $l_2$, 
within the half thin cylinder and then come back down with length $\sqrt{l_3^2+h_j^2}$.
However a curve lying in the $w=-h_j$ level set would travel only $l_1$ then
$l_2$ in the thin cylinder, then $\pi r$ around the thin cylinder, and then $l_3$
to its endpoint.  However
\be
\sqrt{l_1^2+h_j^2}\,\, + \,\,l_2\,\,+\,\,\sqrt{l_3^2+h_j^2} \,\,\,\ge\,\,\, l_1+l_2+l_3+\pi r
\ee
by our choice of $h_j $.  Thus $\varphi'_j$ is an isometric embedding.

Now $Z'_j$ has a naturally defined current structure $B_j$
such that $\mass(B_j)=\vol(Z_j')$
and such that $\partial B_j= \varphi_{j\#}T_j -\varphi_{\infty\#}T_\infty$.  
So we have $M(B_j)$ equal to
\be
 2 \vol\left(S^2\times [-h_j,1/j]\right) + 
\frac{1}{2} \vol\left([-1,1]\times S^2_{1/j}\right) 
+\vol\left([-1,1]\times S^1_{1/j}\right)h_j 
\ee
and thus
\be
d_F^{Z_j}\left(\varphi_{j\#}T_j , \varphi_{\infty\#}T_\infty\right)\le 
\mass\left(B_j\right)=\vol\left(Z_j\right)\to 0.
\ee
Furthermore, it is easy to see that
\be
d_{GH}\left(M_j , M_\infty\right) \le 
d_{H}^{Z_j}\left(\varphi_{j}\left(M_j\right) ,\varphi_{0}\left(M_\infty\right)\right)\le \frac{\pi}{2j}.
+ h_j \to 0.
\ee
So $M_j$ converge in the intrinsic flat sense to $N_\infty$
but in the Gromov-Hausdorff sense to $M_\infty$.
\end{proof}

\begin{rmrk} \label{rmrk-pipe-filling}
The process used in Example~\ref{example-not-length}
can be used more generally to show an integral current space $M$
which is collection of $k$ disjoint spheres, $S^m_{R_j}$, of radius $R_j\le R$
for $j=1..k$
connected by $n$ cylinders $S^{m-1}_{r_i}\times [0,L_i]$ of length $L_i\le L$ and 
radius $r$ for $i=1..n$
between them  is close to an integral current space $N$ which 
is defined by integration over
the same collection of spheres with the metric restricted
from the metric space $X$ which is the same
collection of spheres joined by line segments of length
$L_i$ rather than cylinders.  

More precisely, one can construct a $Z$
by gluing together the collection of $S^m_{R_j}\times[0,\pi r/2]$ together
with thin half cylinders of radius $r$ and length $L_i$, and then 
take $h=\sqrt{\pi r R +(\pi r/2)^2}$, and define $Z'$ by attaching
$M\times [-h,0]$ to $Z$.  Thus the Gromov-Hausdorff distance
\be
d_{GH}\left(M,X\right) \le \pi r+h
\ee
and the intrinsic flat distance can be estimated using the volume
of $Z'$.  In particular,
\be\label{eq-pipe-filling}
d_{\Fm}\left(M,N\right) \le  V (r+h) + 
\vol_{m-1}\left(S^{m}_r\right)L/2 +\vol_m\left(S^{m-1}_r\right) L h ,
\ee
where $L=\sum_{i=1}^k L_i$ and $V=\sum_{j=1}^n \vol_m\left(S^m_R\right)$.  
Note that if one has $r\to 0$, the product
$r^{m-1/2}L \to 0$ and $R$ and $V$ are uniformly bounded above,
then the right hand side of (\ref{eq-pipe-filling})  goes to $0$.
We will call this {\em pipe filling}.  
\end{rmrk}

\begin{ex}\label{example-gym2}  
In Figure~\ref{figure-gym2} we have an example of a sequence of
Riemannian manifolds, $M'_j$, which are collections of spheres of various
sizes joined by cylinders, which converge in the intrinsic flat sense
to a compact integral current space $N_\infty$ consisting of
countably many spheres oriented outward whose
metric is restricted from the Gromov Hausdorff limit, $X_\infty$,
formed by joining the spheres in $N_\infty$ with line segments.
The explicit inductive construction is given in the proof.
\end{ex}

\begin{proof}
We begin the inductive construction of the collection of spheres $N_j$
used to build the Riemannian manifolds, $M'_j$.
Let $N_0$ be four disjoint spheres of radius $R_0$ lying in Euclidean space
whose centers form a square of side length $L_0+2R_0$.

To build $N_j$,  we first rescale $N_{j-1}$ 
by a factor of 3 and make 5 copies, then place them symmetrically around $N_0$,
thus creating $N_{j}$ where $R_j=R_0/3^j$ is the radius of the smallest sphere
and 
\be
Vol\left(N_j\right) =\frac{5}{3^2} Vol\left(N_{j-1}\right) + Vol\left(N_0\right)= \sum_{i=0}^j \left(\frac{5}{9}\right)^j Vol\left(N_0\right)\le\frac{9}{4} Vol\left(N_0\right).  
\ee
Now $M_j$ is built by joining the spheres in $N_j$ with cylinders
of radius $\epsilon_j<<R_j$ chosen so that the total length $L_j$ of all
the cylinders satisfies $\epsilon_jL_j < 1/j$ and  $\lim_{j\to\infty} \vol\left(M_j\right) = 9\vol\left(N_0\right)/4$.
We give $M_j$ the outward orientation and note that there are Riemannian
manifolds $M_j'$ arbitrarily close to $M_j$ in the Lipschitz sense
who will have the same intrinsic flat and Gromov-Hausdorff 
limits as $M_j$ by Theorem~\ref{thm-lip-to-flat}.

Let $X_j$ be created by joining the $N_j$ with line segments and
give $X_j$ the induced length metric so that it is a geodesic metric space.  
Let $X_\infty$ be the union of all these metric spaces, which is also a
compact geodesic metric space with the induced length metric.
The integral current space $N_\infty$ is defined
as the union of all the $N_j$ with the metric $d_\infty$ restricted
from the length metric on $X_\infty$.

Note that for any $\varepsilon>0$, we can find $j$ sufficiently large
that $d_{GH}(M_j,X_j)<\epsilon$ and $d_{\mathcal{F}}(M_j, N_j)<\epsilon$.
This can be seen by creating a pipe filling from $M_j$ to $X_j$ as in Remark~\ref{rmrk-pipe-filling} with $r=\epsilon_j$ $L=L_j$, $R=1$ and
$V=\frac{9}{4} Vol\left(N_0\right)$.

Next we observe that the  maps $\psi_j: X_j \to X_\infty$ are isometric embeddings
because paths between points in $\psi_j\left(X_j\right)$ are shorter if they
stay in $\psi_j\left(X_j\right)$.  Thus
\begin{eqnarray*}
d_{\Fm}\left(N_j, N_\infty\right) &\le& d^{X_\infty}_F\left(\psi_{j\#} \Lbrack N_j  \Rbrack,  N_\infty\right)  \\
&\le & \mass\left(\psi_{j\#} \Lbrack N_j\Rbrack- N_\infty \right) \\
&\le &\sum_{i=j+1}^\infty \left(\frac{5}{9}\right)^j Vol\left(N_0\right)\,\,\to \,\,0.
\end{eqnarray*}
Since $X_\infty \subset T_{R_j}(\psi_j(X_j))$
where $R_j \to 0$ 
we see that $d_H\left(\psi_j\left(X_j\right), X_\infty\right) \to 0$.  Combining this
with our pipe filling estimates above, we see that the integral
current spaces $M_j$ converge to $N_\infty$ in the
intrinsic flat sense and to $X_\infty$ in the Gromov-Hausdorff sense.
\end{proof}

\begin{rmrk} \label{rmrk-pipe-filling-1}\label{rmrk-pipe-filling-adapt}
Note that in the pipe filling construction described in Remark~\ref{rmrk-pipe-filling},
one might have a single sphere with many thin cylinders looping around
and back to it.  One does not need to view the space as a subset of Euclidean
space.  

One can apply the pipe filling approach to any collection of Riemannian manifolds
joined by collections of thin cylinders.  A very small sphere in a Riemannian manifold
is arbitrarily close to a small Euclidean sphere.  As long as the cylinders are standard
isometric products of spheres with line segments, then technique works.  The
metric space $Z$ can be created with thin half cylinders between products of
the manifolds with small intervals, and $Z'$ can be built
using the diameter of the manifolds in the place of $\pi R$
when defining $h$.
\end{rmrk}

\subsection{Collapse in the limit} \label{subsect-collapse}

A sequence of Riemannian manifolds, $M_j$, is said to {\em collapse}
if $\vol(M_j) \to 0$.  Such sequences do not converge in the 
Lipschitz or smooth sense because the limit spaces have the
same dimension and volume converges in that setting.  They
have been studied using Gromov-Hausdorff and metric measure 
convergence.  As mentioned in Remark~\ref{vol-to-zero}, collapsing sequences
of Riemannian manifolds converge in the intrinsic flat sense
to the $\bf{0}$ current space.
In fact if $M_j$ converges in the Gromov-Hausdorff sense
to a lower dimensional limit space then they converge in
the intrinsic flat sense to $\bf{0}$ as well [Corollary~\ref{lower-dim}].  

\begin{example}\label{example-tori-GH}
The sequence of tori, $M_j= S^1_{\pi/j}\times S^1_\pi$,
depicted in Figure~\ref{figure-tori-GH} has
volume $\vol(M_j)=\pi/j \to 0$, so $M_j$ converges in the
intrinsic flat sense to $\bf{0}$.
Note that $M_j$ converges in the Gromov-Hausdorff sense to $S^1$ because 
\be
d_{GH}(S^1, M_j) \le d_H(\{p\}\times S^1_\pi, S^1_{\pi/j}\times S^1_1) = \pi/(2j)\to 0.
\ee
\end{example}

In the next example, is the well known "jungle-gym" example where the
 Gromov-Hausdorff limit is higher dimensional than the sequence.  Here we see that 
the intrinsic flat limit is $\bf{0}$:

\begin{ex} \label{ex-jungle-gym} \label{ex-gym-1}
The Riemannian surface, $M_j$, defined as a submanifold of
Euclidean space by attaching adjacent disjoint spheres of radius $R_j$
centered on lattice points 
of the form $(\frac{n_1}{2^j}, \frac{n_2}{2^j}, \frac{n_3}{2^j})$
where $n_i\in \N$ with cylinders of radius $r_j<<R_j$
with
\be \label{ex-gym-1-1}
\sum_{i=1}^{2^{3j}} \frac{4}{3} \pi R_j^2  \le A_0,
\ee
and total area of the cylinders approaches $0$.  

As $j \to \infty$ this sequence converges to the cube $[0,1]^3$
with the taxicab norm: 
\be
d_{taxi}((x_1,x_2,x_3), (y_1,y_2,y_3) ) = \sum_{i=1}^3 |x_i-y_i|
\ee 
and in the intrinsic flat sense to $\bf{0}$
\end{ex}

We skip the proof of the Gromov-Hausdorff convergence since
this is best done using Gromov's $\epsilon$ nets \cite{Gromov-metric}.

\begin{proof}
By Theorem~\ref{GH-to-flat}, a subsequence of $M_j$ converges
in the intrinsic flat sense
to some integral current space $M_0\subset [0,1]^3$
since $\textrm{area}(M_j)$ is uniformly bounded by
(\ref{ex-gym-1-1}) and the diminishing areas of the cylinders.
By the pipe filling technique [Remark~\ref{rmrk-pipe-filling}], 
we know the collections of spheres, $N_j$, converge in the intrinsic
flat sense to $M_0$ as well.
However each sphere isometrically embeds into a hemisphere
of higher dimension, so we can embed $N_j$ into a collection
of hemispheres and see that
\be
d_\Fm(M_j, {\bf 0})\,\, \le\,\, \sum_{i=1}^{2^{3j}}
\,\, \frac{5}{8}\, \pi \, R_j^3 \,\, \le \,\,A_0 R_j \,\,\to\,\, 0,
\ee
so $M_0$ is the zero space.
\end{proof}


\subsection{Cancellation in the limit}  \label{subsect-cancellation}

Sometimes sequences of integral current spaces converge to the
$\bf{0}$ current space even when their total mass is uniformly bounded
below.  We begin with a classical example of integral currents in
Euclidean space and then give a sequence of Riemannian manifolds
which cancel in the limit [Example~\ref{example-cancels}].

\begin{ex}\label{ex-cancels-in-R3}
Let $T_j \in \intcurr_2(\R^3)$ be defined as integration over
$\{(x,y, 1/j): x^2+y^2\le 1\}$ oriented upward plus integration
over $\{(x,y,-1/j): x^2+y^2\le 1\}$ oriented downward.  As $j\to \infty$,
$T_j$ converges in the flat sense to the $\bf{0}$ current.  Thus
the integral current spaces, $(\set(T_j), d_{\R^3}, T_j)$, 
converge to the $0$ current space.
\end{ex}  

\begin{proof}
This example is easily proven taking $B_j$ equal to integration
over the solid cylinder between the disks in $T_j$
and $A_j$ equal to integration over the cylinder.
\end{proof}

To create a sequence of Riemannian manifolds which cancel in the
limit like this is more tricky.  If one tries to fold a surface onto itself
so that it is close enough to cancel it is not isometrically embedded
into the space.  To create an isometric embedding in a folded position
we need to provide shortcuts between the two sheets.  See
Figure~\ref{figure-cancels}.  

\begin{ex} \label{example-cancels}
Given any compact oriented Riemannian
manifold, $M_0^m$, one can find a sequence of oriented Riemannian
manifolds, $M^m_j$, which converge in the Gromov-Hausdorff sense to
$M_0^m$ and yet in the intrinsic flat sense to ${\bf 0}$.  The sequence, $M_j^m$,
have volumes converging to twice the volume of $M_0^m$.  
\end{ex}

This example is also described in \cite{SorWen1} but the proof 
there is not constructive.

\begin{proof}  
First, let $M_0$ be an arbitrary closed oriented Riemannian manifold and fix $j\in\N$
before defining $M_j$. 
Choose a collection of points, 
\begin{equation}
\{p_1,p_2,...p_{N_j} \} \subset M_0
\end{equation}
such that $d\left(p_i,p_k\right)> 3/j$ and $M_0\subset \bigcup_{i}B\left(p_i, 10/j\right)$.
We choose any $r_n$ such that $r_n \le \min\{1/j, \injrad\left(M_0\right)/2\}$, 
where $\injrad\left(M_0\right)$ denotes the injectivity radius of $M_0$.

Define an integral current space $W_j$ as
a Riemannian manifold with corners via the isometric product 
\be \label{ex-cancels-0}
W_j=\left(M_0\setminus U_j\right)\times [0,\delta_j]
\ee
where
\be
U_j= \bigcup_{i=1}^{N_n}B\left(p_i, r_n\right)\,\,\,
\textrm{ and } \,\,\,
\delta_j < \min\left\{\left(\vol_{m-1}\left(\partial U_j\right) \right)^{-1}, 1/j\right\}.
\ee
Let $M_j = \partial W_j$ 
so that $M_j$ is two copies of $M_0\setminus U_j$
with opposite orientations glued together by cylinders of the form
$\partial B\left(p_i,r_n\right) \times [0,\delta_j]$
as in Figure~\ref{figure-cancels}.  There are smooth
Riemannian manifolds arbitrarily close to the $M_j$ in the Lipschitz sense.

Note that  $d_W\left(\left(x,\delta_j\right),\left(x,0\right)\right)=\delta_j$ while
$d_{M_j}\left(\left(x,\delta_j\right), \left(x,0\right)\right)$ is achieved
by a curve traveling to a cylinder, then a distance $\delta_j$  and back again, 
so $M_j$ does not isometrically embed into $W_j$.  One might try
constructing a bridge $Z_j$ from $M_j$ to $W_j$ using 
Lemma~\ref{lem-bridge-Z} but since $\diam(\partial M_j)$ does not
converge to $0$, we cannot apply this lemma directly.  Instead we will use
a similar technique taking advantage of the increasing
density of $\partial M_j$.

First we set
$\bar{\epsilon}_j=10/j + \delta_j +10/j$ then, by the density of the balls,
\be \label{ex-cancels-1}
d_{M_j}\left(\left(x,0\right), \left(x,s\right)\right) \le\bar{ \epsilon}_j,
\ee
for all choices of $\left(x,s\right)\in M_j=\partial W_j$.

We now construct another Lipschitz manifold $Z_j$
into which $M_j$ does isometrically embed and such that $M_j=\partial Z_j$
where $\mass\left( Z_j\right)=\vol\left(Z_j\right)\to 0$, proving that $M_j$ flat converges
to $0$.  Taking 
\be \label{ex-cancels-1.5}
\epsilon_j:=2 \sqrt{\bar{\epsilon}_j^2 + \bar{\epsilon}_j \diam\left(M_j\right) },
\ee
we define our metric space:
\be \label{ex-cancels-2}
Z_j= \partial W_j \times [0,\epsilon_j]  \cup W_j\times \{\epsilon_j\} \subset W_j\times [0,\epsilon_j],
\ee
where the product is an isometric product and $Z_j$ is endowed with
the induced length metric.   Clearly $M$, $\partial W_j$ and $\partial Z_j$ are
all isometric and
\begin{eqnarray*}
\vol\left(Z_j\right) &= & \vol_m\left(M_j\right)\epsilon_j + \vol_{m+1}\left(W_j\right) \\
& = & \left(2\vol_m\left(M_0\setminus U_j\right) + \vol_{m-1}\left(\partial U_j\right) 2\delta_j\right)\epsilon_j + 
\vol_m\left(M_0\setminus U_j\right) \delta_j \\
& \le & \left(2 \vol_m\left(M_0\right) + 2\right)\epsilon_j  + \vol_m\left(M_0\right) \delta_j,
\end{eqnarray*}
by the choice of $\delta_j$.   Thus to prove $d_{\Fm}\left(M_j,0\right)\to 0$, 
we need only show that the map
$\phi_j: M_j=\partial W_j \to M_j\times \{0\}\subset Z_j$
is an isometric embedding.  

Recall that all points in $M_j$ may be denoted $\left(x,s\right)$ where $x\in M_0\setminus U_j$,
$s\in [0,\delta_j]$.   Note that when $s\in \left(0,\delta_j\right)$, then we
are on a tube and $x\in \partial U_j$.  Thus
all points in $Z_j$ may be denoted $\left(x,s,r\right)$ where $x\in M_0\setminus U_j$,
$s\in [0,\delta_j]$ and $r\in [0,\epsilon_j]$.   Note that when $s\in \left(0,\delta_j\right)$
then either we are in a tube, in which case $x\in \partial U_j$, or we
are in the interior of $W$, in which case $r=\epsilon_j$.     Then
$\phi_j\left(x,s\right):=\left(x,s,0\right)$.  

Let $\gamma\left(t\right)=\left(x\left(t\right),s\left(t\right),r\left(t\right)\right)$ run minimally in $Z_j$ from $\phi_j\left(x_0,s_0\right)$
to $\phi_j\left(x_1,s_1\right)$.  So $r\left(0\right)=r\left(1\right)=0$.  If $r\left(t\right)<\epsilon_j$ for all t, then
$\gamma$ may be deformed decreasing its length to 
\be \label{ex-cancels-3}
\eta\left(t\right)=\left(x\left(t\right),s\left(t\right)\right)\subset \phi_j\left(M_j\right),
\ee
where $\eta$ runs minimally between the endpoints,
in which case the length is $L\left(\gamma\right)=d_{M_j}\left(\left(x_1,s_1\right), \left(x_2,s_2\right)\right)$.  

So we may assume there exists $t$ where $r\left(t\right)=\epsilon_j$.  
Let $t_0, t_1$ be the first and last  times where $r\left(t\right)=\epsilon_j$ respectively.
For $t<t_0$ and $t>t_1$ we can again use the fact that
$\eta\left(t\right)=\left(x\left(t\right),s\left(t\right)\right)$ lies in $M_j$, but this times we make a more careful estimate on the
length.  Since $\gamma$ runs minimally from $\gamma\left(0\right)=\left(x\left(0\right),s\left(0\right),0\right)$ to
$\gamma\left(t_0\right)=\left(x\left(t_0\right),s\left(t_0\right),\epsilon_j\right)$  and our space has an isometric
product metric $M_j \times [0,\epsilon_j]$, 
\be \label{ex-cancels-4}
L\left(\gamma\left([0,t_0]\right) \right) = \sqrt{ L\left(\eta\left([0,t_0]\right)\right)^2 + \epsilon_j^2 }
=\sqrt{ d_0^2 + \epsilon_j^2 }
\ee
where $d_0= d_{M_j}\left( \left(x_0,s_0\right) ,\left(x(t_0),s(t_0)\right)\right)$.
Similarly
\be \label{ex-cancels-5}
L\left(\gamma\left([t_1,1]\right) \right) = \sqrt{ L\left(\eta\left([t_1,1]\right)\right)^2 + \epsilon_j^2 }
=\sqrt{ d_1^2 + \epsilon_j^2 }
\ee   
where $d_1= d_{M_j}\left(\left(x\left(t_1\right),s\left(t_1\right)\right),\left(x_1,s_1\right)\right)$.
We can project the middle segment to $M_0\setminus U_j$ to see that
\be
L\left(\gamma\left([t_0,t_1]\right) \right) \ge L\left( x\left([t_0,t_1]\right)\right) = d_{M_j}\left(\left(x\left(t_0\right),0\right), \left(x\left(t_1\right),0\right) \right).
\ee
By (\ref{ex-cancels-1}) we can estimate the distance in $M_j$ from $\left(x\left(t_i\right),0\right)$
to $\left(x\left(t_i\right), s\left(t_i\right)\right)$ and apply the triangle inequality to see that
\begin{eqnarray} \label{ex-cancels-6}
L\left(\gamma\left([t_0,t_1]\right) \right) & \ge &  d_{M_j}\left(\left(x\left(t_0\right),s\left(t_0\right)\right), \left(x\left(t_1\right),s\left(t_1\right)\right) \right) - 4\bar{\epsilon}_j \\
&= &d_{M_j}\left(\left(x_0,s_0\right),\left(x_1,s_1\right)\right) -d_0-d_1 - 4\bar{\epsilon}_j.
\end{eqnarray}
Combining (\ref{ex-cancels-4}), (\ref{ex-cancels-5}) and (\ref{ex-cancels-6}), and
applying the definition of $\epsilon_j$ in (\ref{ex-cancels-1.5}) using the
fact that $d_j \le \diam \left(M_j\right)$ we have:
\be
L\left(\gamma\right) -d_{M_j}\left(\left(x_0,s_0\right),\left(x_1,s_1\right)\right)\ge \sqrt{ d_0^2 + \epsilon_j^2} + \sqrt{ d_1^2 + \epsilon_j^2 } -d_0-d_1 - 4\bar{\epsilon}_j \ge 0.
\ee
Thus we have an isometric embedding.
 \end{proof}

\subsection{Doubling in the limit}

In this subsection we provide an example of a sequence of Riemannian manifolds
which converge to an integral current space whose integral current structure
is twice the standard structure and whose mass is twice its volume.  The construction
is the same as the one in the last subsection of a canceling sequence except
that all tubes are now twisted so that the orientations line up instead of canceling
with each other.  

\begin{ex} \label{example-doubles}
Given any compact oriented Riemannian manifold $M^m_0=(M_0,d_0,T_0)$ 
we can find a sequence of
a sequence of oriented Riemannian
manifolds $N^m_j$ which converge in the Gromov-Hausdorff sense to
$M_0^m$ and yet in the intrinsic flat sense to $M^m_0$ with weight $2$:
$\left(M_0, d_0, 2T_0\right)$.  
The sequence $N_j^m$
have volumes converging to twice the volume of $M_0^m$ and large
regions converging smoothly to $M_0^m$.  
\end{ex}

\begin{proof}   
We begin the construction exactly as in the beginning of
the construction of Example~\ref{example-cancels}
creating a sequence of $M_j=\partial W_j$ which flat converge to $0$.   We cut
$M_j$ along the level $s=\delta_j/2$ which is a disjoint union of spheres.
These spheres may be made isometric to a standard sphere of appropriate
radius with a bi-Lipscitz map whose constant is very close to 1.
These spheres are glued back together with the reverse orientation to
create an oriented Riemannian manifold $N_j$.  Note that there
are two copies of $M_0\setminus U_j$ in $N_j$, both with the same
orientation defined by $T_0$
and that there is an orientation preserving isometry between
these two copies.  

Let $\left(X_j,d_j\right)$ be the metric space formed by taking
two copies of $M_0$ with line segments of length
$\delta_j$ joining the corresponding points $p_{j,1},...p_{j,N_j}$ 
endowed with the length metric.     
Applying an adaption of the pipe filling technique [Remark~\ref{rmrk-pipe-filling-adapt}]
to $N_j$ and $M_j$ respectively, 
we see that that both are Gromov Hausdorff close to $X_j$.  Furthermore
\be \label{ex-doubles-1}
\lim_{j\to\infty} d_{\Fm}\left(N_j, \left(X_j,d_j, T_j\right)\right) \to 0 \textrm{ and }
\ee
\be \label{ex-doubles-2}
\lim_{j\to\infty} d_{\Fm}\left(M_j, \left(X_j,d_j, S_j\right)\right) \to 0,
\ee
where the
distinction is that $T_j$ has the same orientation on both 
copies of $M_0$ in $X_j$
while $S_j$ has opposite orientations on each slice.

Thus the canonical set, $\set\left(T_j+S_j\right)$, 
is a copy of $M_0$ lying in $X_j$ and
there is a current preserving isometry 
\be \label{ex-doubles-3}
\varphi:\left(M_0, d_0, 2T_0\right) \to \left(\set\left(T_j+S_j\right), d_j, T_j+S_j\right).
\ee  

By Example~\ref{example-cancels}, we know that $d_{\Fm}\left(M_j, 0\right) \to 0$.
Combining this fact with (\ref{ex-doubles-2}), 
we see that $d_{\Fm}\left(\left(X_j,d_j,S_j\right), 0\right) \to 0$ as well.  
So there
exists a metric space $Z_j$ and an isometric embedding 
$\psi: X_j \to Z_j$ such that 
\be \label{ex-doubles-4}
d^{Z_j}_F\left(\psi_\# S_j, 0\right) \to 0.
\ee

By (\ref{ex-doubles-3}), we see that $\psi\circ\varphi$ isometrically embeds $M_0$
in $Z_j$ as well.  Thus,
\begin{eqnarray*}
d_{\Fm}\left(\left(X_j,d_j,T_j\right), \left(M_0, d_0, 2T_0\right) \right) & \le & d^{Z_j}_F \left(\psi_\# T_j, \psi\circ \varphi_\# 2T_0\right) \\
& = & d^{Z_j}_F \left(\psi_\# T_j, \psi_\# \left(T_j+S_j\right)\right) \\
& = & d^{Z_j}_F \left(0, \psi_\# \left(S_j\right)\right)  \to 0.
\end{eqnarray*}
By (\ref{ex-doubles-1}), we then have $M_j$ converging to $\left(M_0, d_0, 2T_0\right)$.
\end{proof}
\vspace{.1cm}

\subsection{Taxi Cab Limit Space}


In this subsection we give an example of a sequence of Riemannian
manifolds which converge in both the Gromov-Hausdorff and Intrinsic
Flat sense to the square torus with the taxicab metric, 
$M_{taxi}=\left(T^2,d_{taxi}\right)$
where 
\be
d\left(\left(x_1,x_2\right),\left(y_1,y_2\right)\right)= |x_1-y_1| +|x_2-y_2|.
\ee
Although the sequence converges without cancellation, the mass
does not converge.

This sequence was described to the first coauthor by Dimitri Burago
as a sequence which converges in the Gromov-Hausdorff sense.    Here
we describe Burago's proof and then prove that
the flat and Gromov-Hausdorff limits agree in this setting.   
We show an
integral current structure exists on the taxicab torus but we do not 
explicitly construct this structure.  It would be of interest to investigate this in
more detail.

\begin{ex}\label{ex-taxi}
There exists a sequence of Riemannian manifolds, $M_j^2$ which converge in the
intrinsic flat and Gromov-Hausdorff sense to the flat $1\times 1$
torus with the taxi metric $M_{taxi}=\left(T^2, d_{taxi}\right)$.   In this example the
mass drops in the limit.
\end{ex}

\begin{proof} 
The manifolds can
be described as submanifolds of $T^2 \times R$ with the standard flat metric
by the following graph:
\be
M_{n,j}^2= \{\left(x,y,f_{n,j}\left(x,y\right)\right): \,\, f_{n,j}\left(x,y\right)=  \left(1-sin^{n}\left(2^j\pi t\right)\right)/2^j \}.
\ee 
The metric on $M_{n,j}^2$ is defined as the length metric induced by the
metric tensor defined by this embedding (which is not an isometric embedding).

Let $G_j$ be the grid of $1/2^j$ squares defined by
\be
G_j= M_{n,j}^2 \cap T^2\times \{0\}.
\ee
As $n \to \infty$ for fixed $j$, $f_{n,j}$ converge pointwise to $h_j:T^2 \to R$ where
$h_j\left(x,y\right)=0$ for  $\left(x,y\right)\in G_j$ and is $1$ elsewhere.
 
Note also that $M_{n,j}^2$ converges in the Gromov-Hausdorff and Lipschitz sense 
as $n \to \infty$ to a metric space $X_j$ defined by
 created by attaching disjoint five-sided 
$1/2^j$ cubes to each square in the $1/2^j$ grid, $G_j$, so that $G_j$ with
the induced length metric isometrically embeds into $X_j$ with its natural length
metric.  We see it is an isometric embedding because a minimizing geodesic
between points in the grid would never be shorter going over the top of a cube
rather than going around the base square.

This space $X_j$ converges in the Gromov-Hausdorff sense to $T^2_{taxi}$.  This can be 
seen because grid $G_j$ isometrically embeds into
both spaces so
\begin{eqnarray*}
d_{GH}\left(X_j, T^2_{taxi}\right) & \le & d_{GH}\left(X_j, G_j\right) + d_{GH}\left(G_j, T^2_{taxi}\right) \\
& \le & d^{X_j}_{H}\left(X_j, G_j\right) + d^{T^2_{taxi}}_{H}\left(G_j, T^2_{taxi}\right) \\
& \le &  2/2^j  + 1/2^j  \to 0.
\end{eqnarray*}
  Here we will see that
the flat limit is also the torus with the taxicab metric.

By the Lipschitz convergence
we have a natural current structure $T_j$ on $X_j$ and
we can choose $n_j$ large enough that $d_{\Fm}\left(\left(X_j,d_j,T_j\right), M_{n_j,j}\right) < 1/j$
and $d_{GH}\left(X_j, M_{n_j,j}\right) <1/j$.    So if we set $M_j=M_{n_j,j}$
and prove $M_j$ converges in the intrinsic flat sense to $T^2_{taxi}$
we are done.

By Theorem~\ref{GH-to-flat}, we know a subsequence $\left(X_{j_i}, d_{j_i}, T_{j_i}\right)$
converges to an integral current space $\left(X, d_{taxi}, T_\infty\right)$ where $X\subset T^2_{taxi}$.
Since $X_{j_i}$ are locally contractible, we may apply Theorem 1.3 from
\cite{SorWen1}, to see that $X=T^2_{taxi}$ 
(c.f. Theorem~\ref{contractibility-no-cancellation}).  
It is not immediately clear what the limit current structure on $T^2_{taxi}$ looks like
so we just call it $T_\infty$.

We can also explicitly check that $\left(X_j, d_j, T_j\right)$ is a Cauchy sequence with
respect to the intrinsic flat distance..   This can be seen because $G_j$
isometrically embeds into $G_{j+1}$ and so we may glue $X_j$ to $X_{j+1}$
along this embedding to create a geodesic metric space $W_{j}$.   The metric
space $W_j$ consists of $\left(2^j\right)^2$ copies of a $\left(1/2^j\right)\times \left(1/2^j\right)$ five-sided 
cube attached to  four $\left(1/2^{j+1}\right)\times \left(1/2^{j+1}\right)$ five sided cubes.   The
restriction of $T_j-T_{j+1}$ to this collection of five cubes has no boundary
(as can be seen because the collection of five cubes is bi-Lipscitz to
a sphere).    By isometrically embedding $W_j$ into a Banach space,
we may apply the second author's filling theorem
 \cite{Wenger-flat} to fill in each collection of five cubes
with a 3 dimensional integral current of mass $M_0 \left(1/2^j\right)^3$.   Thus
\be
d_{\Fm}\left(\left(X_j,d_j,T_j\right), \left(X_{j+1}, d_{j+1}, T_{j+1}\right) \right) \le \left(2^j\right)^2 M_0 \left(1/2^j\right)^3= M_0/2^j.
\ee
and our sequence is Cauchy.  

Thus $(X_j,d_j, T_j)$ converges to the limit of the subsequence 
$(T^2_{taxi}, d_{taxi}, T_\infty)$.

Note $\mass(T_j) \to 5$ due to the five faces
on each cube.  Thus $\mass(T_\infty)\le 5$ by the lower semicontinuity of mass.  

Now we slightly alter the top face of each cube to have
a central peak, creating a new sequence of manifolds which also converge
to the taxicab space in both the Gromov-Hausdorff and intrinsic flat sense
with the exat same arguments as above.  
These new manifolds have mass
converging to a limit strictly greater than $5$.  Thus we
have found a sequence of integral current spaces whose Gromov-Hausdorff and
intrinsic flat limits agree but whose masses do not converge.  
\end{proof}


\subsection{Limit whose Completion has Higher Dimension}

There were many reasons that we defined integral current spaces using
the set of positive lower density of the current rather than the support 
[Definition~\ref{defn-current-space}].
The key reason is that the set of a current has the correct dimension
so that our integral current spaces are always countable $H^k$
rectifiable of the correct dimension even though they need not be
compact or complete.  If one takes the completion on an integral
current space, it may have higher dimension as we see here:

\begin{example}\label{example-dense-support}
There is a sequence of Riemannian surfaces $M'_j$
that converge to a nonzero 2 dimensional
integral current space $N_\infty$ 
such that the closure of $N^3$ is the solid 3 dimensional
cube with the standard Euclidean metric.  
\end{example}

\begin{proof}
As in Example~\ref{example-gym2} , our sequence of $M_j'$
will be constructed using spheres joined by cylinders. 
In that example, we never used anything special about the
arrangement of the spheres used to defined $M_j$
except that  the total volumes of the spheres were uniformly
bounded and the radius, $\epsilon_j$ of the connecting cylinders
were chosen small enough that the total length of the cylinders
$L_j$ satisfied $L_j\epsilon_j \to 0$.   Note that it was not necessary
that the spheres and cylinders isometrically embed into Euclidean space
as this embedding was only used to describe the locations
of the spheres.  Here we will again start with a sequence of
inductively defined spheres embedded into Euclidean 3 space,
but we will connect them with abstract cylinders so that we need 
not concern ourselves with intersections.

We begin by constructing a sequence of outward oriented spheres
which are disjoint and dense in the solid unit cube, $[0,1]^3$.  
The first $n_1=8$ spheres are centered on points of the
form $\left(n/4, m/4\right)$ where $\left(n,m\right)\in \{1,2,3\}\times\{1,2,3\}$
and have radius $r_1>0$ and sufficiently small that
they are disjoint, they
have total area $n_1 4\pi r_1^2< 1$ and the total of their diameters
is $n_1\pi r_1 <1/2$
.  The next collection of  $n_2$ spheres
are centered on points of the
form $\left(n/8, m/8\right)$ where $\left(n,m\right)\in \{1,2,3...7\}\times\{1,2,3...7\}$
but excluding any such points which already lie on the
first $n_1$ spheres.  Then the radius $r_2$ of these
$n_2$ spheres is chosen small enough
that all the $n_1+n_2$ spheres are disjoint,
the total area of the spheres, $n_1 4\pi r_1^2 +n_2 4\pi r_2^2< 1$
and the total of the diameters, $n_2\pi r_2 <1/4$
.  We continue
in this matter creating a dense collection of
disjoint spheres lying in $[0,1]^3$ whose closure is $[0,1]^3$
and whose total area is $\le 1$ and total of the last
$n_j$ spheres' diameters is  $<1/2^j$
We will let $V_j$ denote the first $n_1+...+n_j$ disjoint
spheres.

We next create geodesic metric spaces $X_j$ by
connecting the spheres in $V_j$ with line segments and
prove $X_j$ converges in the Gromov-Hausdorff sense
to $[0,1]^3$ with the standard Euclidean metric.   
The $X_j$ will have induced length metrics
and will not isometrically embed into $[0,1]^3$.  
The line segments connecting the spheres may appear
to intersect in $[0,1]^3$ but, by definition, do not intersect.  More
precisely, we will say we have connected a sphere,
$S_1$, to a sphere, $S_2$, if we find points
$x_1\in S_1$ and $x_2\in S_2$ such that $d_{[0,1]^3}\left(x_1,x_2\right)$
achieves the distance, $d$, between $S_1$ and $S_2$ as measured
in $[0,1]^3$ and then we attach an abstract line segment
of length $d$ between these two points.  

\begin{figure}[h] 
   \centering
   \includegraphics[width=4.7in]{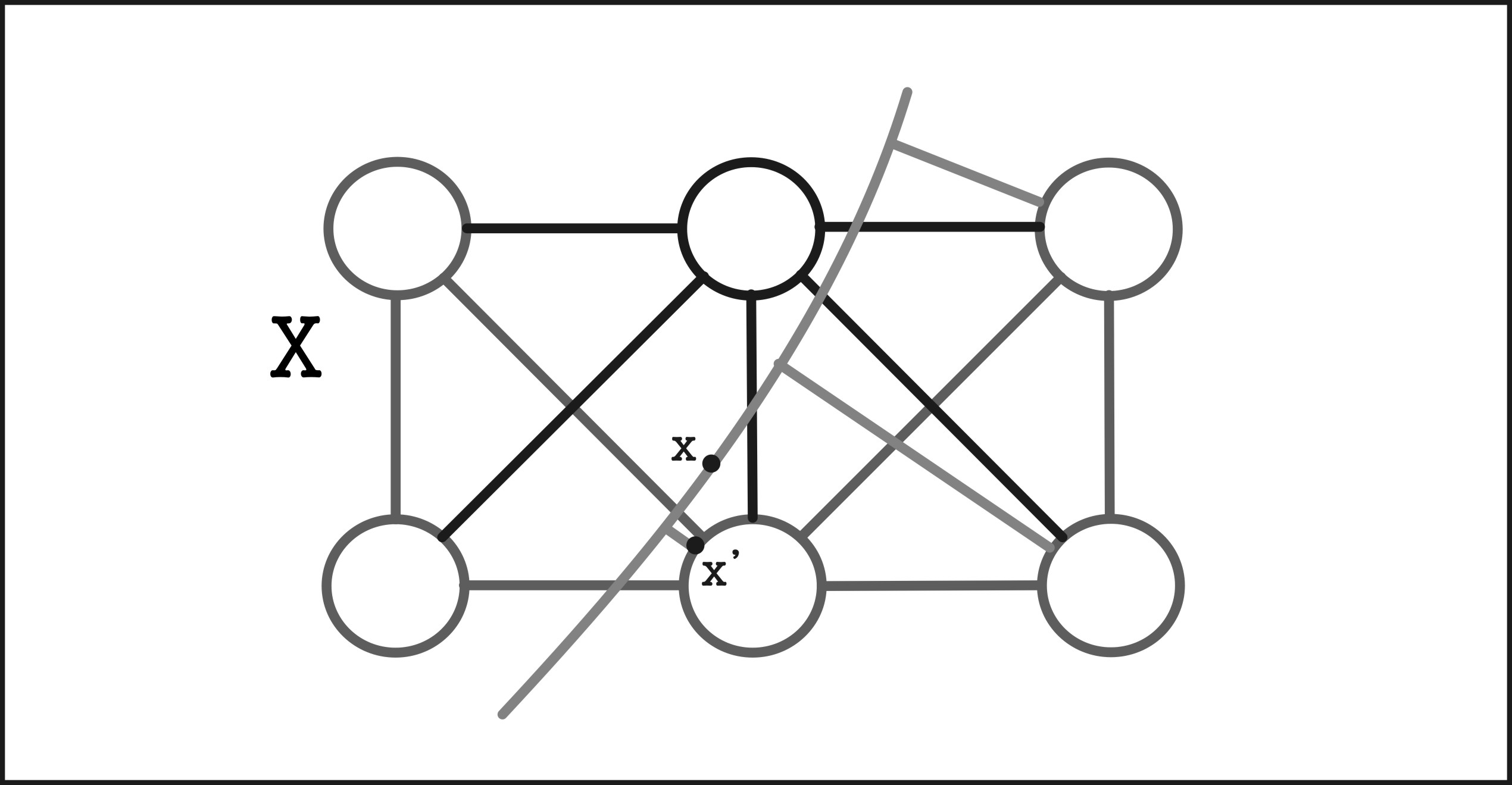} 
   \caption{Here the spheres are drawn as circles.}
   \label{fig-dense-support-1}
\end{figure}

Each space $X_j$ is a connected collection of the
first $n_1+...+n_j$ spheres.  Not all
spheres will be connected to each other.    See Figure~\ref{fig-dense-support-1}
for a view of a tiny region in the cube where the
spheres are depicted as circles and the endpoints of segments
connecting spheres are solid points.  
To build 
$X_j$ we take each sphere $\partial B_1$ of radius $r_j$ and
connect it to any other neighboring sphere $\partial B_2$ of radius $r\le r_j$
(whose line segment is of length at most $1/j$) and
such that $B_1 \cap B_2 =\emptyset$.   This second condition
will help with orientation later.    Note that none of the larger spheres
are connected directly to each other, only via connections among the
smaller spheres.  This $X_j$ is a connected
geodesic metric space and let $L_j$ be the total lengths of all segments
in $X_j$.   We can create an integral current space
$N_j=\left(\set\left(T_j\right), d_{X_j}, T_j\right)$ where $T_j$ is integration over the
spheres in $X_j$ with outward orientation.

We define Lipschitz Riemannian manifolds, $M_j =\partial T_{\epsilon_j}\left(X_j\right)$, 
as the boundary
of an abstract tubular neighborhood around $X_j$,
where $\epsilon_j$ is taken so small that any pair of spheres
in $X_j$ is still disjoint when the radii are $\epsilon_j$ larger and
such that $\epsilon_j L_j < 1/j$ and such that the area of $M_j$
is less than $1+1/j$.   This abstractly defined space does not lie in 
$[0,1]^3$ but each geodesic segment has been replaced by a cylinder
of the appropriate width so that $M_j$ immerses into $X_j$ with a local
isometry.  Note that by our careful connection of the spheres in the
previous paragraph, $M_j$ is orientable and we orient it so that all the
spheres are outward oriented.  See Figure~\ref{fig-dense-support-2}.

\begin{figure}[h] 
   \centering
   \includegraphics[width=4.7in]{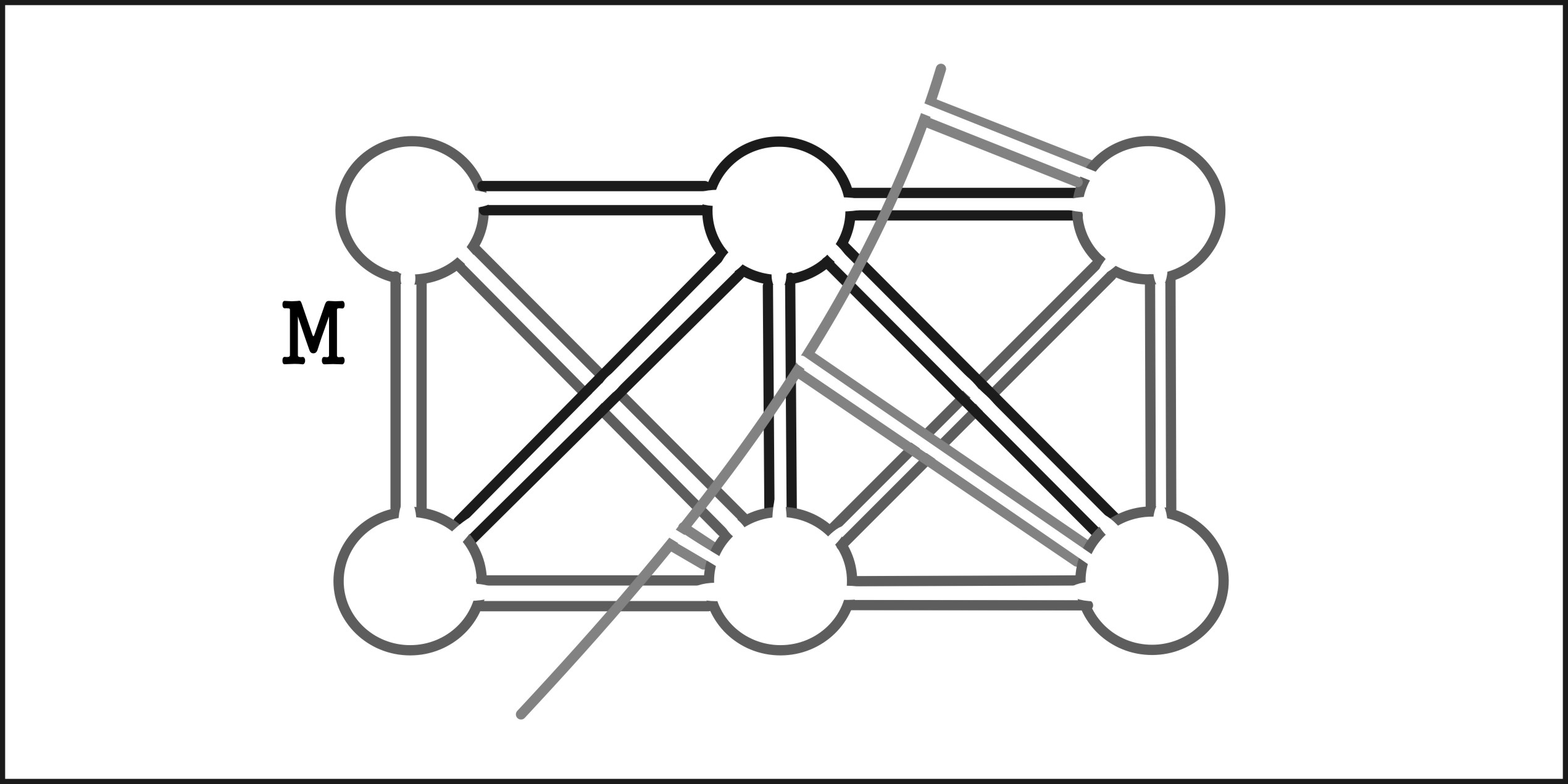} 
   \caption{Note the outward orientation.}
   \label{fig-dense-support-2}
\end{figure}

By the pipe filling technique [Remark~\ref{rmrk-pipe-filling}] and the bounds on $\epsilon_j$,
$L_j$ and the total area, $d_{\Fm}\left(M_j, N_j \right) \to 0$ and $d_{GH}\left(M_j, X_j\right) \to 0$.
To complete our example we need only prove that $X_j$ converges
in the Gromov-Hausdorff sense to $[0,1]^3$ and $N_j$ converges
in the intrinsic flat sense to $N_\infty$, where 
$N_\infty=\left(\set (T_\infty), d_{[0,1]^3}, T_\infty\right)$ and $T_\infty$ is
defined by integration over all the spheres in our collection
with outward orientation.  Note that by the density of the spheres
in $[0,1]^3$ the completion of $N_\infty$ is $[0,1]^3$.

Notice that $d_{GH}\left(N_j, X_j\right) \le d_H^{X_j}\left(N_j, X_j\right) \le 1/j$
by the shortness of the joining line segments in the creation of $X_j$.
So we need only prove $N_j$ converges in the Gromov-Hausdorff
sense to $[0,1]^3$ and in the flat sense to $N_\infty$.

There is a natural map $f_j: N_j \to [0,1]^3$ which is
 not an isometry.
However we claim there is a uniform distortion, $D_j$, such that
if $x,y \in N_j$, then 
\be \label{spheres-in-cube-1}
d_{[0,1]^3}\left(f_j\left(x\right), f_j\left(y\right)\right)- d_{X_j}\left(x,y\right) | \le D_j \to 0
\ee
as $j\to \infty$.  After proving this claim we will use it to prove
our convergence claims.

Given $x \in N_j$, there exists $x'$ in a sphere of radius $r_j$ in $N_j$ outside
the sphere containing $x$ such that $d_{[0,1]^3}\left(x',x\right) < 6/2^j$ by the density of
the smallest spheres in $[0,1]^3$.  See Figure~\ref{fig-dense-support-1} again.
By the connecting
of the spheres by line segments we know 
$d_{X_j} \left(x',x\right) \le  \pi 6/2^j +1/j $ since an arclength can always
be bounded by $\pi$ times a secant length and we need travel down at most one
line segment to reach the smaller sphere.  Similarly for $y\in N_j$,
there exists $y'$  with $d_{X_j} \left(y',y\right) \le  2 6/2^j +1/j $.   So we need only prove
(\ref{spheres-in-cube-1}) for $x', y'$ lying in smallest spheres in $X_j$.

Between $x'$ and $y'$, one can draw a straight line segment in $[0,1]^3$ and
then select the smallest spheres in $X_j$ with radius $r_j$ which
are closest to this line segment.  By the density of the smallest
spheres we know there are many spheres very close to this segment but we
need to avoid zigzagging between them.   We apply the fact that
the connecting segments in $X_j$ get as long as $1/j$ while the
density of the spheres is $1/2^j$, so that we may actually select
smallest spheres between $x'$ and $y'$ which are joined
by segments whose total length approximate $d_{[0,1]^3}\left(x',y'\right)$.
Between the segments a path between $x'$ and $y'$
lying in $X_j$ must go around the small spheres, however, their
total diameter has been bounded above by $1/2^j$ so this 
does not add to the error significantly and we have (\ref{spheres-in-cube-1}).

We now create spaces 
$Z_j = X_j \times [0,h_j] \disjointunion [0,1]^3$ 
where
 \be
 h_j= \sqrt{ (D_j/2)(2 \diam(N_j) + D_j/2}
 \ee
so that $\left(x, h_j\right)$ is identified with $f_j\left(x\right)$ with the induced
length metric.   Note that there is a distance nonincreasing
retraction to $[0,1]^3$, so there is an isometry
$\varphi: [0,1]^3 \to Z_j$.   
We claim there is also an isometric embedding
$\psi: N_j \to N_j\times \{0\}\subset Z_j$ since a shortest curve between
points in $N_j\times \{0\}$ either stays in the $X_j \times\{0\}$ level or enters
the $[0,1]^3$ region where we can apply (\ref{spheres-in-cube-1})
to control the short cut in that region.   
To enter the $[0,1]^3$ region, it first travels a distance
$\sqrt{L_1^2+h_j^2}$ to the region, then a distance greater than $L_2-D_j$ in the 
region and then a distance $\sqrt{L_3^2+h_j^2}$ back from the region where
$L_1+L_2+L_3$ equals the distance in $N_j$ between the endpoints of the curve. 
However, by the choice of $h_j$ this causes a contradiction.

Thus $d_GH\left(N_j, [0,1]^3\right) \le d_H^{Z_j}\left(\psi\left(N_j\right), \varphi\left([0,1]^3\right) \right) \to 0$.  
Furthermore 
\be
d_{\Fm}\left(N_j, N_\infty\right) \le d_F^{Z_j}\left(\psi_\# N_j, \varphi_\# N_\infty \right)
\le \mass\left(A_j\right) +\mass\left(B_j\right)
\ee
where $A_j\in \intcurr_2\left(Z_j\right)$ is integration over the spheres of radius $r_j$ in $[0,1]^3$
and  $B_j \in \intcurr_3\left(Z_j\right)$ 
is integration over the collection of cylinders $N_j \times [0,h_j]$.
By our bound on the total area of the spheres, $\mass A_j \to 0$
and $\mass \left(B_j\right) \le h_j \to 0$.  So we are done.
\end{proof}

    
\subsection{Gabriel's Horn and the Cauchy Sequence with no Limit}

In this section we present an example of a sequence of compact Riemannian manifolds which are Cauchy with respect to the intrinsic flat distance but have no limit.  
This example demonstrates the necessity of the uniform bound
on volume in Theorem~\ref{Cauchy-to-Converge}.  See also Remark~\ref{rmrk-cauchy}.
It is based on the classical example of
Gabriel's Horn:
\be
M_0=\{(x,y,z): \,\,x^2+y^2= 1/(1-z)^2, \, z\ge 0\} \subset E^3
\ee
which is a rotationally symmetric surface of infinite area enclosing a finite volume.
Note that $M_0$ is not an integral current space because it has infinite mass.
The fact that it is unbounded is not a problem as seen in Example~\ref{ex-unbounded}.

\begin{example} \label{ex-Gabriel's-horn}
Define the sequence of Riemannian manifolds diffeomorphic to the sphere 
\be
M_j=\{(x,y,z): \,\,x^2+y^2= f_j(z)/(1-z)^2,\, \} \subset E^3
\ee
such that $f_j(z)$ is $sin(z)$ for $z\in [0,1]$, is $1$ for $z\in [1,j]$ 
and then decreases to $0$ at $z=j+1/j$
so that  each $M_j$ is smooth.
This is a sequence of integral current spaces without a uniform upper
bound on their total mass that is Cauchy with respect to the intrinsic flat 
distance but has no limit in the intrinsic flat sense.
\end{example}

\begin{proof}
First we prove that $M_j$ is a Cauchy sequence by explicitly building a metric
space $Z$ between an arbitrary pair $M_i$ and $M_j$ with fixed $i\ge j$.
Let $T_1$ be the current structure on $M_j$ and $T_2$ the current structure on $M_i$.
Let $U_1=M_j \cap \{z\in [0, j]\}$ and $U_2=M_i \cap \{z\in [0,j]\}$ so
$U_1$ and $U_2$ with the induced length metrics are isometric.  
We now apply Proposition~\ref{prop-bridge-filling} to estimate the
flat distance between them.  In applying this proposition we take
$X_1=V_1=M_j\setminus U_1$ and $B_1=0$ and
$A_1$ to be integration over $X_1$.
Then one can find a constant $C_1$ not depending on $i$ or $j$
such that
\be
\mass(A_1)\le \frac{C_1}{j^2} \qquad \textrm{   and   } \qquad \mass(B_1)=0
\ee
Unlike $V_1$, $V_2$ may be very long and have large area.  So let
\be
X_2=\{(x,y,z,w): \,\,x^2+y^2+w^2= f_i(z)/(1-z)^2,\, z \ge j \, w\ge 0\} \subset E^3
\ee
so that $V_2$ isometrically embeds into $X_2$ and let $B_2$ be
integration over $X_2$ and $A_2$ be integration over 
the disk, $X_2 \cap \{z=j\}$,
with the appropriate orientation.  Then   there exists constants $C_2, C_3$ such that
\be
\mass(A_2)\le {C_2}/{j^2}  \qquad \textrm{   and   }\qquad 
\mass(B_2)=\vol(V_2) \le {C_3}/{j}
\ee
So by Proposition~\ref{prop-bridge-filling}, we have
\be \label{prop-here}
d_{\Fm}\left(M_i, M_j\right) \le \vol\left(U_1\right) \left(h_1+h_2\right) + \mass\left(B_1\right) 
+ \mass\left(B_2\right) +  \mass\left(A_1\right) +\mass\left(A_2\right)
\ee
where 
\be
\begin{split}
h_i  \qquad &\le \qquad \diam(\partial U_i)\, ( 2\diam (U_i) +\diam (\partial U_i))\\
 & \le \qquad {\pi}/{(1-j)^2}\,\,\left(\,2(2j)\,\,\,+\,\,\, {\pi}/{(1-j)^2}\right) 
 \,\,\,\le\,\,\, \frac{C_4}{j}.
\end{split}
\ee
By integrating one sees that  $\vol(U_1)\le C_5 \textrm{Ln}(j)$.
Substituting this into (\ref{prop-here}), we see that the sequence is Cauchy. 

To prove there is no limit for this sequence, we assume on the contrary
that $M_j$ converge in the intrinsic flat sense to an integral current space
$M_\infty$.  We will prove that
there are large balls in $M_\infty$ isometric to large balls in 
\be
N_\infty=\{(x,y,z): \,\,x^2+y^2= f_\infty(z)/(1-z)^2,\, \} \subset E^3
\ee
where $f_\infty(z)$ is $sin(z)$ for $z\in [0,1]$, is $1$ for $z\in [1,\infty)$.
Then apply this to force $\mass(M_\infty)=\infty$ which is a contradiction.

Suppose $M_\infty$ is not the $\bf{0}$ integral current space.  Then there exists
$x\in M_\infty$ and 
there exists $y_j\in M_j$ converging to $x$ and
for almost every $R>0$, there exists $R_j$ increasing to $R$, such that
\be
\liminf_{j\to\infty}\vol(B(y_j,R_j)) \ge \mass(B(x,R)) >0.
\ee
However we need a lower bound $\mass(B(x,R))$.

By our particular choice of $M_j$,
there thus exists $D>0$ such that $y_j\subset M_j \cap \{z\in[0,D]\}$
otherwise the volumes would go to zero.  
For $j$ sufficiently large, there also exist isometries
\be 
\varphi_j: M_j\cap \{z\in[0,D]\} \to N_\infty\cap \{z\in[0,D]\}.
\ee
Since $N_\infty\cap \{z\in[0,D]\}$ is compact,
 a subsequence of the $\varphi_j(y_j)$ converges to some $y_\infty \in N_\infty$.
By the fact that $R_j$ increases to $R$,
$B(\varphi_j(y_j),R_j)$ converges in the Lipschitz sense to 
the open ball $B(y,R)\subset M_\infty$.
Thus by Theorem~\ref{thm-lip-to-flat},
$S(y_j, R_j)=T_j\rstr B(y_j,R_j)$ converge in the intrinsic flat sense
to the integral current space $T_R$ defined by integration over $B(y,R)$
in $N_\infty$.   Note that  $\mass(T_R) \to \infty$.

The Lipschitz convergence also implies that the total mass of $S(y_j,R_j)$
are uniformly bounded above.   We see that $S(y_j,R_j)$
converge in the intrinsic flat sense
flat sense to $S(x,R)= T_\infty\rstr B(y,R) \in\intcurr_2 (M_\infty)$.  Thus there is a current 
preserving isometry from $B(x,R)\subset M_\infty$ to $B(y,R)\subset N_\infty$
for almost every $R>0$.    In particular, we see that
\be
\mass(M_\infty) \ge \lim_{R\to\infty} \mass(B(x,R))=\lim_{R\to\infty}(T_R) =\infty,
\ee
which contradicts the fact that $M_\infty$ is an integral current space.

The only other possibility is that the $M_j$ converge to the $\bf{0}$ current
space.    Then by Theorem~\ref{convergeto0}, we can choose points
$p_j\in M_j$ and find isometric embeddings
$\varphi_j:M_j \to Z$ such that 
$\varphi_j(p_j)=z\in Z$ and $\varphi_{j\#}(T_j) \Fto 0$ in $Z$.

We can choose the $p_j=(0,0,0)\in M_j$ so that all the $B(p_j, R)$
are isometric for $j$ sufficiently large.  Note that
$\varphi_j$ maps $B(p_j,R)$ isometrically onto $B(z,R)\cap \varphi_j(M_j)$.
So for almost every $R>0$ fixed, we have
$
\varphi_{j\#}S(p_j,R)=\varphi_{j\#}T_j \rstr B(z,R) \Fto 0.
$
However this is a constant sequence of nonzero integral current spaces,
so we have a contradiction.
\end{proof}

\bibliographystyle{amsalpha}
\bibliography{SorWen2}

\end{document}